\documentclass[11pt,a4paper]{amsart}
\usepackage[margin=30mm]{geometry}
\usepackage[T1]{fontenc}
\usepackage{lmodern,microtype}
\usepackage{amsmath,amssymb,amsthm,mathtools,mathrsfs}
\numberwithin{equation}{section}
\mathtoolsset{showonlyrefs=true}
\usepackage{booktabs,longtable,array,xurl,hyperref}
\hypersetup{colorlinks=true,linkcolor=blue,citecolor=blue,urlcolor=blue}
\newtheorem{theorem}{Theorem}[section]
\newtheorem{proposition}[theorem]{Proposition}
\newtheorem{lemma}[theorem]{Lemma}
\newtheorem{corollary}[theorem]{Corollary}
\theoremstyle{definition}
\newtheorem{definition}[theorem]{Definition}
\newtheorem{remark}[theorem]{Remark}
\theoremstyle{plain}
\newtheorem{mainthm}{Theorem}

\newtheorem*{scalarconsequence}{Scalar consequence}
\DeclareMathOperator{\Law}{Law}
\newcommand{\E}{\mathbb E}
\newcommand{\Prob}{\mathbb P}
\newcommand{\T}{\mathbb T}
\newcommand{\F}{\mathcal F}
\newcommand{\1}{\mathbf1}
\newcommand{\dd}{\,\mathrm d}
\newcommand{\ii}{\mathrm i}

\newcommand{\boldparagraph}[1]{\par\addvspace{0.5em}
  \paragraph{\textbf{#1}}
}

\hypersetup{
 pdftitle={Fourier asymptotics of critical Gaussian multiplicative chaos on the circle},
 pdfauthor={Yin Cai, Bonan Chen, Xiang Fang, Feng Guo},
 pdfsubject={Critical Gaussian multiplicative chaos on the circle: conditional stable Fourier limits and almost-sure integral classes},
 pdfkeywords={critical Gaussian multiplicative chaos, Fourier coefficients, stable convergence, stable random measures, integral test}}

\begin{document}
\title[Fourier asymptotics of critical chaos]{Fourier asymptotics of critical
Gaussian multiplicative chaos on the circle}

\author[Y. Cai]{Yin Cai}
\address{School of Mathematics\\
Hangzhou Normal University\\
Hangzhou 311121\\
P. R. China}
\email{cymath@hznu.edu.cn}

\author[B. Chen]{Bonan Chen}
\address{School of Mathematical Sciences\\
Soochow University\\
Suzhou 215006\\
P. R. China}
\email{bnchen@suda.edu.cn}

\author[X. Fang]{Xiang Fang}
\address{Department of Applied Mathematics\\
National Yang Ming Chiao Tung University\\
Hsinchu 30010\\
Taiwan}
\email{xfang@nycu.edu.tw}

\author[F. Guo]{Feng Guo}
\address{School of Mathematics\\
South China University of Technology\\
Guangzhou 510640\\
P. R. China}
\email{70207994@nuaa.edu.cn}

\date{}
\begin{abstract}
We study the high-frequency Fourier coefficients $C_n$ of critical Gaussian
multiplicative chaos $M$ on the circle, for logarithmic covariances with an
arbitrary admissible smooth, possibly nonstationary remainder. For every
fixed finite set of integer offsets, the corresponding vector of neighboring
coefficients, multiplied by $\sqrt{\log n}$, converges stably relative to
the full Gaussian field without additional centering. Conditionally on
that field, the limit consists of Fourier coefficients of one symmetric
complex stable random measure of index one with control proportional to
$M$. For every $s>1$, the normalized modulated measures also converge
stably in $H^{-s}(\T)$ to this conditional stable noise. Unconditionally, the same joint approximation remains valid for offsets of size at most $(\log\log n)^{1/64}$ for every fixed admissible covariance. The scalar limit is an isotropic Cauchy mixture; the joint limiting law retains the spatial distribution of $M$, not only its total mass.  
We also prove an almost-sure
integral test along the entire integer frequency sequence. For a class of regular deterministic gauges $H$, the limsup of
$\sqrt{\log|n|}\,|C_n|/H(\log\log|n|)$ is zero or infinity according as
$\int^\infty dx/H(x)$ converges or diverges. The proof combines index-one
compensation with relative excursion estimates and transfers these conclusions from the canonical model by smooth covariance comparison.
\end{abstract}
\maketitle

\section{Introduction}
\label{sixintro:sec-introduction}

\subsection{Critical Fourier fluctuations}
\label{sixintro:subsec-opening}

Critical Gaussian multiplicative chaos (GMC) is a renormalized
exponential measure of a log-correlated Gaussian field at critical
coupling \cite{DRSV,SH}. On the circle, the canonical measure,
normalized by its total mass, also has a spectral construction through
random orthogonal polynomials associated with circular beta ensembles
\cite{CN}. Its Fourier coefficients measure cancellation in the random
mass, not only concentration. They vanish almost surely despite zero
Fourier dimension \cite{Rajchman}; polynomial decay exponents do not
describe the critical asymptotics.

For star-scale invariant fields, Wong \cite{WongPV} proves tightness at
the square-root-logarithmic normalization and almost-sure
logarithmic-power decay. We ask what nondegenerate law
governs coefficients at this normalization, how neighboring frequencies
are coupled, and how large exceptional coefficients can be along a
single realization. For circle fields with admissible smooth,
possibly nonstationary covariance remainders, we prove a conditional
index-one stable Fourier limit and a separate almost-sure integral test.
The first retains the spatial critical measure in the joint limiting law;
the second classifies rare excursions along the full frequency sequence.

Let $g\in C^\infty(\T^2;\mathbb R)$ be deterministic and symmetric, where
$\T=\mathbb R/(2\pi\mathbb Z)$, and suppose that
\begin{equation}\label{intro:eq-covariance}
 K_g(x,y)=\log|e^{ix}-e^{iy}|^{-1}+g(x,y)
\end{equation}
is positive semidefinite as a covariance distribution. We call such a
remainder $g$ admissible. Let $X^g$ be a real centered Gaussian generalized
field of covariance $K_g$. No stationarity or positivity of $g$ itself is
assumed. With unnormalized arc length $dx$, use the Seneta--Heyde convention
\cite{SH}:
\begin{equation}\label{intro:eq-chaos}
 M_g(dx)=\lim_{\varepsilon\downarrow0}
 \sqrt{\log(1/\varepsilon)}
 e^{\sqrt2X^g_\varepsilon(x)-\E[X^g_\varepsilon(x)^2]}\,dx,
 \qquad 0<M_g(\T)<\infty\quad\hbox{a.s.}
\end{equation}
Here $X^g_\varepsilon$ is convolution with a fixed standard nonnegative
smooth approximate identity of integral one. The limit is weak convergence
in probability along the full scale family. Proposition~\ref{uni:prop-construction} gives its construction and a
canonical version measurable with respect to $X^g$. Define
\[
 C_{g,n}=\int_\T e^{inx}M_g(dx),\qquad
 a_n=\sqrt{\log(2+|n|)},\qquad n\in\mathbb Z.
\]
The exact canonical model is $g=0$. For it we reserve
$\phi=X^0$, $M=M_0$, and $C_n=C_{0,n}$. The canonical field has zero
constant mode; no zero-mode restriction is imposed on general $X^g$.
The unindexed notation is used in the canonical analysis below.

\subsection{Main results}
\label{sixintro:subsec-results}

On an extension of the probability space, define
$(\mathcal Z_{g,j})_{j\in\mathbb Z}$ by the conditional characteristic functions
\begin{equation}\label{contract:kernel}
 \widetilde\E\left[\exp\left(i\operatorname{Re}\sum_{j\in J}
 \overline{\xi_j}\mathcal Z_{g,j}\right)\middle|X^g\right]
 =\exp\left[-\frac1{\sqrt\pi}\int_\T
 \left|\sum_{j\in J}\overline{\xi_j}e^{ijx}\right|M_g(dx)\right]
\end{equation}
for every finite $J\subset\mathbb Z$ and $(\xi_j)_{j\in J}\in\mathbb C^J$.
Equivalently $\mathcal Z_{g,j}=\int e^{ijx}\,d\mathcal S_{M_g}(x)$ for one
common symmetric complex stable random measure of index one with control
$M_g/\sqrt\pi$. The construction uses compensated small jumps, zero
Gaussian part, and zero symmetric drift.

\begin{mainthm}[Full-field stable local Fourier limit]\label{contract:A}
For every fixed admissible $g$ and every fixed finite $J\subset\mathbb Z$,
as $n\to\infty$ through positive integers,
\[
 (a_{n+j}C_{g,n+j})_{j\in J}\longrightarrow
 (\mathcal Z_{g,j})_{j\in J}
 \qquad\hbox{stably relative to }\sigma(X^g).
\]
Explicitly, for every bounded $\sigma(X^g)$-measurable random variable $V$
and every bounded continuous $f:\mathbb C^J\to\mathbb R$,
$\E[Vf((a_{n+j}C_{g,n+j})_{j\in J})]
\to\widetilde\E[Vf((\mathcal Z_{g,j})_{j\in J})]$.
\end{mainthm}

\begin{corollary}[Functional critical Fourier limit]\label{contract:functional}
Fix an admissible $g$ and $s>1$. Define the random distribution
\[
 \mathcal T_n^g(f)=a_n\int_\T e^{inx}f(x)M_g(dx).
\]
Viewed as an $H^{-s}(\T;\mathbb C)$-valued random element,
\[
 \mathcal T_n^g\longrightarrow \mathcal S_{M_g}
 \qquad\hbox{stably relative to }\sigma(X^g),
\]
where the limit is the same conditional symmetric complex stable noise of
index one and control $M_g/\sqrt\pi$ as in Theorem~\ref{contract:A}.
Equivalently, the normalized modulated critical measures converge as random
distributions in every $H^{-1-\varepsilon}(\T)$, $\varepsilon>0$.
The same functional limit implies the corresponding weighted limit for every fixed finite family of deterministic $H^s$ test functions and for every fixed finite family of $\sigma(X^g)$-measurable $H^s$-valued random test functions.
\end{corollary}

\begin{scalarconsequence}
The variable $a_n C_{g,n}$ converges stably relative to $\sigma(X^g)$ to
an isotropic planar Cauchy variable of conditional scale
$S_g=M_g(\T)/\sqrt\pi$. Its conditional density and radial tail are
\[
 p_{S_g}(z)=\frac{S_g}{2\pi(S_g^2+|z|^2)^{3/2}},\qquad
 \widetilde\Prob(|\mathcal Z_{g,0}|>r\mid X^g)
 =\frac{S_g}{\sqrt{S_g^2+r^2}}\quad(r\ge0).
\]
\end{scalarconsequence}
The full-field formulation records the joint asymptotics with every fixed
bounded observable of $X^g$. The limit is conditionally stable; its
unconditional law is a mixture. Write $\nu_g=M_g/M_g(\T)$. After division
by $M_g(\T)$, the scalar limit has deterministic Cauchy scale $1/\sqrt\pi$
and can be realized independently of the field. The normalized vector
still retains the random spatial measure $\nu_g$. Already for offsets
$0$ and $1$, its exponent integrates
$|\overline{\xi_0}+\overline{\xi_1}e^{ix}|$ against that measure; it depends
on the location of the mass, not only on its total.

The normalization is suggested by the auxiliary coarse/fine decomposition.
At spatial resolution $e^{-t}$, the critical additive mass is of order
$t^{-1/2}$ in probability by the Seneta--Heyde theorem \cite{SH}, while
the untruncated fine masses have $1/r$ tails \cite{Wong}. At index one,
tail intensities are linear in the coarse weights. Since a frequency
$n$ probes $t$ of order $\log n$, this suggests the factor $\sqrt{\log n}$,
but does not identify the uncentered oscillatory limit. Its phase
averaging and compensation are part of the proof of
Theorem~\ref{contract:A}.

For each realization, squared normalized amplitudes are Fourier
coefficients of the spatial autocorrelation measure, since
\[
 |\widehat{\nu_g}(n)|^2
 =\int_{\T^2}e^{in(x-y)}\nu_g(dx)\nu_g(dy).
\]
Theorem~\ref{contract:A} concerns a fixed observation window;
our second result classifies exceptional amplitudes along the entire
integer frequency sequence, rather than only on a subsequence.

\begin{mainthm}[Full-frequency integral classes]\label{contract:B}
Fix admissible $g$. Let $H:[1,\infty)\to(0,\infty)$ be deterministic and
$C^1$. Suppose that, for some $x_0>1$ and $A_H<\infty$, all $x\ge x_0$
satisfy
\begin{equation}\label{intro:eq-gauge}
 H(x)\ge x,\qquad 0\le\frac{xH'(x)}{H(x)}\le A_H.
\end{equation}
Writing $\vartheta_n=\log\log(e^e+|n|)$, almost surely
\[
 \limsup_{|n|\to\infty}\frac{a_n|C_{g,n}|}{H(\vartheta_n)}
 =\begin{cases}
 0,&\displaystyle\int_{x_0}^\infty\frac{dx}{H(x)}<\infty,\\[4pt]
 \infty,&\displaystyle\int_{x_0}^\infty\frac{dx}{H(x)}=\infty.
 \end{cases}
\]
The same alternatives hold after replacing $C_{g,n}$ by
$C_{g,n}/M_g(\T)$. For each fixed $g$, the probability-one event may
depend on $H$; one event works for any prescribed countable family of
such deterministic gauges.
\end{mainthm}

For $H(x)=x(\log x)^p$, $p\ge0$, modified near $1$ to remain positive,
Theorem~\ref{contract:B} gives the zero alternative for $p>1$ and the
infinite alternative for $0\le p\le1$. In particular, for every fixed
$\varepsilon>0$, almost surely
\[
 |C_{g,n}|=o\!\left(
 \frac{(\log\log|n|)(\log\log\log|n|)^{1+\varepsilon}}
      {\sqrt{\log|n|}}\right)
 \qquad (|n|\to\infty),
\]
whereas
\[
 \limsup_{|n|\to\infty}
 \frac{\sqrt{\log|n|}\,|C_{g,n}|}
      {(\log\log|n|)(\log\log\log|n|)}
 =\infty\quad\hbox{a.s.}
\]
The first bound implies almost-sure Fourier decay. Together these
statements separate the deterministic-frequency fluctuation scale from
the almost-sure envelope.

The double logarithm reflects the scale separation in the recurrence
argument. In the auxiliary model, put $u=\log n$.
Theorem~\ref{l4w:thm-weak-bridge} approximates a terminal coefficient
using the field through time $Ku$, for a sufficiently large fixed $K$. The lower-class proof selects
$n_j=\lceil\exp(d^j)\rceil$, with fixed $d$ large enough that this
approximation is measurable before the next coarse scale. Then
$\log\log n_j\sim j\log d$, and the relative $1/r$ excursion estimates
lead to the series $\sum_j1/H(j\log d)$ and hence to the integral in
Theorem~\ref{contract:B}. The selected coefficients are not independent:
their dependence and approximation errors are treated in
Section~\ref{paper:sec-lower}, while Section~\ref{paper:sec-upper}
controls all intervening frequencies. The integral test therefore
requires more than fixed-threshold convergence.

Both principal theorems hold for each fixed admissible covariance. The
random measure $M_g$ changes with $g$; it is the form of the limiting
kernel and the integral classes that are universal. We do not assert
convergence rates uniform over a class of remainders.

Write $\mathcal Z_j=\mathcal Z_{0,j}$ in the canonical model.
Theorem~\ref{canon:thm-window} first gives an unconditional approximation
on the growing window $|j|\le w_n$, where
$w_n=\lfloor(\log\log n)^{1/64}\rfloor$. Its distance
$d_{\mathrm{BL},w}$ uses tests of absolute value and Lipschitz constant
at most one for the metric
$d_w(z,z')=1\wedge\max_{|j|\le w}|z_j-z'_j|$ on
$\mathbb C^{\{-w,\ldots,w\}}$. The smooth-covariance transfer in
Theorem~\ref{uni:thm-window} shows that the same explicit window holds
for every fixed admissible $g$. This refinement is unconditional and
pointwise in $g$; no uniformity over covariance remainders or optimality
of the exponent $1/64$ is asserted.

\subsection{Relation to earlier work}
\label{sixintro:subsec-comparison}

Garban and Vargas \cite{SixGV} initiated the harmonic analysis of circle
GMC. They proved almost-sure Fourier decay throughout the subcritical
range $\gamma<\sqrt2$ and, for $\gamma<1/\sqrt2$, a Gaussian-mixture
limit whose random variance is determined by the chaos at parameter
$2\gamma$. They also asked for optimal almost-sure Fourier asymptotics
and for the critical Rajchman property. Theorem~\ref{contract:B} addresses
the almost-sure asymptotic question at the critical parameter; it does
not determine the subcritical almost-sure envelopes.

At the level of polynomial Fourier decay, Lin, Qiu and Tan \cite{SixLQT}
determined the Fourier dimension of one-dimensional subcritical GMC.
Orsoni and Verreault \cite[Theorem 1.1]{SixOV} obtained the exact Fourier
dimension on tori for smooth log-correlated covariances in every dimension,
under $0<\gamma<\sqrt{2d}$, together with results exhibiting the role of
boundary geometry. These results concern the strict subcritical range.
At criticality the Fourier dimension is zero, so the questions here are
instead logarithmic: the distribution of a coefficient at high frequency
and the exceptional amplitudes along a realization.

For the canonical critical circle measure, the qualitative Rajchman
property was proved in \cite{Rajchman}. Arguin and Hamdan
\cite[Theorem 1.1]{SixAH} considered critical chaos of a compact-range
star-scale field on an interval and proved
$(\log n)^\alpha c_n\to0$ in probability for every $\alpha<1/4$.
Their Section~1.3 also suggests an order $(\log n)^{-1}$ on heuristic
grounds. 
Theorem~\ref{contract:A} gives a nondegenerate limit at normalization
$\sqrt{\log n}$ for the smooth periodic covariance class; our circle
theorems do not by themselves settle the interval prediction.

During the final preparation of the present manuscript, Wong
\cite{WongPV} posted a related work on power variations of critical GMC
and their spectral applications. For star-scale invariant critical fields
in arbitrary dimension, \cite[Theorem 1.8]{WongPV} proves tightness of
$\sqrt{\log\|\mathbf n\|_\infty}\,\widehat M(\mathbf n)$ and, for every
$\varepsilon>0$, the almost-sure decay
$(\log\|\mathbf n\|_\infty)^{1/2-\varepsilon}\widehat M(\mathbf n)\to0$.
The same paper proves stable convergence for power variations
\cite[Theorem 1.3]{WongPV}. Thus the
square-root-logarithmic normalization appears at the level of
tightness, while quantitative almost-sure decay is known up to a
  small logarithmic-power loss there. The main contribution in
Theorem~\ref{contract:A} is the
identified index-one law of the complex Fourier coefficients. At
index one, the heavy tail alone does not determine the uncentered
oscillatory limit; the compensation needed to identify it is part of the
present argument. Theorem~\ref{contract:B} proceeds along a different direction: it gives a
zero--infinity integral criterion for regular gauges along the entire
integer frequency sequence.

An exact description of circle GMC is available through random orthogonal
polynomials. Chhaibi and Najnudel \cite[Corollaries 2.3 and 2.5]{CN}
describe finite collections of its Fourier coefficients, including at
criticality, in terms of Verblunsky coefficients and the total mass.
This identifies the law of the measure but does not give the high-frequency
limits studied here. Theorem~\ref{contract:A} also records the joint
asymptotics with the original Gaussian field, rather than only a marginal
law of the measure.

For holomorphic multiplicative chaos, Najnudel, Paquette and Simm
\cite{NPS} studied high-frequency coefficients, and Najnudel, Paquette,
Simm and Vu \cite{NPSV} proved joint neighboring-coefficient limits in the
$L^1$ phase. Atherfold and Najnudel \cite{AN} established the corresponding
critical holomorphic result. Those coefficients belong to a holomorphic
random function and have Gaussian-mixture limits. Here the observable is
the Fourier coefficient of the positive critical measure itself, and its
limiting law has stable index one. The two settings therefore lead to related local-limit questions but different observables and limiting
laws.

General covariance decompositions for log-correlated fields are already
available through Junnila, Saksman and Webb \cite{JSW}, including
constructions of critical chaos under regular covariance perturbations.
Our transfer argument addresses the more specific issue of preserving
the full-field Fourier kernel and the full-frequency integral classes.
The canonical construction, positive derivative truncations, rooted law,
terminal locality and canonical coupling are imported from
\cite{Rajchman}; Section~\ref{crm:sec-model} states these results explicitly.
The index-one compensation and relative excursion estimates used for the
new asymptotic conclusions are proved in this paper.

\subsection{Proof strategy}
\label{sixintro:subsec-mechanisms}

The canonical analysis starts from oscillatory cell sums for the auxiliary
finite-range field and its positive barrier-truncated derivative measure.
Sufficiently separated cells have independent futures conditional on the
complete coarse field. The terminal barrier event is used only to
localize estimates, not as an additional conditioning event.

For Theorem~\ref{contract:A}, the compensation step controls the actual
conditional center of the cell sum. Lemma~\ref{crm:lem-capped-first-moment}
bounds the remainder after the logarithmic term in the capped first
moment is removed. Spatial regularity of this remainder and an exact
cutoff identity yield the center estimate of
Theorem~\ref{l3c:thm-uniform-center}.
Theorem~\ref{l3:thm-signed-tail-audit} controls the signed tail error by
Gaussian smoothing of a bounded primitive in the logarithmic mass
variable; convergence of the remainder itself is not assumed.
The local spatial function is allowed to vary with the smoothing variable,
rather than being treated as independent of the Gaussian root.
Together with coarse-measure convergence, phase averaging and marked
Poisson approximation, these estimates give conditional kernel control.
One common marked process handles all offsets, so the construction
identifies their joint law rather than separate scalar limits.
Approximation of bounded field observables and the fixed canonical
coupling then yield Proposition~\ref{canon:prop-profile}.

For the divergent class in Theorem~\ref{contract:B},
Theorem~\ref{pc:thm-excursion} retains unbounded marks and gives a
lower excursion estimate suited to growing thresholds. Recurrence
also requires control of the dependence between terminal coefficients.
The weak-$L^1$ approximation in Theorem~\ref{l4w:thm-weak-bridge}
makes the coefficient of the $1/r$ error arbitrarily small, with a
remaining error summable on the selected scales. Its finite-time
approximations are measurable before the next coarse scale.
Lemma~\ref{p4:lem-count} converts this relative control into
recurrence even though the principal approximation error need not be
summable. Uniform positive probability over integer excursion amplitudes,
followed by a Cameron--Martin zero--one argument, gives the infinite
limsup for the canonical measure.

The convergent class is proved separately in
Proposition~\ref{l4v:cor-upper-integral-class}: one partition per frequency
block controls the discarded positive mass once, while conditional
Bernstein bounds remain summable after a union over every frequency
in the block. Combining the two arguments yields the canonical
classification in Proposition~\ref{canon:prop-integral}.

For general covariances, Lemma~\ref{uni:lem-sandwich} supplies independent
smooth fields so that the augmented field laws agree. Critical-chaos
factorization then reduces transfer to positive smooth multiplication.
The localized fractional-moment bound of
Lemma~\ref{uni:lem-canonical-endpoint}, proved independently of
Proposition~\ref{canon:prop-profile}, controls the Fourier-convolution tails. This is a uniform bound on the
moments of individual modes, not a moment bound for their supremum.
Proposition~\ref{uni:prop-multiplier} then permits multiplication and
inversion even for weights dependent on the field, while
Lemma~\ref{uni:lem-environment} retains the environmental convergence
under changes of field law. At index one, the positive multiplier
changes the stable control in the same way as the measure, giving
$M_g/\sqrt\pi$ without an additional covariance weight.
The integral classes instead transfer pathwise by
Lemma~\ref{l4t:lem-pathwise-multiplier}, without using
Theorem~\ref{contract:A}.

\subsection{Organization}
Section~\ref{crm:sec-model} fixes the auxiliary construction and the
canonical coupling; Section~\ref{paper:sec-common} collects the common
quantitative estimates. Sections~\ref{l3c:sec-uniform-center} and
\ref{paper:sec-profile} establish the compensation estimates and canonical stable Fourier limit. The convergent and divergent integral classes are
proved in Sections~\ref{paper:sec-upper} and \ref{paper:sec-lower},
respectively. Section~\ref{uni:sec-universality} completes
Theorems~\ref{contract:A} and \ref{contract:B} by covariance transfer.
The supporting local estimates are proved in
Appendices~\ref{crm:sec-fine}--\ref{crm:sec-anticlustering}.

\section{Construction, localization and the canonical coupling}
\label{crm:sec-model}
Sections~\ref{crm:sec-model}--\ref{paper:sec-lower} work in the auxiliary
and exact canonical models. Thus $\phi$, $M$, $C_n$, and $\mathcal Z_j$
have the unindexed canonical meaning fixed in the introduction. General
$X^g$ and $M_g$ return in Section~\ref{uni:sec-universality}.

\subsection{Auxiliary field and positive measures}
Fix an even nonconstant kernel $k\in C_c^\infty(\mathbb R)$, supported
in $[-1,1]$, positive definite, with $k(0)=1$ and $xk'(x)\le0$.
In particular $0\le k\le1$. On $\T=\mathbb R/(2\pi\mathbb Z)$ set
\[
 q_v(x)=\sum_{j\in\mathbb Z}k(v(x+2\pi j)),\qquad
 \E X_t(x)X_s(y)=\int_0^{t\wedge s}q_{e^a}(x-y)\dd a.
\]
Let $X=(X_t(x))_{t\ge0,\,x\in\T}$ be the centered Gaussian field
with the displayed covariance, and let $(\F_t)_{t\ge0}$ be its
completed natural filtration.
Use the Fourier convention $\widehat\nu(n)=\int_\T e^{\ii nx}\nu(\dd x)$,
with unnormalized Lebesgue measure $\dd x$. For fixed deterministic $\beta>0$ put
\[
 R_t(x)=\beta+\sqrt2t-X_t(x),\quad
 I_t(x)=\1_{\{\inf_{0\le s\le t}R_s(x)>0\}},\quad
 w_t(x)=e^{\sqrt2X_t(x)-t},
\]
\[
 W_t(\dd x)=w_t(x)\dd x,\qquad
 Z_t^\beta(\dd x)=R_t(x)I_t(x)w_t(x)\dd x.
\]
When the dependence on \(\beta\) needs to be displayed, we write \(R_t^\beta\) and \(I_t^\beta\), otherwise the superscript is suppressed.

\begin{theorem}
\label{crm:thm-construction}
The following assertions hold.
\begin{enumerate}
\item[(i)] \textup{Positive truncations and locality.}
$Z_t^\beta$ converges to a positive finite measure $Z^\beta$,
with the measure-valued martingale closure
\[
 \E[Z^\beta(A)\mid\F_t]=Z_t^\beta(A),\qquad
 \E Z^\beta(A)=\beta|A|
\]
for deterministic Borel sets $A$. Its restriction to $A$ is measurable
with respect to the coarse past and future field over $A$. Thus restrictions
at distance greater than $e^{-t}$ have independent futures given $\F_t$.

\item[(ii)] \textup{Rooted law.}
The finite rooted density
\[
 \frac{\dd\mathbb Q_{t,y}^\beta}{\dd\Prob}
 =\beta^{-1}R_t(y)I_t(y)w_t(y)
\]
makes $(R_s(y))_{s\le t}$ a BES$(3)$ process started at $\beta$.

\item[(iii)] \textup{Untruncated measure and barriers.}
There is a positive finite untruncated derivative measure $D_*$, with
$D_*(\mathbb T)>0$ almost surely,
$Z^\beta\le D_*$ and $Z^\beta=D_*$ on
\[
 E_\beta=\bigcap_{N=1}^\infty
 \{\min_{0\le s\le N,\,x\in\T}R_s(x)>0\}.
\]
Moreover $\Prob(E_\beta)\uparrow1$ as $\beta\to\infty$.

\item[(iv)] \textup{Seneta--Heyde convergence.}
$\sqrt t\,W_t\to\sqrt{2/\pi}D_*$ weakly in probability
(the measure-valued Seneta--Heyde theorem \cite{SH}).
\end{enumerate}

The same construction on a fixed real interval, with covariance
$\int_0^{s\wedge t}k(e^a(x-y))\dd a$, gives the fine measures used below.
\end{theorem}

The construction results in this theorem are imported from
\cite[Sections 2, 4 and 5.2]{Rajchman}, using the same periodized kernel,
unnormalized arc length and derivative normalization. Section 2 there
establishes the positive martingale closure, the rooted BES$(3)$ law and
the global barrier construction; Section 4 establishes measurability of
terminal restrictions and their conditional locality. The
measure-valued Seneta--Heyde identification is proved in Section 5.2
from the real-line result of \cite[Theorem 5]{SH}. These results use the
critical construction of \cite{DRSV}; no Fourier asymptotic is imported.
The quantitative estimates required here are proved below.

The additive critical measure tends to zero at infinite fine time.
Consequently positive derivative truncations with different starting
heights agree at the terminal time whenever neither barrier is crossed,
although their finite-time derivative densities differ.

Write $E_{\beta,t}=\{\min_{0\le s\le t,x\in\T}R_s(x)>0\}$.
Then $E_{\beta,t}\downarrow E_\beta$ as $t\to\infty$, so
$\Prob(E_{\beta,t}\setminus E_\beta)\to0$ without any assumption
about atoms of an infinite-time supremum. This no-hit convention contains
the strict global-supremum event and preserves all terminal comparisons
and barrier exhaustion. The event $E_{\beta,t}$ is $\F_t$-measurable;
$E_\beta$ is not. Independence is always conditional
on $\F_t$, never on $E_\beta$. A family is \emph{tight on $E_\beta$} if
$\lim_{A\to\infty}\sup_n\Prob(E_\beta,|Y_n|>A)=0$; finite initial
indices can equivalently be omitted.

\subsection{All-integer endpoint scales and the mesh}
For an integer $n$ sufficiently large, always use
\begin{equation}\label{crm:eq-scales}
 u=\log n,\qquad L=\sqrt u,\qquad
 b=u+4\log u,\qquad t=u-6\log u,\qquad\rho=e^{-b}.
\end{equation}
Thus $b-t=10\log u$, $ne^{-b}=u^{-4}$ and $e^t/n=u^{-6}$.
Choose
\[
 m=3\left\lceil\frac{2\pi}{3\rho}\right\rceil,
 \quad d=2\pi/m,\quad J_j=[jd,(j+1)d),\quad y_j=jd,\quad z_j=e^{\ii ny_j}.
\]
For large $n$, $\rho/2<d\le\rho$. The three colors $j\bmod3$ work
across the circular seam because $m$ is divisible by three. Distinct
same-color cells are separated by at least $2d>\rho$.
All cell sums run over $0\le j<m$.
Let $\mu=Z^\beta$ and $U_j=L\mu(J_j)$.
For $f\in C^1(\T;\mathbb C)$, use
$\|f\|_{C^1}=\|f\|_\infty+\|f'\|_\infty$,
put $f_j=f(y_j)$, and set
$T_{n,f}=\sum_jz_jf_jU_j$.  Then
\begin{equation}\label{crm:eq-mesherror}
 \left|L\widehat{f\mu}(n)-T_{n,f}\right|
 \le Ld(n\|f\|_\infty+\|f'\|_\infty)\mu(\T).
\end{equation}
Its expectation is $O_{\beta}(\|f\|_{C^1}u^{-7/2})$. We write
$\lambda=\log u$ and $W_s^\beta=I_s^\beta W_s$ throughout.

\subsection{Comparison with the canonical field}
Let $X_*$ denote the limit of $X_t$ as $t\to\infty$ in
$L^2(\Omega;H^{-1}(\T;\mathbb R))$. The covariance comparison of \cite[Section 6]{Rajchman} can be made
on one probability space. There is an independent smooth stationary
Gaussian field $Y$ with zero constant mode such that
\[
 H=X_*+Y=\phi_*+G,\qquad
 G=\lim_{s\to\infty}\frac1{2\pi}\int_\T X_s(x)\,dx,
 \qquad \sigma_G^2=\E G^2,
\]
where $\phi_*$ has the exact-circle law of the introduction and $G$ is
independent of $\phi_*$. Writing
$F_*=e^{\sqrt2Y-\E Y(x)^2}$ and $q_G=e^{\sqrt2G-\sigma_G^2}$, the terminal
measure identity is
\begin{equation}\label{crm:eq-same-space-transfer}
 M_H^{\rm crit}=\sqrt{2/\pi}\,F_*D_*=q_GM_{\phi_*}^{\rm crit}.
\end{equation}
In particular, $M=\sqrt{2/\pi}\,q_G^{-1}F_*D_*$ on this coupling.
The multiplier $F_*$ and its reciprocal are smooth almost surely.
We do not assume that $q_G$ is independent of $D_*$.
If $G_b=(2\pi)^{-1}\int X_b$, then
$\operatorname{Var}(G-G_b)=\sigma_G^2e^{-b}$, by integration of the
constant Fourier mode of the scale covariance.
The comparison is an identity of terminal measures; all subsequent
frequency limits are taken after these measures have been constructed.

\section{Quantitative cell and coarse-measure estimates}\label{paper:sec-common}
All constants in this section may depend on the fixed kernel and
barrier. Quantitative bounds are obtained before barrier exhaustion.
We collect the elementary cell estimates first; their proofs are in
Appendices~\ref{crm:sec-fine}--\ref{crm:sec-endpoint-closure}.

\subsection{Fine tails, rooted increment estimates and independent moment control}
The real-line fine field $Y_s(v)$ has covariance
$\int_0^{s\wedge t}k(e^a(v-v'))\,da$ and starts from zero.
Let $D$ be its derivative chaos on $I=[0,1]$, and $Z^{[q]}$ its
positive version killed at spatial starting barrier $q\ge0$.
Let \(C_D=C_D(k)<\infty\) denote a constant for which the weak-\(L^1\)
tail bound for \(D(I)\) below holds.
For bounded nonnegative $g,q$,
\begin{equation}\label{paper:eq-fine}
 \E Z^{[q]}(g)=\int_I qg,\qquad
 0\le Z^{[q]}(g)\le D(g),\qquad
 \Prob(D(I)>v)\le C_D/v.
\end{equation}
The terminal measures agree if
$\sup_{s\ge0,v\in I}(Y_s(v)-\sqrt2s)<\inf_I q$.
The probability that this supremum exceeds $a$ is at most
$C_ke^{-\sqrt2a}$. Also
$\E|Z^a(g)|^2\le a^2|\int g|^2+C_ke^{\sqrt2a}\|g\|_2^2$.
The tail estimate is the fixed-weight consequence of \cite{Wong};
the crossing and second-moment bounds are proved in Appendix~\ref{crm:sec-fine}.

At a mesh anchor set
\[
 H_j(v)=X_b(y_j+e^{-b}v)-X_b(y_j),\qquad c_j=Le^{-b}w_b(y_j),
\]
\[
 M_j=e^{\sqrt2\sup_{[0,1]}H_j},\qquad \mathcal A_n=\sum_jc_jM_j.
\]
On $E_{\beta,b}$, conditional coarse/fine scaling gives
\begin{equation}\label{paper:eq-cell-scaling}
 U_j=c_jZ^{[R_b(y_j)-H_j]}(g_j),\qquad
 g_j=e^{\sqrt2H_j}\1_{[0,e^bd]}.
\end{equation}
Under the derivative-rooted law, every fixed exponential moment of
$\|H_j\|_\infty$ is bounded uniformly in the anchor and time.
Moreover $\E_{\mathbb Q}R_b^{-2}\le2/b$. The following estimates hold for the full cell array.

\begin{proposition}
\label{crm:prop-budget}
For fixed $\beta>0$,
\begin{equation}
\label{paper:eq-budget}
 \sup_n\E[\1_{E_{\beta,b}}\mathcal A_n]\le C_{\beta,k},\qquad
 v\sum_j\Prob(U_j>v\mid\F_b)\le C_D\mathcal A_n\quad(v>0),
\end{equation}
\[
 \sum_j\E[U_j^2\1_{\{U_j\le\delta\}}\mid\F_b]
 \le2C_D\delta\mathcal A_n\quad(\delta>0).
\]
More generally, for every fixed $C<\infty$,
\begin{equation}\label{crm:eq-general-envelope}
 \sup_n\E\left[\1_{E_{\beta,b}}
        \sum_jc_j e^{C\|H_j\|_\infty}\right]<\infty.
\end{equation}
\end{proposition}

The two conditional estimates above hold without inserting the global
survival event.
The exponential-envelope bound remains valid when the
local increment \(H_j\) is taken on any fixed bounded rescaled interval.
The proof is given in Appendix~\ref{crm:sec-fine}.

Three-color conditional independence implies that a centered sum of
these small marks, with deterministic coefficients bounded by one,
has conditional second moment at most $C_kv\mathcal A_n$.

We next record an endpoint moment estimate obtained from positive barriers and exact conditional means.

\begin{theorem}
\label{crm:thm-endpoint-moment}
For every fixed $\beta>0$, deterministic $f\in C^1(\T)$, and
$0<p<1$, there is $C_{\beta,k,p}<\infty$ such that
\begin{equation}\label{crm:eq-endpoint-pmoment}
 \sup_{\substack{n\in\mathbb N\\n\ge e^2}}
 \E\left[\1_{E_\beta}
 \left|\sqrt{\log n}\int_\T f(x)e^{\ii nx}Z^\beta(\dd x)
 \right|^p\right]
 \le C_{\beta,k,p}\|f\|_{C^1}^{p}.
\end{equation}
\end{theorem}

The proof is given in Appendix~\ref{crm:sec-endpoint-closure}.
As a consequence, the endpoint estimate yields
\begin{equation}
\label{paper:eq-all-mode}
 \sup_{m\in\mathbb Z}\E[\1_{E_\beta}
                 |a_m\widehat D_*(m)|^p]<\infty
 \qquad(\beta>0,\ 0<p<1).
\end{equation}
Indeed, use $f=1$ and $D_*=Z^\beta$ on $E_\beta$ for positive large
modes, with the bounded ratio of $a_m$ to $\sqrt{\log m}$. Negative
modes follow by conjugation. The finitely many remaining modes, including
zero, are bounded by
$\E[\1_{E_\beta}D_*(\T)^p]\le\E Z^\beta(\T)^p\le(2\pi\beta)^p$.
This consequence uses no limiting Fourier law.

The corresponding linear derivative-martingale bracket is
\begin{equation}\label{paper:eq-bracket}
 \E\frac{d}{ds}\langle Z_s^\beta(f),\overline{Z_s^\beta(f)}\rangle
 \le C_{\beta,k}\|f\|_2^2s^{-3/2},\qquad s\ge1.
\end{equation}
Integrating from $t$ to $b$ costs $C(b-t)t^{-3/2}\|f\|_2^2$.
On $E_\beta$, integration by parts at the early time also gives
$\E[\1_{E_\beta}L|\widehat{fZ_t^\beta}(n)|]
\le C_{\beta,k}\|f\|_{C^1}u^{-5}$.

\subsection{Quantitative Seneta--Heyde convergence and phase averaging}
\label{l3:sec-coarse-phase-rate}

This section proves a polynomial rate directly from the fixed-barrier
martingale estimates. No external quantitative Seneta--Heyde theorem is
used. The constants depend on the deterministic barrier; no uniformity
as that barrier tends to infinity is asserted.

Write
\[
 W_b^\beta(\dd x)=I_b^\beta(x)w_b(x)\dd x,
 \qquad Z^\beta=\lim_{s\to\infty}Z_s^\beta,
 \qquad c_{\rm SH}=\sqrt{2/\pi}.
\]
For finite positive measures on a compact metric space $(S,d)$, define
\[
 d_{\rm BL}(\nu,\nu')
 =\sup_{\|F\|_\infty\le1,\,\operatorname{Lip}(F)\le1}
                   |\nu(F)-\nu'(F)|.
\]
For the product of circles we use the sum of the two geodesic distances.

\begin{lemma}\label{l3:lem-SH-bridge-martingale}
Fix \(b>0\) and a deterministic \(f\in L^2(\T)\), and set
\[
 M_s^b(f)=\E[W_b^\beta(f)\mid\F_s],
 \qquad 0\le s\le b.
\]
Then \(M^b(f)\) is square-integrable and, for \(1\le a\le b/2\),
\begin{equation}\label{l3:eq-SH-bridge-variance}
 b\,\E|W_b^\beta(f)-M_a^b(f)|^2
 \le C_{\beta,k}\|f\|_2^2
 \left(a^{-1/2}+b^{-1/2}\log(2+b)\right).
\end{equation}
\end{lemma}

\begin{proof}
Let \(\Phi\) denote the standard normal distribution function and set
\[
 h_\tau(r)=2\Phi(r/\sqrt\tau)-1,
 \qquad r\ge0,\quad \tau>0.
\]
Conditioning at a spatial point and tilting the future Gaussian increment
by its exponential martingale gives, for \(s<b\),
\[
 M_s^b(f)
 =
 \int_\T f(x)I_s^\beta(x)w_s(x)
 h_{b-s}(R_s(x))\,\dd x.
\]
The remaining radial process is standard Brownian motion started at
\(R_s(x)\) and killed at zero, so \(h_{b-s}(R_s(x))\) is precisely its
survival probability over the remaining time \(b-s\).
The finite-time terminal measure has finite second moments, so its conditional expectation is square integrable.

For $\tau>0$,
\[
 h_\tau(0)=0,\qquad
 \partial_rh_\tau(r)=c_{\rm SH}\tau^{-1/2}e^{-r^2/(2\tau)},
 \qquad \partial_\tau h_\tau=\tfrac12\partial_{rr}h_\tau.
\]
It\^o's formula before killing, then stopping at zero, gives the
stochastic coefficient
\[
 I_s^\beta(x)w_s(x)a_{b-s}(R_s(x)),\qquad
 a_\tau(r)=\sqrt2h_\tau(r)-\partial_rh_\tau(r).
\]
The time and spatial heat-equation terms cancel, as do the radial
$\sqrt2$ drift and the cross variation with $w_s$.
There is no boundary term because $h_\tau(0)=0$ for $s<b$.

Let \(p_s^0(x,y)\) denote the transition density of standard Brownian motion killed upon hitting \(0\).
The same double Gaussian tilt used for the derivative-martingale
bracket gives, at a fixed point,
\[
 \E[I_s^\beta w_s^2a_\tau(R_s)^2]
 =e^{s+\sqrt2\beta}\int_0^\infty
       a_\tau(r)^2e^{-\sqrt2r}p_s^0(\beta,r)\dd r.
\]
Use $p_s^0(\beta,r)\le C\beta r s^{-3/2}$ and the spatial kernel
$L^1$ norm $e^{-s}\|k\|_1$. Stationarity, Cauchy--Schwarz, and
Young's inequality reduce the expected bracket to
\[
 C\beta e^{\sqrt2\beta}\|k\|_1\|f\|_2^2s^{-3/2}
       \int_0^\infty r e^{-\sqrt2r}a_\tau(r)^2\dd r.
\]
For $\tau\ge1$, use
$h_\tau(r)\le Cr/\sqrt\tau$ and
$|\partial_rh_\tau(r)|\le C/\sqrt\tau$; the last integral is
$O(\tau^{-1})$. For $0<\tau\le1$, use $h_\tau\le1$ and
\[
 \tau^{-1}\int_0^\infty r e^{-r^2/\tau}\dd r=\frac12.
\]
Thus the integral is bounded by $C/(1+\tau)$, including the
apparently singular terminal layer.
Consequently, for \(s\ge1\),
\begin{equation}
\label{l3:eq-SH-bridge-bracket}
 \E\frac{\dd}{\dd s}\langle M^b(f),\overline{M^b(f)}\rangle_s
 \le C_{\beta,k}\|f\|_2^2
                  \frac{s^{-3/2}}{1+b-s}.
\end{equation}

As $s\uparrow b$, the displayed conditional density converges to
$I_bw_b$ in $L^2$ after spatial integration. This follows either from
conditional-expectation convergence or finite-time Gaussian domination;
the event of hitting the barrier at the exact deterministic time has
probability zero at each spatial point. The integrable bracket bound
also permits the martingale isometry up to $b$.
Split its integral at $b/2$. On $[a,b/2]$,
$b/(1+b-s)\le2$; on $[b/2,b]$, use $s^{-3/2}\le Cb^{-3/2}$.
This gives \eqref{l3:eq-SH-bridge-variance}.
\end{proof}

\begin{theorem}
\label{l3:thm-SH-rate}
For each fixed \(\beta>0\), there is \(C_{\beta,k}<\infty\) such that,
for all sufficiently large \(b\),
\begin{equation}
\label{l3:eq-SH-BL-rate}
 \E\,d_{\rm BL}\!\left(
   \sqrt b\,W_b^\beta,
   c_{\rm SH}Z^\beta
 \right)
 \le C_{\beta,k}b^{-1/5}.
\end{equation}
Moreover, for every deterministic \(f\in L^1(\T)\cap L^2(\T)\),
\begin{equation}
\label{l3:eq-SH-test-rate}
 \E\left|
   \sqrt b\,W_b^\beta(f)-c_{\rm SH}Z^\beta(f)
 \right|
 \le
 C_{\beta,k}b^{-1/5}
 \bigl(\|f\|_{L^1}+\|f\|_{L^2}\bigr).
\end{equation}
\end{theorem}

\begin{proof}
Decompose the signed measure as
\begin{align*}
 \sqrt b W_b^\beta-c_{\rm SH}Z^\beta
 &=\sqrt b(W_b^\beta-M_a^b)
   +(\sqrt b M_a^b-c_{\rm SH}Z_a^\beta)\\
 &\quad+c_{\rm SH}(Z_a^\beta-Z^\beta).
\end{align*}
The first and third measures have uniformly bounded second moments
against $L^2$ tests. The first bound is
\eqref{l3:eq-SH-bridge-variance}. For the third, the derivative bracket
integrated from $a$ to infinity gives
\[
 \E|Z_a^\beta(f)-Z^\beta(f)|^2
 \le C_{\beta,k}\|f\|_2^2a^{-1/2}.
\]
The terminal martingale closure and square integrability are already
part of the fixed-barrier construction.

For the middle measure, Taylor's formula in integral form gives,
for every $r\ge0$ and $a\le b/2$,
\[
 |\sqrt b h_{b-a}(r)-c_{\rm SH}r|
 \le C\left(\frac ab r+\frac1b r^3\right).
\]
This uses $0\le z-\int_0^z e^{-v^2/2}\dd v\le z^3/6$, valid
for all $z\ge0$. Finite rooting and
$\E_{\mathbb Q_{a,x}^\beta}R_a^2=\beta^2+3a$ therefore give
\begin{equation}\label{l3:eq-SH-bias-TV}
 \E\|\sqrt b M_a^b-c_{\rm SH}Z_a^\beta\|_{\rm TV}
 \le C\beta\frac{a+\beta^2}{b}.
\end{equation}
For an arbitrary deterministic test, the corresponding right side
carries its $L^1$ norm.

To pass from individual tests to the BL supremum, use the Fourier
$H^{-1}(\T)$ norm of each of the first and third signed measures.
For a random signed measure $\eta$ with
$\E|\eta(e^{\ii k\cdot})|^2\le A$ for every integer $k$,
\[
 \E\|\eta\|_{H^{-1}}\le
 \left(CA\sum_{k\in\mathbb Z}(1+k^2)^{-1}\right)^{1/2}
 \le C\sqrt A.
\]
Every test with $\|f\|_\infty,\operatorname{Lip}(f)\le1$ has
uniformly bounded $H^1$ norm on the circle. Duality hence bounds its
supremum by $C\|\eta\|_{H^{-1}}$. Combine these estimates with
\eqref{l3:eq-SH-bias-TV} to obtain
\begin{equation}\label{l3:eq-SH-two-time-rate}
 \E\,d_{\rm BL}\left(\sqrt b\,W_b^\beta,
                       c_{\rm SH}Z^\beta\right)
 \le C_{\beta,k}\left(
 a^{-1/4}+b^{-1/4}\sqrt{\log(2+b)}+\frac{a+\beta^2}{b}\right).
\end{equation}
Taking $a=b^{4/5}$ gives
\eqref{l3:eq-SH-BL-rate}; the term
$b^{-1/4}\sqrt{\log(2+b)}$ is bounded by a constant times $b^{-1/5}$.

For a fixed deterministic \(f\in L^1(\T)\cap L^2(\T)\), the same
decomposition, using Cauchy--Schwarz for the first and third terms
and the preceding pointwise rooting estimate for the middle term, gives
\[
 \E\left|
   \sqrt b\,W_b^\beta(f)-c_{\rm SH}Z^\beta(f)
 \right|
 \le C_{\beta,k}
 \left[
 a^{-1/4}+b^{-1/4}\sqrt{\log(2+b)}
 +\frac{a+\beta^2}{b}
 \right]
 \bigl(\|f\|_{L^1}+\|f\|_{L^2}\bigr).
\]
With the same choice \(a=b^{4/5}\), this proves
\eqref{l3:eq-SH-test-rate}.
\end{proof}

\begin{theorem}
\label{l3:thm-marked-phase-rate}
Let $u=\log n$, $L=\sqrt u$, $b=u+4\log u$, and $t=u-6\log u$.
Define the finite measures on $\T\times\mathbb S^1$ by
\[
 \Gamma_n(\dd x,\dd\zeta)
       =LW_b(\dd x)\delta_{e^{\ii nx}}(\dd\zeta),\qquad
 \Gamma^\beta(\dd x,\dd\zeta)
       =c_{\rm SH}Z^\beta(\dd x)m_{\rm H}(\dd\zeta).
\]
Then, for each fixed $\beta>0$,
\begin{equation}\label{l3:eq-marked-BL-rate}
 \E\left[\1_{E_{\beta,b}}
          d_{\rm BL}(\Gamma_n,\Gamma^\beta)\right]
 \le C_{\beta,k}u^{-1/5}.
\end{equation}
Here $m_{\rm H}$ is Haar probability measure, and the target equals
$c_{\rm SH}D_*\otimes m_{\rm H}$ on $E_\beta$.
\end{theorem}

\begin{proof}
On $E_{\beta,b}$, $W_b=W_b^\beta$. The target $Z^\beta$
is defined on the whole probability space; its identification with
$D_*$ is used only on $E_\beta$. Also
\(L/\sqrt b=1+O((\log u)/u)\). Thus
\eqref{l3:eq-SH-BL-rate}, together with this deterministic normalization error, controls the spatial marginal against BL tests. The needed expected masses are bounded because
\[
 \E[LW_b^\beta(\T)]=2\pi Lh_b(\beta)
 \le C\beta L/\sqrt b,\qquad
 \E Z^\beta(\T)=2\pi\beta.
\]

Let $F(x,\theta)$ be a real BL test on the two-dimensional torus, and
subtract its angular average $F_0(x)$. The latter is a spatial BL test
and was just handled. Write $F_\perp=F-F_0$.
Apply the product Fej\'er kernel with cutoff $K$ to $F_\perp$.
Its trigonometric approximation $F_{\perp,K}$ has no angular zero
mode, its $H^1(\T^2)$ norm is bounded by a universal constant, and
\[
 \|F_\perp-F_{\perp,K}\|_\infty
 \le C\frac{\log(2+K)}K.
\]
For completeness, the last estimate follows by integrating the
Lipschitz modulus against the product kernel; the first absolute
moment of the one-dimensional Fej\'er kernel is
$O((\log(2+K))/K)$. Differentiation commutes with its convolution,
so the $H^1$ norm does not increase. The approximation error against
the source has the asserted expected bound by the mass estimate;
both functions have zero angular average against the target.

Write
\[
 F_{\perp,K}(x,\theta)
 =\sum_{|k|,|a|\le K,\ a\ne0}d_{k,a}e^{\ii kx+\ii a\theta}.
\]
Then
\[
 \sum_{k,a}(1+k^2+a^2)|d_{k,a}|^2\le C,\qquad
 \sum_{|k|,|a|\le K}(1+k^2+a^2)^{-1}\le C\log(2+K).
\]
Every evaluated source mode has frequency $q=an+k$, and
$|q|\ge n/2$. On $E_{\beta,b}$, both barriers survive through $b$, so
\[
 W_b-W_t=(Z_b^{\beta+1}-Z_t^{\beta+1})
                       -(Z_b^\beta-Z_t^\beta).
\]
Denote the right side multiplied by $L$ by $\mathcal N_n$.
For every integer frequency $q$, the derivative bracket gives
\[
 \E|\widehat{\mathcal N_n}(q)|^2
 \le C_{\beta,k}L^2(b-t)t^{-3/2}
 \le C_{\beta,k}u^{-1/2}\log u.
\]
Weighted Fourier Cauchy--Schwarz therefore bounds the expectation of
the supremum over these truncated BL tests by
\[
 C_{\beta,k}u^{-1/4}\sqrt{\log u}\,
                              \sqrt{\log(2+K)}.
\]
No independence of the Fourier coefficients is used.

For the remaining early measure, periodic integration by parts gives
simultaneously for every $|q|\ge n/2$,
\[
 L|\widehat W_t(q)|\le\frac{2L}{n}\|w_t'\|_1.
\]
At each point $X_t$ and $\partial_xX_t$ are independent, and
$\E\|w_t'\|_1\le C_ke^t$. The coefficient $\ell^1$ norm is at most
$C\sqrt{\log(2+K)}$ by the same weighted Cauchy--Schwarz estimate.
Thus the expected early contribution is at most
$C_ku^{-11/2}\sqrt{\log(2+K)}$. These are unlocalized estimates,
so they remain valid after the external indicator $\1_{E_{\beta,b}}$
is inserted.
Combining these estimates, for \(2\le K\le n/2\),
\begin{equation*}
\begin{split}
 &\E\left[\1_{E_{\beta,b}}
 d_{\rm BL}(\Gamma_n,\Gamma^\beta)\right]\\
 &\qquad\le
 C_{\beta,k}\left[
 b^{-1/5}
 +\frac{\log u}{u}
 +\frac{\log(2+K)}K
 +\sqrt{\log(2+K)}
   \left(u^{-1/4}\sqrt{\log u}+u^{-11/2}\right)
 \right].
 \end{split}
\end{equation*}
Choosing \(K=\lceil u\rceil\), the phase term is
\(O(u^{-1/4}\log u)\), while all remaining terms are bounded by
\(Cu^{-1/5}\). This proves \eqref{l3:eq-marked-BL-rate}.
\end{proof}

\begin{corollary}
\label{l3:cor-discrete-marked-rate}
For the deterministic circular mesh with cell length $d\le e^{-b}$,
put
\[
 \Gamma_n^{\rm mesh}
 =L\sum_j W_b(J_j)\,\delta_{(y_j,e^{\ii ny_j})}.
\]
Then
\begin{equation}
\E\left[\1_{E_{\beta,b}}
       d_{\rm BL}(\Gamma_n^{\rm mesh},\Gamma^\beta)\right]
 \le C_{\beta,k}u^{-1/5}.
 \label{l3:eq-discrete-marked-rate}
\end{equation}
\end{corollary}

\begin{proof}
Map each source point $(x,e^{\ii nx})$, $x\in J_j$, to
$(y_j,e^{\ii ny_j})$. The product geodesic distance is at most
$(n+1)d$. This deterministic mass-preserving coupling gives
\[
 d_{\rm BL}(\Gamma_n^{\rm mesh},\Gamma_n)
 \le(n+1)d\,LW_b(\T).
\]
On $E_{\beta,b}$ the expected source mass is bounded by
$\E[LW_b^\beta(\T)]\le C\beta$.
Since \(d\le e^{-b}\) and \(ne^{-b}=u^{-4}\),
\[
 \E\left[\1_{E_{\beta,b}}
 d_{\rm BL}(\Gamma_n^{\rm mesh},\Gamma_n)\right]
 \le C_\beta u^{-4}.
\]
The triangle inequality and
Theorem~\ref{l3:thm-marked-phase-rate} imply \eqref{l3:eq-discrete-marked-rate}.
\end{proof}

\boldparagraph{Transport representation of the bounded-Lipschitz metric.}
For finite positive measures \(\nu,\nu'\) on a compact metric space
\((S,d)\), one has
\begin{equation}\label{l3:eq-BL-transport}
 d_{\rm BL}(\nu,\nu')
 =
 \inf_{\pi}
 \left\{
 \int_{S\times S}\min\{d(x,y),2\}\,\pi(\dd x,\dd y)
 +\nu(S)
 +\nu'(S)-2\pi(S\times S)
 \right\},
\end{equation}
where the infimum is over subcouplings \(\pi\) whose marginals are
dominated by \(\nu\) and \(\nu'\).
This is the standard Kantorovich--Rubinstein representation obtained by
adjoining a cemetery point and truncating the transport cost at \(2\).

\subsection{Late global barrier crossings}
\label{l4:sec-barrier-time-tail}

Use the auxiliary finite-range field, the positive derivative barrier
$R_t^\beta=\beta+\sqrt2t-X_t$, and the events
\[
 E_{\beta,b}=\{\min_{0\le s\le b,\,x\in\mathbb T}
                          R_s^\beta(x)>0\},\qquad
 E_\beta=\bigcap_{N\ge1}E_{\beta,N}.
\]
The sharp unit-interval future crossing bound and the rooted exponential-moment estimate imply the following rate.

\begin{lemma}
\label{l4:lem-late-global-crossing}
For every fixed $\beta>0$, there is $C_{\beta,k}<\infty$ such that
\[
 \mathbb P(E_{\beta,b}\setminus E_\beta)
       \le C_{\beta,k}b^{-1/2},\qquad b\ge1.
\]
\end{lemma}

\begin{proof}
Put $\rho=e^{-b}$ and cover the circle by the deterministic mesh of
cell lengths $d\in(\rho/2,\rho]$. At an anchor $y_j$, set
\[
 r_j=R_b^\beta(y_j),\quad
 H_j(v)=X_b(y_j+\rho v)-X_b(y_j),\quad
 M_j=e^{\sqrt2\sup_{0\le v\le1}H_j(v)},\quad
 a_j=\rho w_b(y_j).
\]
On $E_{\beta,b}$, the barrier at every point in that cell is positive.
Its minimum starting height is
$q_j=r_j-\sup_{0\le v\le d/\rho}H_j(v)>0$.
After conditioning on the complete field through $b$, its rescaled
future is the fixed unit-interval fine field. If a future crossing
occurs in this cell, that fine field crosses the constant height
$q_j$. The sharp crossing estimate therefore bounds this probability
by $C_ke^{-\sqrt2q_j}$. The estimate extends from heights at least one
to every positive height by increasing its constant.

Take a union over the cells, without making any independence assertion
between them. Since
$e^{-\sqrt2r_j}=e^{-\sqrt2\beta}a_j$, the conditional probability of
any future crossing is at most
\[
 C_ke^{-\sqrt2\beta}\sum_j a_jM_j
 \qquad\text{on }E_{\beta,b}.
\]
The rooted estimate bounds its expectation. Indeed
$\mathbf1_{E_{\beta,b}}\le I_b^\beta(y_j)$, and the finite-root
identity gives
\[
 \mathbb E[\mathbf1_{E_{\beta,b}}a_jM_j]
 \le\rho\beta\,
   \mathbb E_{\mathbb Q_{b,y_j}^\beta}[M_j/R_b^\beta(y_j)]
 \le C_k\rho\beta b^{-1/2}.
\]
The last inequality is Cauchy--Schwarz, the uniform exponential moment bound for the rooted local increment, and
$\mathbb E_{\mathbb Q}(R_b^\beta)^{-2}\le2/b$.
There are $O(\rho^{-1})$ cells. Summing proves the assertion. A crossing
at a finite future time is precisely the failure relevant to
$E_\beta$; no assumption about an atom of an infinite-time supremum is
used.
\end{proof}

Consequently, along any deterministic sequence of times with
$\sum_j b_j^{-1/2}<\infty$, late crossings after survival through $b_j$
occur only finitely often almost surely. More quantitatively, any
bounded random quantity with absolute value at most one can have its
localizing indicator changed from $E_\beta$ to $E_{\beta,b}$ at an
$L^1$ cost at most $C_{\beta,k}b^{-1/2}$. This is a time-localization
estimate at fixed barrier, not a rate for letting the barrier itself
tend to infinity.

\section{Critical compensation and signed tails}
\label{l3c:sec-uniform-center}

An index-one tail leaves a logarithmic centering term. We identify the actual truncated center by controlling the error after subtracting this logarithmic term from the capped first moment, and then control the
signed tail error by Gaussian smoothing.
The estimates are uniform over the cutoffs used later in the center and excursion estimates.

Use the common mesh and scales, and put $\gamma=\beta+1$.
We initially allow $1<t<b$, $t\asymp b\asymp u$ and $\tau=b-t\ge1$.
Write $\theta=e^bd\in(1/2,1]$. The normalizer is $L=\sqrt u$.
For $0<\delta\le1$ and an integer frequency $k$, the actual hard center
at mass cutoff $\delta$ is
\begin{equation}\label{l3c:eq-hard-center}
 D_{n,\delta}(k)=\sum_{j=0}^{m-1}e^{iky_j}
       \E[U_j\1_{\{U_j\le\delta\}}\mid\F_b].
\end{equation}
For $0\le w\le n/2$, set $N_w=2\lfloor w\rfloor+1$ and
$\mathcal K_{n,w}=\{n+r:r\in\mathbb Z,\ |r|\le w\}$.

The fixed fine derivative chaos $D$ has tail constant
$\kappa=1/\sqrt2$. We write 
\(\log_+T=\max\{\log T,0\}\).  For $g_H=e^{\sqrt2H}\1_{[0,\theta]}$, write
\[
 G_H=\int_0^\theta e^{\sqrt2H(v)}\,dv,\qquad
 J_H=\int_0^\theta H(v)e^{\sqrt2H(v)}\,dv,
\]
\[
 B_g(T)=\E\min(D(g),T)-\kappa\left(\int g\right)\log_+T,
 \qquad \Phi_\delta(H,c)=J_H+B_{g_H}(\delta/c).
\]

\subsection{Bounds and regularity for capped first moments}
\label{l1:sec-response-regularity}
\begin{lemma}[Uniform capped first moment]
\label{crm:lem-capped-first-moment}
Let $D$ be the untruncated derivative chaos of the fixed real-line fine
field on $I=[0,1]$. For a deterministic bounded nonnegative Borel function
$g$ supported in $I$, put
\[
 G=\int_I g(u)\,du,\qquad J=\int_I g(u)^2\,du,\qquad
 \ell_g(T)=\mathbb E[\min\{D(g),T\}],\quad T>0.
\]
There is a constant $C_k<\infty$, independent of $g$ and $T$, such that
\begin{equation}
\label{crm:eq-capped-remainder}
 \left|\ell_g(T)-\frac{G}{\sqrt2}\log_+T\right|
 \le C_k(1+J+G^2),\qquad T>0.
\end{equation}
In particular, this is a bounded
remainder estimate, stronger than a leading-order asymptotic
$\ell_g(T)\sim G\log T/\sqrt2$.
\end{lemma}
\begin{proof}
Use the positive constant-barrier measures $Z^q$, $q>0$, on the same fine
field as $D$. The terminal comparison and the sharp crossing estimate
in Lemma~\ref{crm:lem-crossing} give
\[
 0\le Z^q\le D,\qquad
 Z^q=D\ \hbox{on }\{S_I<q\},\qquad
 \mathbb P(S_I\ge q)\le C_k e^{-\sqrt2q},
\]
where the last bound extends from $q\ge1$ to $q>0$ by increasing its
constant. Lemma~\ref{crm:lem-finel2} and martingale closure give
\[
 \mathbb E Z^q(g)=qG,\qquad
 \mathbb E[Z^q(g)^2]\le q^2G^2+C_k e^{\sqrt2q}J.
\]
Fix $T>1$ and set $q=\log T/\sqrt2$. On the no-crossing event the two
terminal masses agree; elsewhere the cap is at most $T$. Therefore
\[
 \ell_g(T)\le\mathbb E Z^q(g)+T\mathbb P(S_I\ge q)
 \le qG+C_k.
\]
Conversely, positivity and $(x-T)_+\le x^2/T$ imply
\begin{align*}
 \ell_g(T)&\ge\mathbb E\min\{Z^q(g),T\}\\
 &\ge qG-\frac{\mathbb E[Z^q(g)^2]}{T}\\
 &\ge qG-C_kJ-q^2e^{-\sqrt2q}G^2
 \ge qG-C_kJ-\frac{2}{e^2}G^2.
\end{align*}
For $0<T\le1$, use $0\le\ell_g(T)\le T\le1$. This also treats
$T=1$ without introducing a zero starting barrier.
\end{proof}

\begin{remark}[Barrier comparison for the bounded remainder]
For $T>1$, choose $q=\log T/\sqrt2$. The crossing error is bounded
by a constant and the second-moment error by $C_kJ+C G^2$, yielding
\eqref{crm:eq-capped-remainder} without integrating an unquantified
$o(1/T)$ tail error.
\end{remark}

\begin{lemma}
\label{l1:lem-response-lipschitz}
Fix the fine derivative-chaos law $D$ on $[0,1]$. For
$\alpha\in[1/2,1]$, $H\in C([0,1];\mathbb R)$ and $c>0$, put
\[
 g_H=e^{\sqrt2H},\quad G_H=\int_0^\alpha g_H,
 \quad J_H=\int_0^\alpha Hg_H,
\]
\[
 B_H(T)=\mathbb E\min\{D(g_H\mathbf1_{[0,\alpha]}),T\}
                 -\frac{G_H}{\sqrt2}\log_+T,
 \qquad \Phi(H,c)=J_H+B_H(1/c).
\]
There is a constant $C_k<\infty$, independent of $\alpha,c,c',H,H'$,
such that, with $S=\max(\|H\|_\infty,\|H'\|_\infty)$,
\begin{align}
 |\Phi(H,c)|&\le C_k e^{2\sqrt2\|H\|_\infty},
 \label{l1:eq-response-envelope}\\
 |\Phi(H,c)-\Phi(H',c')|
 &\le C_k e^{2\sqrt2 S}
       \bigl(\|H-H'\|_\infty+|\log c-\log c'|\bigr),
 \label{l1:eq-response-lipschitz}\\
 |c\Phi(H,c)-c'\Phi(H',c')|
 &\le C_k(c+c')e^{2\sqrt2 S}
       \bigl(\|H-H'\|_\infty+|\log c-\log c'|\bigr).
 \label{l1:eq-weighted-response-lipschitz}
\end{align}
In the last two bounds, the sum of the two differences may also be
replaced by its minimum with $1$, after increasing $C_k$, by the
individual envelope bounds. These inequalities include $c>1$, $c'>1$,
and crossings of the $\log_+$ threshold.
\end{lemma}
\begin{proof}
Lemma~\ref{crm:lem-capped-first-moment} and
\(|J_H|\le\|H\|_\infty e^{\sqrt2\|H\|_\infty}\) prove
\eqref{l1:eq-response-envelope}. For fixed $T>0$, the functional $g\mapsto B_g(T)$ is concave on
nonnegative weights: $g\mapsto\mathbb E\min\{D(g),T\}$ is concave,
and the subtracted logarithmic term is linear. By \eqref{crm:eq-capped-remainder},
\[
 |B_g(T)|\le C_k(1+\|g\|_\infty^2)
\]
uniformly in $T$ and $\alpha\le1$.

Put $\delta=\|H-H'\|_\infty$ and $M=e^{\sqrt2S}$. If $\delta$ is
bounded below by a fixed positive constant, the two individual bounds
already give $|B_H(T)-B_{H'}(T)|\le C_kM^2\delta$. It remains to
consider $\delta$ so small that
\[
 \varepsilon=\|g_{H'}/g_H-1\|_\infty\le1/4.
\]
If $\varepsilon=0$ there is nothing to prove. Otherwise set
$r=(4\varepsilon)^{-1}$ and consider
\[
 g_a=g_H+a(g_{H'}-g_H),\qquad -r\le a\le1+r.
\]
Every such weight satisfies $g_H/2\le g_a\le3g_H/2$, hence
$|B_{g_a\mathbf1_{[0,\alpha]}}(T)|\le C_kM^2$.
The concave function
$a\mapsto B_{g_a\mathbf1_{[0,\alpha]}}(T)$
has its secant slope on $[0,1]$ between the secant slopes on
$[-r,0]$ and $[1,1+r]$. Therefore
\[
 |B_H(T)-B_{H'}(T)|\le C_kM^2/r
 \le C_kM^2\varepsilon\le C_kM^2\delta.
\]
This proves the spatial Lipschitz estimate uniformly in $T>0$, including
$T<1$. The elementary derivative of $x\mapsto xe^{\sqrt2x}$ similarly
gives $|J_H-J_{H'}|\le C_ke^{2\sqrt2S}\delta$.

For the radial variable, $s\mapsto B_H(e^s)$ is locally absolutely
continuous and has derivative, for almost every $s\ne0$,
\[
 e^s\mathbb P(D(g_H\mathbf1_{[0,\alpha]})>e^s)
 -\frac{G_H}{\sqrt2}\mathbf1_{\{s>0\}}.
\]
For $s>0$, its absolute value is at most $C_ke^{\sqrt2S}$ by
$D(g_H)\le e^{\sqrt2S}D([0,1])$ and the fixed weak-$L^1$ bound.
For $s<0$, it is at most $e^s\le1$. The function is continuous at
zero. Integration therefore proves the required global Lipschitz
bound in $\log c$, including threshold crossings. Combining spatial
and radial bounds proves \eqref{l1:eq-response-lipschitz}.
Finally
$|c-c'|\le(c+c')|\log c-\log c'|$
 together with \eqref{l1:eq-response-envelope} to obtain \eqref{l1:eq-weighted-response-lipschitz}.
\end{proof}

Applying Lemma~\ref{l1:lem-response-lipschitz} with $c/\delta$ in place of $c$ gives,
uniformly over $\delta,c>0$ and $\theta\in[1/2,1]$,
\begin{align}
 |\Phi_\delta(H,c)|
 &\le C e^{C\|H\|_\infty},\notag\\
 |\Phi_\delta(H,c)-\Phi_\delta(H',c')|
 &\le C e^{C\max(\|H\|_\infty,\|H'\|_\infty)}
 \min\{1,\|H-H'\|_\infty+|\log(c/c')|\}.
 \label{l3c:eq-response-input}
\end{align}
There is no dependence on $|\log\delta|$ in these inequalities.

\subsection{Local fluctuations and the late-scale increment}
The following two results separate late random fluctuations from the
early oscillatory mean. In the second statement write
$\rho=e^{-b}$, $1\le t<b$, and
$\Psi_\theta(H,T)=J_H+B_{e^{\sqrt2H}\mathbf1_{[0,\theta]}}(T)$.

\begin{theorem}
\label{lev1:thm-late-decoration}
Let $1<\eta<3/2$, and let $1\le t<b$ be large. Put $\rho=e^{-b}$
and use the circular deterministic mesh with spacing
$\rho/2<\ell\le\rho$. Let $y_j$ be its anchors, and let
\[
 H_j(v)=X_b(y_j+\rho v)-X_b(y_j),\quad 0\le v\le1,
 \qquad c_j=L e^{\sqrt2X_b(y_j)-2b}=Le^{-b}w_b(y_j),
\]
where $L>0$ is deterministic. Suppose $A_j=A_j(c_j,H_j)$ is measurable
with respect to the endpoint field on this cell and satisfies
\begin{equation}\label{lev1:eq-decoration-envelope}
 |A_j|\le C_0e^{\gamma\|H_j\|_\infty}
\end{equation}
for fixed $C_0,\gamma<\infty$. Deterministic dependence on the row,
cell index, or normalized cell length is allowed. Define
\[
 Y_j=I_b^\beta(y_j)c_jA_j,
 \qquad M=\sum_j a_jY_j,
 \qquad U=\E[M\mid\F_t],
\]
where $a_j$ are deterministic and $|a_j|\le A_0$.
Then $M$ is integrable, and
\begin{equation}\label{lev1:eq-late-variance}
 \E|M-U|^2
 \le C_{k,\gamma,\eta}C_0^2A_0^2\,
       \beta e^{\sqrt2\beta}L^2(b-t+1)(t-\log3)^{-\eta}
\end{equation}
for $t$ sufficiently large. Constants do not depend on $b,t,L$, the
number of mesh cells, or the chosen deterministic coefficients.
For the endpoint scales $L=\sqrt{\log n}$,
$b=\log n+4\log\log n$, and $t=\log n-6\log\log n$, the right side of
\eqref{lev1:eq-late-variance} tends to zero for each fixed $\beta$.
\end{theorem}

The proof is given in Appendix~\ref{lev1:sec-late-fluctuation}.

\begin{lemma}
\label{l1:lem-profile-removal}
Use the complete circular mesh with $m=O(e^b)$ anchors $y_j$, cell length
$d\le\rho$, and $\theta=d/\rho$. Put
\[
 c_j=L e^{\sqrt2X_b(y_j)-2b},\qquad T_j=\delta/c_j,
\]
\[
 H_j(v)=X_b(y_j+\rho v)-X_b(y_j),\qquad
 H_j^{\rm new}(v)=
 [X_b-X_t](y_j+\rho v)-[X_b-X_t](y_j).
\]
Here $\delta>0$ is deterministic; the bound below is uniform in this
cutoff. For arbitrary coarse-measurable complex weights $\zeta_j$ with
$|\zeta_j|\le A$, where $A$ is deterministic, define
\[
 \mathcal R_n=
 \sum_j\zeta_j I_b^\gamma(y_j)c_j\Psi_\theta(H_j,T_j),\qquad
 \mathcal R_n^{\rm new}=
 \sum_j\zeta_j I_b^\gamma(y_j)c_j
                     \Psi_\theta(H_j^{\rm new},T_j).
\]
Then
\begin{equation}\label{l1:eq-old-profile-removal}
 \mathbb E|\mathcal R_n-\mathcal R_n^{\rm new}|
 \le C_{k,A}L e^{t-b}.
\end{equation}
For the endpoint scales the right side is $O(u^{-19/2})$.
The bound in \eqref{l1:eq-old-profile-removal} also holds for the $L^1$ norm of the difference after
conditioning on $\mathcal F_t$.
\end{lemma}

The proof is given in Appendix~\ref{l1:sec-old-profile}.

Here are the other quantitative estimates used below, stated in the exact
forms needed. All constants may depend on the fixed barrier, kernel,
and displayed fixed exponents; none depends on $n,w,\delta$ or $k$.
In particular, Theorem~\ref{lev1:thm-late-decoration} gives, for
$1<\eta<3/2$, every endpoint-local functional satisfying
$|A_j|\le C_0e^{a\|H_j\|_\infty}$, and deterministic $|a_j|\le1$,
\begin{equation}\label{l3c:eq-late-input}
 \E\left|\sum_j a_j I_b^\gamma(y_j)c_jA_j
  -\E\left[\sum_j a_j I_b^\gamma(y_j)c_jA_j\mid\F_t\right]\right|^2
 \le C_{\beta,k,a,\eta}C_0^2\,V_\eta.
 \end{equation}
 Here $V_\eta=L^2(\tau+1)(t-\log3)^{-\eta}$.

The fixed-barrier linear bracket \eqref{paper:eq-bracket} gives, for $a=\beta,\beta+1$,
\begin{equation}\label{l3c:eq-linear-input}
 \E|L(\widehat Z_b^a(k)-\widehat Z_t^a(k))|^2
 \le C_{\beta,k}V_{\rm lin},\qquad
 V_{\rm lin}=L^2\tau t^{-3/2},
 \qquad k\in\mathbb Z.
\end{equation}
The kernel subscript on constants denotes the fixed kernel rather than
the Fourier frequency.

For the signed fine-tail error define
\begin{equation}\label{l3c:eq-tail-decoration}
 e_H(T)=T\Prob(D(g_H)>T)-\kappa G_H,\qquad
 Y_j^{\rm tail}(\delta)=I_b^\gamma(y_j)c_j e_{H_j}(\delta/c_j),
\end{equation}
and define the explicit early-tail quantity
\begin{equation}\label{l3c:eq-Q-definition}
 Q_{n,\delta}(t,b)=
 \E\sum_j\left|\E[Y_j^{\rm tail}(\delta)\mid\F_t]\right|.
\end{equation}
The weak-$L^1$ fine tail bound gives
$|e_H(T)|\le Ce^{C\|H\|_\infty}$, so
\eqref{l3c:eq-late-input} already applies to these terms.
Corollary~\ref{l3:cor-common-early-tail-audit} below supplies
a numerical bound for \eqref{l3c:eq-Q-definition}.  Its absolute sum is crucial:
it bounds every deterministic choice of coefficients of modulus at
most one simultaneously.

\subsection{One common bound for all early oscillatory means}
\begin{lemma}
\label{l3c:lem-common-increment}
Suppose $a_0>0$ is fixed, $0<h\le a_0e^{-t}$, and $y\in\T$.
Put $\Delta_h=\sup_{|v|\le h}|X_t(y+v)-X_t(y)|$.
For each fixed $A<\infty$ and $0<p\le1$,
\begin{equation}\label{l3c:eq-common-increment}
 \E[w_t(y)e^{A\Delta_h}\Delta_h^p]
 \le C_{k,A,a_0,p}(e^th)^p.
\end{equation}
\end{lemma}
\begin{proof}
Under the probability tilt with density $w_t(y)$, consider the Gaussian
process $Q(z)=e^{-t}\partial_xX_t(y+e^{-t}z)$ on $[-a_0,a_0]$.
Its variance and the variance of $Q'(z)$ are uniformly bounded, by
the integrals of $e^{2s-2t}$ and $e^{4s-4t}$, respectively.
Its tilted mean and its first derivative are uniformly bounded as well:
the mean is
$\sqrt2e^{-t}\int_0^t e^s k'(e^{s-t}z)\,ds$.
Period images vanish on this fixed local window for sufficiently large
$t$; a bounded remaining range of times changes only the constant.
The fundamental theorem of calculus bounds
$\sup|Q|$ by $|Q(0)|+C\|Q'\|_2$.  Gaussian exponential-square estimates
therefore imply all fixed linear exponential moments of $\sup|Q|$,
uniformly in $t,y$.  Finally
$\Delta_h\le he^t\sup|Q|$.  Since $he^t\le a_0$, this proves the
claim, including the exponential multiplier.
\end{proof}

\begin{lemma}
\label{l3c:lem-common-early}
Suppose $w\le n/2$, $ne^{-b}\to0$, and $e^t/n\le c_0$ for a fixed
constant $c_0$.  Set
\[
 M_\delta(k)=\sum_j e^{iky_j}I_b^\gamma(y_j)c_j\Phi_\delta(H_j,c_j),
 \qquad U_\delta(k)=\E[M_\delta(k)\mid\F_t].
\]
There exists a nonnegative random variable $A_n^{\rm early}$ such that
on $E_{\beta,b}$,
\[
 \max_{k\in\mathcal K_{n,w}}|U_\delta(k)|\le A_n^{\rm early},
\]
and, uniformly over $0<\delta\le1$,
\begin{equation}\label{l3c:eq-common-early}
 \E A_n^{\rm early}\le C_{\beta,k}E_n,
 \qquad E_n=Le^{t-b}
       +L\left\{\frac{e^t}{n}
               +\left(\frac{e^t}{n}\right)^{1/2}\tau^{-1/4}\right\}.
\end{equation}
\end{lemma}
\begin{proof}
Remove the part of $H_j$ coming from scales at most $t$.  The sum of
the absolute cell errors has expectation at most $CLe^{t-b}$ by \eqref{l1:eq-old-profile-removal} and \eqref{l3c:eq-response-input}.
Its conditional expectation is a common envelope for the changes
in every $U_\delta(k)$.

After replacing \(H_j\) by \(H_j^{\rm new}\), exact future exponential tilting gives
\[
 U_\delta^0(k)=Le^{-b}\sum_j e^{iky_j}I_t^\gamma(y_j)V_t(y_j),
 \quad
 V_t(y)=w_t(y)A_{\tau,\delta}(r_t(y),C_t(y)),
\]
where $r_t=\gamma+\sqrt2t-X_t$ and $C_t=Le^{-b}w_t$.
For a fixed anchor $y$, let $\E_{\rm tilt}$ denote expectation
under the future-field law, conditional on $\F_t$, tilted by
\[
 \exp\{\sqrt2(X_b(y)-X_t(y))-\tau\}.
\]
Under this law,
\[
 B_s=X_{t+s}(y)-X_t(y)-\sqrt2s,\qquad 0\le s\le\tau,
\]
is standard Brownian motion. Put
\[
 H(v)=[X_b-X_t](y+\rho v)-[X_b-X_t](y),
 \qquad 0\le v\le1.
\]
Both $B$ and $H$ are obtained from this tilted future field.
Define
\[
 A_{\tau,\delta}(r,C)=\E_{\rm tilt}
 \left[\1_{\{\max_{s\le\tau}B_s<r\}}
 \Phi_\delta(H,Ce^{\tau+\sqrt2B_\tau})\right].
\]
Gaussian exponential moments of the future increment and the density of the Brownian maximum give, for every real $r,a$ and every $C>0$,
\begin{align}
 |A_{\tau,\delta}(r,C)|&\le C_k,\notag\\
 |A_{\tau,\delta}(r-a,Ce^{\sqrt2a})-A_{\tau,\delta}(r,C)|
 &\le C_k\{\min(1,|a|)
       +\min(1,|a|^{1/2}\tau^{-1/4})\}.
 \label{l3c:eq-future-response-modulus}
\end{align}
Indeed the estimate for the change in \(\Phi_\delta\) uses \eqref{l3c:eq-response-input};
the change of the barrier indicator is confined to a maximum strip
of probability at most $C|a|/\sqrt\tau$, and Cauchy--Schwarz applies to the exponentially bounded local-increment factor.  These constants are independent of $\delta$.

On $E_{\beta,b}$ the indicators $I_t^\gamma$ all equal one.  For each
$k\in[n/2,3n/2]$, choose $q_k$ nearest to $\pi/(kd)$ and let
$s_k=q_kd$.  For all sufficiently large $n$,
$|s_k|\le 7/n$ and $|1-e^{iks_k}|\ge1$, uniformly in $k$.
Reindexing the complete mesh gives the exact inequality
\[
 \left|Le^{-b}\sum_j e^{iky_j}V_t(y_j)\right|
 \le Le^{-b}\sum_j|V_t(y_j+s_k)-V_t(y_j)|.
\]
Hence the maximum over all $k\in\mathcal K_{n,w}$ is bounded by the
single random variable
\[
 Le^{-b}\sum_j\sup_{|s|\le7/n}|V_t(y_j+s)-V_t(y_j)|.
\]
Using \eqref{l3c:eq-future-response-modulus} and
\(|e^{\sqrt2a}-1|\le C|a|e^{\sqrt2|a|}\), each supremum is bounded by
\[
 Cw_t(y_j)e^{\sqrt2\Delta_{7/n}}
 \{\Delta_{7/n}+\Delta_{7/n}^{1/2}\tau^{-1/4}\}.
\]
Apply Lemma~\ref{l3c:lem-common-increment} with exponents $1$ and
$1/2$, and use $me^{-b}\le C$.
Adding the replacement error
\eqref{l1:eq-old-profile-removal} proves
\eqref{l3c:eq-common-early}.
\end{proof}

\subsection{The exact cutoff identity and its common errors}

For the cell increment \(H_j\), write
\[
 g_j:=g_{H_j}
     =e^{\sqrt2H_j}\1_{[0,\theta]},
 \qquad
 G_j:=G_{H_j}
     =\int_0^\theta e^{\sqrt2H_j(v)}\,dv.
\]

\begin{proposition}
\label{crm:prop-cap-barrier-removal}
Set $a_j=c_j/L$ and let $\mu=Z^\beta$ be the positive terminal derivative measure. Here $\mathbb E_D$ integrates only the independent fine field, with
the coarse parameters fixed, and $\ell_g(T)=\mathbb E_D\min(D(g),T)$.
On the coarse-survival event $E_{\beta,b}$, define
\[
 C_j=\mathbb E\!\left[
 \min\{L\,\mu(J_j),1\}\mid\mathcal F_b\right],
 \qquad
 F_j=\mathbb E_D\min\{c_jD(g_j),1\}
         =c_j\ell_{g_j}(1/c_j).
\]
Then
\begin{equation}
\label{crm:eq-cap-barrier-removal}
 0\le F_j-C_j
 \le C_k e^{-\sqrt2\beta}a_jM_j
 \quad\hbox{on }E_{\beta,b}.
\end{equation}
Consequently, for any coarse-measurable complex numbers $z_j$ with
$|z_j|\le1$, and with
$\mathcal A_n=L\sum_j a_jM_j$,
\begin{equation}
\label{crm:eq-summed-cap-barrier-removal}
 \mathbf1_{E_{\beta,b}}
 \left|\sum_jz_j(F_j-C_j)\right|
 \le C_k e^{-\sqrt2\beta}
       \mathbf1_{E_{\beta,b}}\frac{\mathcal A_n}{L}.
\end{equation}
For the late mesh in \eqref{crm:eq-scales}, the right side tends to zero in
$L^1$. No first-moment comparison between $D$ and its finite-barrier
version is asserted.
\end{proposition}

The proof is given in Appendix~\ref{crm:sec-response-toolkit}.

Put
\[
 h_\delta=\beta+\kappa\log(L/\delta),\qquad
 W_b(J)=\int_Jw_b(x)\,dx,\qquad
 p_j(\delta)=\Prob(U_j>\delta\mid\F_b).
\]
For the exact cutoff decomposition, set
\begin{align}
 \mathscr C_{n,\delta}(k)
 &:=
 \sum_j e^{iky_j}c_jG_j
       \bigl(h_\delta-R_b^\beta(y_j)\bigr)_+,
 \notag\\
 \mathscr V_{n,\delta}(k)
 &:=
 \sum_j e^{iky_j}
 \left\{
   \E_D[\min\{c_jD(g_j),\delta\}]
   -
   \E[\min\{U_j,\delta\}\mid\F_b]
 \right\}.
 \label{l3c:eq-cap-corrections}
\end{align}
For sufficiently large $n$, $h_\delta>0$.  On $E_{\beta,b}$ the
coarse/fine factorization gives the exact identity
\begin{align}
 \sum_j e^{iky_j}\E[\min(U_j,\delta)\mid\F_b]
 &=P_n(k)-h_\delta L\sum_j e^{iky_j}W_b(J_j)\notag\\
 &\quad+\mathscr R_{n,\delta}(k)+\mathscr C_{n,\delta}(k)
     -\mathscr V_{n,\delta}(k),\label{l3c:eq-cap-identity}\\
 P_n(k)&=\sum_j e^{iky_j}\E[U_j\mid\F_b],\notag\\
 \mathscr R_{n,\delta}(k)&=\sum_j e^{iky_j}c_j\Phi_\delta(H_j,c_j).
 \notag
\end{align}
This follows from
\[
 \E\min(c_jD(g_j),\delta)
 =c_j\{\kappa G_j\log_+(\delta/c_j)+B_{g_j}(\delta/c_j)\}
\]
and $\log(\delta/c_j)=\sqrt2(R_b^\beta(y_j)-h_\delta)$.
The log-positive-part correction has nonnegative cell terms
$c_jG_j(h_\delta-R_b^\beta(y_j))_+$.
The fine-barrier correction is bounded cellwise by
$C_{\beta,k}(\delta/L)c_j e^{C\|H_j\|_\infty}$, since the cap
has range $[0,\delta]$ and the probability of crossing the future
barrier is at most $C_{\beta,k}c_j e^{C\|H_j\|_\infty}/L$.
Thus both corrections possess common absolute envelopes over every
choice of coefficients of modulus at most one.

The first correction \(\mathscr C_{n,\delta}(k)\) is controlled by the following rooted estimate.
\begin{lemma}
\label{l3c:lem-small-radius}
For every fixed $1<\eta<3/2$, fixed $A<\infty$, and $h>0$,
\begin{align}
 \E\left[\1_{E_{\beta,b}}\sum_jc_j e^{A\|H_j\|_\infty}
             (h-R_b^\beta(y_j))_+\right]
 &\le C_{\beta,k,A,\eta}Lb^{-\eta}h^{2\eta},
 \label{l3c:eq-radial-positivepart}\\
 \E\left[\1_{E_{\beta,b}}\sum_jc_j e^{A\|H_j\|_\infty}
             \1_{\{R_b^\beta(y_j)\le h\}}\right]
 &\le C_{\beta,k,A,\eta}Lb^{-\eta}h^{2\eta-1}.
 \label{l3c:eq-radial-indicator}
\end{align}
\end{lemma}
\begin{proof}
Anchor rooting changes $I_b^\beta w_b$ to $\beta/R_b$ times
the BES$(3)$ root law.  Set $p=3/(2\eta)\in(1,3/2)$.
The density bound $p_b^{(3)}(\beta,r)\le Cr^2b^{-3/2}$ gives
\[
 \E_{\mathbb Q}\left(\frac{(h-R_b)_+}{R_b}\right)^p
 \le C_pb^{-3/2}h^3,
 \qquad
 \E_{\mathbb Q}\frac{\1_{\{R_b\le h\}}}{R_b^p}
 \le C_pb^{-3/2}h^{3-p}.
\]
Apply H\"older with this $p$ and the uniform exponential-moment bounds of every fixed order under the rooted law.  Multiplication by $\beta Le^{-b}$ and summation over
$O(e^b)$ anchors proves the assertions.
\end{proof}

Consequently the common cap errors satisfy
\begin{equation}\label{l3c:eq-cap-errors}
 \E[\1_{E_{\beta,b}}\sup_k
       (|\mathscr C_{n,\delta}(k)|+|\mathscr V_{n,\delta}(k)|)]
 \le C_{\beta,k,\eta}
       \{Lb^{-\eta}h_\delta^{2\eta}+\delta/L\}.
\end{equation}
Here and below $\sup_k$ can mean any set of integer frequencies,
because the proof uses absolute sums before taking the supremum.

Finally, subtracting $\delta p_j(\delta)$ converts the cap into the
hard cutoff.  Fine-barrier coupling gives the common estimate
\begin{equation}\label{l3c:eq-tail-barrier}
 \E\left[\1_{E_{\beta,b}}\sum_j
 \left|\delta p_j(\delta)-\kappa c_jG_j
                      -c_je_{H_j}(\delta/c_j)\right|\right]
 \le C_{\beta,k}\delta/L.
\end{equation}
Thus the signed-tail error in \eqref{l3c:eq-tail-decoration},
rather than a pointwise Wong remainder, is the exact additional
quantity needed at an arbitrary cutoff.

\subsection{The uniform quantitative theorem}
Define the common bound for the linear and mesh errors
\begin{equation}\label{l3c:eq-linear-mesh-budget}
 A_n^{\rm lin}=
 \frac{L(1+\beta+\sqrt t)(1+e^t)}{n}+Ln e^{-b},
\end{equation}
and let
\begin{align}
 \mathcal V_{n,w,\delta}
 &=N_w\{V_\eta+(1+h_\delta)^2V_{\rm lin}\},
 \label{l3c:eq-V-definition}\\
 \mathcal A_{n,\delta}
 &=E_n+(1+h_\delta)A_n^{\rm lin}
       +Lb^{-\eta}h_\delta^{2\eta}+\delta/L
       +Q_{n,\delta}(t,b).
 \label{l3c:eq-A-definition}
\end{align}

\begin{theorem}[Uniform estimate for the actual truncated centers]
\label{l3c:thm-uniform-center}
Fix $\beta>0$ and $1<\eta<3/2$.  Under the scale assumptions above,
uniformly over $0\le w\le n/2$, $0<\delta\le1$, and $\varepsilon>0$,
\begin{equation}\label{l3c:eq-max-probability}
 \Prob\left(E_{\beta,b}\cap
 \left\{\max_{k\in\mathcal K_{n,w}}|D_{n,\delta}(k)|>
                   \varepsilon\right\}\right)
 \le C_{\beta,k,\eta}
 \left\{\frac{\mathcal V_{n,w,\delta}}{\varepsilon^2}
               +\frac{\mathcal A_{n,\delta}}{\varepsilon}\right\}.
\end{equation}
In fact the localized first moment itself satisfies
\begin{equation}\label{l3c:eq-truncated-expectation}
 \E\left[\1_{E_{\beta,b}}
       \max_{k\in\mathcal K_{n,w}}|D_{n,\delta}(k)|\right]
 \le C_{\beta,k,\eta}
       \{\sqrt{\mathcal V_{n,w,\delta}}+\mathcal A_{n,\delta}\}.
\end{equation}
These are bounds for the actual conditional center.  They do not
postulate independence of neighboring cells or convergence of $B_g(T)$.
\end{theorem}
\begin{proof}
For the sums involving
\(\Phi_\delta(H_j,c_j)\) and
\(e_{H_j}(\delta/c_j)\), subtract their respective
conditional means at $t$.  Each late fluctuation satisfies
\eqref{l3c:eq-late-input}, uniformly in $k$ and $\delta$.  Hence the
sum of the expected squares over $\mathcal K_{n,w}$ is at most
$CN_wV_\eta$.  Their early means have common absolute envelopes:
Lemma~\ref{l3c:lem-common-early} handles the terms involving \(\Phi_\delta\), and
\eqref{l3c:eq-Q-definition} handles the signed tail.
All pointwise survival indicators equal one on $E_{\beta,b}$, so these
globally defined local decompositions coincide there with the \(\Phi_\delta\) terms and tail terms.

For completeness, the early linear Fourier coefficients have a common
bound over every $|k|\ge n/2$.  On $E_{\beta,b}$ the density of $Z_t^a$,
for $a=\beta,\beta+1$, is $R_t^aw_t$.  Integration by parts gives
\[
 \sup_{|k|\ge n/2}|L\widehat Z_t^a(k)|
 \le\frac{2L}{n}\int_\T
       |(\sqrt2R_t^a(x)-1)w_t(x)\partial_xX_t(x)|\,dx.
\]
The expectation of the right side, without any survival restriction,
is at most $CL(1+\beta+\sqrt t)e^t/n$:
evenness makes $\partial_xX_t(x)$ independent of $X_t(x)$, and
exponential tilting changes the radial variable into $a-N(0,t)$.
The late linear terms satisfy \eqref{l3c:eq-linear-input}; taking the
sum of their squared moduli over the window costs exactly $N_w$.
Freezing a frequency $k\in[n/2,3n/2]$ at the cell anchor has common
absolute error at most $C Lnd$ times the positive total mass, whose
expectation is bounded.  For the additive term use, on $E_{\beta,b}$,
\[
 W_b(dx)=Z_b^{\beta+1}(dx)-Z_b^\beta(dx).
\]
Its early, late, and frozen-phase errors are therefore bounded by
the same estimates.  Its coefficient in the hard-center identity is
$h_\delta+\kappa$, absorbed by $C(1+h_\delta)$.

Insert these decompositions into \eqref{l3c:eq-cap-identity}, subtract
the tail as in \eqref{l3c:eq-tail-barrier}, and apply
\eqref{l3c:eq-cap-errors}.  The resulting maximum is bounded, on
$E_{\beta,b}$, by the sum of a random variable whose second moment is at
most $C\mathcal V_{n,w,\delta}$ and a nonnegative common envelope
whose first moment after multiplication by $\1_{E_{\beta,b}}$ is
at most $C\mathcal A_{n,\delta}$.  Equivalently, multiply the cap
and fine-barrier error envelopes by this indicator before adding them
to the globally integrable early-mean and linear envelopes.
Chebyshev and Markov prove \eqref{l3c:eq-max-probability};
Cauchy--Schwarz proves \eqref{l3c:eq-truncated-expectation}.
The same bound therefore holds after replacing the maximum by its
minimum with one.
\end{proof}

\subsection{Gaussian smoothing and aggregate tail estimates}
Put $\kappa=1/\sqrt2$. For $H\in C([0,1])$ and
$\alpha\in[1/2,1]$, write
\[
 g_H=e^{\sqrt2H},\qquad G_H=\int_0^\alpha g_H(v)\,dv,
 \qquad D_H=\int_0^\alpha g_H(v)D(dv),
\]
\[
 b_H(s)=\mathbb E\min(D_H,e^s)-\kappa G_Hs_+,
 \qquad Q_H(s)=e^s\mathbb P(D_H>e^s)-\kappa G_H.
\]
Thus \(b_H\) provides the bounded primitive used below to control the signed tail error in the logarithmic mass variable.
Gaussian smoothing gives this control while allowing the local increment \(H\) to vary with the Gaussian root.
Lemmas~\ref{crm:lem-capped-first-moment}, \ref{l1:lem-response-lipschitz}
and \ref{crm:lem-finetail} give, uniformly in $s$ and $\alpha$,
\begin{equation}\label{l3:eq-smoothing-inputs}
\begin{aligned}
 |b_H(s)|+|Q_H(s)|&\le C e^{C\|H\|_\infty},\\
 |b_H(s)-b_{H'}(s)|
 &\le C e^{C\max(\|H\|_\infty,\|H'\|_\infty)}
 \|H-H'\|_\infty.
\end{aligned}
\end{equation}
Moreover $b_H$ is locally absolutely continuous, with almost-everywhere
derivative
\begin{equation}\label{l3:eq-smoothing-derivative}
 q_H(s)=e^s\mathbb P(D_H>e^s)-\kappa G_H\mathbf1_{\{s>0\}}.
\end{equation}
All these statements concern the same fine measure $D$ on the same interval.
\begin{lemma}
\label{l3:lem-moving-shape-smoothing}
Let $Z\sim N(0,\tau)$, $\tau\ge\tau_0$, and let $G$ be a random
continuous function independent of $Z$, with uniformly bounded
exponential moments of its supremum norm of every fixed order. Let
$H(z)=H_0+m+az+G$, where $H_0,m,a$ are deterministic continuous
functions,
\[
 \|m\|_\infty\le C_0,\qquad \|a\|_\infty\le C_0/\tau<1/2.
\]
Then, uniformly in $q\in\mathbb R$,
\begin{equation}\label{l3:eq-moving-shape-smoothing}
 \left|\mathbb E q_{H(Z)}(q-\sqrt2Z)\right|
 \le C e^{C\|H_0\|_\infty}\tau^{-1/2}.
\end{equation}
No differentiability of $H\mapsto b_H$ is assumed.
\end{lemma}
\begin{proof}
Mollify $b_H$ only in its real argument $s$, using a smooth approximate
identity of support $[-\varepsilon,\varepsilon]$, and denote the result
by $b_H^\varepsilon$. Fix $G$ and put
$F_\varepsilon(z)=b_{H(z)}^\varepsilon(q-\sqrt2z)$.
This function is locally Lipschitz. Splitting each difference quotient
into a change of \(H(z)\) at fixed radial argument and a change of radial
argument with \(H(z)\) fixed, and using \eqref{l3:eq-smoothing-inputs}, gives
for almost every $z$
\[
 \left|F_\varepsilon'(z)+
       \sqrt2\,\partial_s b_{H(z)}^\varepsilon(q-\sqrt2z)\right|
 \le C e^{C\|H(z)\|_\infty}\|a\|_\infty.
\]
This difference-quotient argument avoids a multivariable chain rule on
a potentially exceptional curve. Gaussian integration by parts gives
\[
 \mathbb E F_\varepsilon'(Z)
       =\mathbb E\left[\frac Z\tau F_\varepsilon(Z)\right].
\]
The bound \eqref{l3:eq-smoothing-inputs} and
\(\|a\|_\infty\le C_0/\tau\) show that the absolute
value of the right side, averaged also over $G$, is at most
$Ce^{C\|H_0\|}\tau^{-1/2}$. The error due to the variation of \(H(z)\) is at most
$Ce^{C\|H_0\|}/\tau$.

It remains to justify taking \(\varepsilon\downarrow0\) for the
varying increment \(H(z)\). Introduce the independent fine measure $D$. For every
fixed realization of $G,D$, the quantity
\[
 \frac{D_{H(z)}}{e^{q-\sqrt2z}}
 =e^{-q}\int_0^\alpha
     e^{\sqrt2(H_0(v)+m(v)+G(v))}
     e^{\sqrt2(1+a(v))z}D(dv)
\]
is strictly increasing in $z$ whenever the measure is nonzero, since
$1+a(v)>1/2$. If the measure is zero, it never equals one. Thus equality
to one occurs for at most one $z$. Fubini implies
$\mathbb P(D_{H(z)}=e^{q-\sqrt2z}\mid G)=0$ for Lebesgue-almost every
$z$, for almost every $G$. The additional kink $q-\sqrt2z=0$ also
occurs at only one point. At all remaining points the radial derivative
in \eqref{l3:eq-smoothing-derivative} is continuous as a function of
its radial argument, with \(H(z)\) fixed. Therefore radial mollification
converges there to $q_{H(z)}(q-\sqrt2z)$. Its exponential envelope is
Gaussian-integrable uniformly in $\varepsilon$. Dominated convergence
proves the claim.
\end{proof}

\begin{theorem}[Quantitative signed aggregate tail matching]
\label{l3:thm-signed-tail-audit}
Use the circular mesh and the endpoint scales
\[
 u=\log n,\quad L=\sqrt u,\quad b=u+4\log u,\quad
 t=u-6\log u,\quad \tau=b-t=10\log u,\quad \rho=e^{-b}.
\]
At its anchors set $c_j=L\rho w_b(y_j)$ and
$H_j(v)=X_b(y_j+\rho v)-X_b(y_j)$, with
$G_j=\int_0^\alpha e^{\sqrt2H_j(v)}dv$ and
$\alpha=d/\rho\in[1/2,1]$. Fix a barrier $\gamma>0$ and any
$1<\eta<3/2$. For arbitrary deterministic complex weights
$|a_j|\le1$, let
\[
 T_v=\sum_j a_j I_b^\gamma(y_j)c_j
                         Q_{H_j}(\log(v/c_j)).
\]
For $0<\delta\le1\le A$, choose
\[
 R=\gamma+\kappa\log(L/\delta)
                  +K\sqrt{\tau\log u},
\]
where $K$ is any sufficiently large fixed constant. Then, uniformly in
$v\in[\delta,A]$,
\begin{equation}\label{l3:eq-signed-tail-rate}
 \mathbb E|T_v|
 \le C_{\gamma,k,\eta,K}\left[
 \tau^{-1/2}+L t^{-3/4}R^{1/2}+u^{-3}
       +L(\tau+1)^{1/2}(t-\log3)^{-\eta/2}\right].
\end{equation}
In particular, for any fixed powers of $\log u$ used for $\delta^{-1}$
and $A$, the right side is $O_{\gamma,k}((\log u)^{-1/2})$.
The bound is uniform over all the deterministic weights; it is not an
expectation of the supremum over such weights.
\end{theorem}

The proof is given in Appendix~\ref{l3:sec-radial-smoothing-audit}.

\begin{corollary}
\label{l3:cor-common-early-tail-audit}
In the notation and parameter range of
Theorem~\ref{l3:thm-signed-tail-audit}, define
\[
 \mathcal Q_{n,v}
 =\mathbb E\sum_j\left|
   \mathbb E\left[I_b^\gamma(y_j)c_j
        Q_{H_j}(\log(v/c_j))\mid\mathcal F_t\right]\right|.
\]
Then, uniformly for every $v\ge\delta$,
\[
 \mathcal Q_{n,v}
 \le C_{\gamma,k,K}
       \left[\tau^{-1/2}+L t^{-3/4}R^{1/2}+u^{-3}\right],
 \qquad
 R=\gamma+\kappa\log(L/\delta)+K\sqrt{\tau\log u}.
\]
In particular, for every fixed logarithmic-power lower cutoff,
$\mathcal Q_{n,v}\le C_{\gamma,k}(\log u)^{-1/2}$.
The statement bounds the sum of the absolute conditional means, so it
simultaneously controls all bounded deterministic coefficient choices.
\end{corollary}

The proof is given in Appendix~\ref{l3:sec-radial-smoothing-audit}.

\begin{corollary}
\label{l3:cor-actual-tail-audit}
Fix $\beta>0$, let $U_j=LZ^\beta(J_j)$, and take
$\gamma=\beta+1$ in
Theorem~\ref{l3:thm-signed-tail-audit}. For arbitrary deterministic $|a_j|\le1$,
\[
 \mathbb E\left[\mathbf1_{E_{\beta,b}}
 \left|v\sum_j a_j\mathbb P(U_j>v\mid\mathcal F_b)
            -\kappa\sum_j a_j c_jG_j\right|\right]
 \le C_{\beta,k}\left[(\log u)^{-1/2}+A/L\right]
\]
uniformly for $v\in[\delta,A]$ when $\delta^{-1}$ and $A$ are fixed
powers of $\log u$.
\end{corollary}

The proof is given in Appendix~\ref{l3:sec-radial-smoothing-audit}.

\subsection{Fully numerical rates and admissible windows}
For the scales in \eqref{crm:eq-scales}, $b=u+4\log u$, $t=u-6\log u$,
$\tau=10\log u$, one has
\begin{align}
 \mathcal V_{n,w,\delta}
 &\le C N_w\{u^{1-\eta}\log u
           +(1+h_\delta)^2u^{-1/2}\log u\},
 \label{l3c:eq-standard-variance}\\
 \mathcal A_{n,\delta}
 &\le C\{u^{-5/2}(\log u)^{-1/4}
       +(1+h_\delta)u^{-7/2}
       +u^{1/2-\eta}h_\delta^{2\eta}+\delta u^{-1/2}
       +Q_{n,\delta}(t,b)\}.
 \label{l3c:eq-standard-absolute}
\end{align}
The harmless exponentially small $n^{-1}$ contributions and the smaller
$u^{-11/2}$ and $u^{-19/2}$ terms are absorbed in this display.

Corollary~\ref{l3:cor-common-early-tail-audit} gives the numerical bound needed in \eqref{l3c:eq-Q-definition}:
\begin{equation}
\label{l3c:eq-signed-tail-rate}
\begin{aligned}
Q_{n,\delta}(t,b)
 &\le C_{\beta,k,K}\{\tau^{-1/2}+Lt^{-3/4}R_\delta^{1/2}+u^{-3}\},\\
 R_\delta
 &=\gamma
 +\kappa\log(L/\delta)
 +K\sqrt{\tau\log u},
\end{aligned}
\end{equation}
where $K$ is a sufficiently large fixed constant.  The first term comes from integrating the bounded remainder against a Gaussian radial density of variance $\tau$; the small-root sector costs the second term.  The proof bounds each early conditional mean in absolute value before summing, so this is a common envelope and has no window-size factor.  
For each fixed $a<\infty$, uniformly over
$(\log u)^{-a}\le\delta\le1$, it gives
\[
 Q_{n,\delta}(t,b)
 \le C_{\beta,k,a}\{(\log u)^{-1/2}
                         +u^{-1/4}(\log u)^{1/2}\}.
\]
Together with \eqref{l3c:eq-max-probability} and
\eqref{l3c:eq-truncated-expectation}, these are fully numerical bounds.

Consequently, for every fixed $a<\infty$
and $0<r<1/2$, choose $1+r<\eta<3/2$.  Uniformly for
$w\le u^r$ and $(\log u)^{-a}\le\delta\le1$,
\begin{align}\label{l3c:eq-polynomial-window-rate}
 &\E\left[\1_{E_{\beta,b}}\min\left\{1,
       \max_{|j|\le w}|D_{n,\delta}(n+j)|\right\}\right]\notag\\
 &\quad\le C\bigl\{
 u^{(r+1-\eta)/2}(\log u)^{1/2}
 +u^{r/2-1/4}(\log u)^{3/2}
 +(\log u)^{-1/2}\notag\\
 &\qquad+u^{1/2-\eta}(\log u)^{2\eta}
 +u^{-1/4}(\log u)^{1/2}\bigr\}
 \longrightarrow0.
\end{align}
In particular the cutoff $\delta=(\log u)^{-1/32}$ requires no
integration over thresholds and no additional logarithmic loss.
For $w\le(\log u)^q$, with any fixed $q$, all the variance and cap
terms in this display decay faster than any fixed negative power of
$\log u$; the displayed radial-smoothing rate is the dominant term.
The corresponding failure bound at tolerance
$\varepsilon=(\log u)^{-s}$ is
$O((\log u)^{s-1/2})$ plus terms negligible at every fixed logarithmic
power, as long as $s<1/2$.

\begin{corollary}
\label{l3c:cor-pilot-center}
Put $\lambda=\log\log n=\log u$ and
$w=\lfloor\lambda^{1/64}\rfloor$.  For every fixed $a<\infty$,
uniformly over $\lambda^{-a}\le\delta\le1$,
\begin{equation}\label{l3c:eq-pilot-center-mean}
 \E\left[\1_{E_{\beta,b}}
       \max_{|j|\le w}|D_{n,\delta}(n+j)|\right]
 \le C_{\beta,k,a}\lambda^{-1/2}.
\end{equation}
In particular,
\begin{equation}\label{l3c:eq-pilot-center-prob}
 \Prob\left(E_{\beta,b}\cap
 \left\{\max_{|j|\le w}|D_{n,\delta}(n+j)|
 >\lambda^{-1/128}\right\}\right)
 \le C_{\beta,k,a}\lambda^{-63/128}.
\end{equation}
\end{corollary}
\begin{proof}
Fix, for example, $\eta=5/4$.  In
\eqref{l3c:eq-truncated-expectation} and its proof, every term except
$Q_{n,\delta}$ is a negative power of $u$ times a fixed power of
$\log u$, and therefore is $o(\lambda^{-1/2})$.
Combining \eqref{l3c:eq-truncated-expectation} with
\eqref{l3c:eq-signed-tail-rate} gives
\eqref{l3c:eq-pilot-center-mean}, and Markov's inequality gives
\eqref{l3c:eq-pilot-center-prob}.
The same argument applies to
$w\le c_0\lambda^{1/64}$ for every fixed $c_0<\infty$, with the
constant also depending on $c_0$.
\end{proof}

\begin{corollary}
\label{l3c:cor-scalar-growing-center}
For every fixed $a>0$, with $\lambda=\log\log n$,
\[
 \sup_{\lambda^{-a}\le\delta\le\lambda^a}
 \E\bigl[\mathbf1_{E_{\beta,b}}
             |D_{n,\delta}(n)|\bigr]
 \le C_{\beta,k,a}\lambda^{-1/2}
\]
for all sufficiently large $n$. The supremum is outside the expectation.
\end{corollary}
\begin{proof}
For $\delta\le1$, use Corollary~\ref{l3c:cor-pilot-center}.
For $1\le\delta\le\lambda^a$, repeat the proof of
Theorem~\ref{l3c:thm-uniform-center} with $w=0$. The estimates \eqref{l3c:eq-response-input} hold for every $\delta>0$, and
$h_\delta=\beta+\log(L/\delta)/\sqrt2$ remains positive and
$O(\lambda)$. The cap and tail comparison errors are at most
$C\delta/L$, uniformly in the displayed range. The other terms in
\eqref{l3c:eq-truncated-expectation}, apart from the common early-tail
term, are negative powers of $u$ times fixed powers of $\log u$.
Corollary~\ref{l3:cor-common-early-tail-audit}, with lower cutoff
$\lambda^{-a}$, bounds that term uniformly by $C\lambda^{-1/2}$.
This proves the assertion for the scalar center.
\end{proof}

\section{Stable Fourier limits}\label{paper:sec-profile}
We use one approximation of the marked cell process both for the stable Fourier limit and for the divergent integral class.
The approximation retains unbounded large marks, which will be essential for rare excursions.
Write \(\lambda=\log u=\log\log n\) and $E=E_{\beta,b}$ in localized expectations, let
$s_l=(y_l,z_l)$ with $z_l=e^{\ii ny_l}$, and put
\[
 p_v=\sum_l\Prob(U_l>v\mid\F_b),\qquad
 D_v(j)=\sum_l z_le^{\ii jy_l}
             \E[U_l\1_{\{U_l\le v\}}\mid\F_b].
\]
Thus $D_v(j)=D_{n,v}(n+j)$ in the center notation. Here
$d_w(x,x')=1\wedge\max_{|j|\le w}|x_j-x_j'|$ and
$\mathcal W_w$ is the infimum expected cost over couplings for this
metric. Constants below may depend on $\beta,k$ and fixed exponents.
\subsection{Tail probabilities and intensity estimates}

\begin{proposition}
\label{crm:prop-local-anticlustering}
For the deterministic circular mesh with
$\rho/2<d\le\rho=e^{-b}$ and
$U_j=L\,Z^\beta(J_j)$, let
$\mathcal E_n$ consist of unordered pairs of distinct cells whose indices
differ by one or two modulo the number of cells. Then for every fixed
$\varepsilon>0$,
\begin{equation}
\label{crm:eq-anticlustering}
 \mathbb E\!\left[\mathbf1_{E_{\beta,b}}
 \sum_{\{j,k\}\in\mathcal E_n}
 \mathbb P(U_j>\varepsilon,
           U_k>\varepsilon\mid\mathcal F_{b})\right]
 \longrightarrow0.
\end{equation}
Moreover, with
$p_j(\varepsilon)=
\mathbb P(U_j>\varepsilon\mid\mathcal F_{b})$,
\[
 \mathbf1_{E_{\beta,b}}\max_jp_j(\varepsilon)
 \longrightarrow0\quad\hbox{in probability},
\]
and the summed edge-product error
$\mathbf1_{E_{\beta,b}}
\sum_{\{j,k\}\in\mathcal E_n}p_j(\varepsilon)p_k(\varepsilon)$
tends to zero in $L^1$.
\end{proposition}

The proof is given in Appendix~\ref{crm:sec-anticlustering}.

For deterministic complex coefficients of modulus at most one, the
conditionally centered sum of $U_l\1_{\{U_l\le v\}}$ has conditional
second moment at most $Cv\mathcal A_n$ on $E$, where
$\E[\1_E\mathcal A_n]\le C$. Also
$\E[\1_E p_v]\le C/v$ for all $v>0$.
The cells admit blocks of $\lceil u^2\rceil$ consecutive retained cells
separated by two omitted cells. Given $\F_b$, retained blocks are
independent. Put $m_B=LZ_b^\beta(B)$, summing over retained blocks.
For $q\in(1/2,3/4)$, we have
\begin{align}
 \E\left[\1_E\sum_{l\text{ omitted}}\Prob(U_l>v\mid\F_b)\right]
       &\le Cv^{-1}u^{-3/2},\label{pc:eq-gaps}\\
 \E\left[\1_E\sum_{\substack{\{l,h\}\text{ retained}\\
                  \mathrm{dist}(l,h)\le2}}
     \Prob(U_l>v,U_h>v\mid\F_b)\right]
       &\le Cv^{-2q}u^{q-3/4},\label{pc:eq-neighbours}\\
 \E[\1_E\sum_Bm_B^2]&\le C\lambda u^{-1/2}.
       \label{pc:eq-blockmass}
\end{align}
Pairs at cell distance at least three are conditionally independent,
and $\sum_{l\in B}\Prob(U_l>v\mid\F_b)\le m_B/v$.
These estimates follow from Lemma~\ref{crm:lem-adjacent-mixed-moment},
Lemma~\ref{crm:lem-higher-power-budget}, and
Lemma~\ref{crm:lem-block-squaremeans} in Appendix~\ref{crm:sec-anticlustering}. For \eqref{pc:eq-neighbours}, apply
Lemma~\ref{crm:lem-higher-power-budget} with exponent $2q$.
The statements concern the original conditional law; conditioning on
the terminal survival event would not preserve the independence.

\boldparagraph{Signed tail matching.}
Put $c_l=Le^{-b}w_b(y_l)$,
$H_l(s)=X_b(y_l+e^{-b}s)-X_b(y_l)$ and
$G_l=\int_0^{e^bd}e^{\sqrt2H_l(s)}ds$.
Then $c_lG_l=LW_b(J_l)$ exactly. For fixed logarithmic-power cutoffs
$0<\delta<1<A$, set
\begin{align}
 e_T&=\sup_{\substack{|a_l|\le1\text{ deterministic}\\\delta\le v\le A}}
 \E\left[\1_E\left|\sum_la_l
 \{v\Prob(U_l>v\mid\F_b)-2^{-1/2}c_lG_l\}\right|\right]
 \nonumber\\[-1mm]
 &\le C(\lambda^{-1/2}+A/\sqrt u).
 \label{pc:eq-signed}
\end{align}
The absolute value is after the signed sum. An expectation of the
supremum over weights, or a sum of absolute cell errors, is not asserted.
The weights may depend on $n$ and a deterministic partition.

Corollary~\ref{l3:cor-actual-tail-audit} gives
\eqref{pc:eq-signed}.

\boldparagraph{Center range.}
Corollary~\ref{l3c:cor-pilot-center} gives, for each fixed $a,c_0<\infty$,
\begin{equation}\label{pc:eq-centers}
 \E[\1_E\max_{|j|\le c_0\lambda^{1/64}}|D_v(j)|]
 \le C\lambda^{-1/2},\qquad \lambda^{-a}\le v\le1.
\end{equation}
For the scalar coordinate $j=0$,
Corollary~\ref{l3c:cor-scalar-growing-center} gives the same bound up to
$v\le\lambda^a$, with the supremum over deterministic cutoffs outside
the expectation.

\boldparagraph{Identification of the intensity.}
Let \(m_{\rm H}\) be Haar probability measure on \(\mathbb S^1\), and set
\[
 \Lambda_n=\frac L{\sqrt2}\sum_lW_b(J_l)\delta_{s_l},
 \qquad
 \Lambda_*=\frac1{\sqrt\pi}Z^\beta\otimes m_{\rm H}.
\]
Let \(\mathfrak T\) be the unbalanced transport cost with matched-mass
cost
\[
 1\wedge\bigl(d_\T(x,x')+|z-z'|\bigr)
\]
and unit cost for unmatched mass.
Since
\(\Lambda_n=2^{-1/2}\Gamma_n^{\rm mesh}\),
\(\Lambda_*=2^{-1/2}\Gamma^\beta\), and
\(|z-z'|\le d_{\mathbb S^1}(z,z')\),
Corollary~\ref{l3:cor-discrete-marked-rate} and
\eqref{l3:eq-BL-transport}, together with homogeneity, give
\begin{equation}
 \E[\1_E\mathfrak T(\Lambda_n,\Lambda_*)]\le Cu^{-1/5}.
 \label{pc:eq-coarse}
\end{equation}
The total-mass consequence of \eqref{pc:eq-signed} and \eqref{pc:eq-coarse}, with
$S=Z^\beta(\T)/\sqrt\pi$, is
\begin{equation}
 \E[\1_E|p_v-S/v|]
 \le C\{(\lambda^{-1/2}+u^{-1/5})/v+u^{-1/2}\}
 \quad(\lambda^{-a}\le v\le\lambda^a).
 \label{pc:eq-scalar-tail}
\end{equation}
For $v<1$ the extra $1/(v\sqrt u)$ is absorbed in the first term.

\subsection{Poisson approximation of the complete marked process}
\label{pc:sec-process}
For any $v>0$ let
\[
 N_v=\sum_l\1_{\{U_l>v\}}\delta_{(s_l,U_l)},\qquad
 \nu_v=\sum_l\delta_{s_l}\otimes
            \Law(U_l\mid\F_b)|_{(v,\infty)}.
\]
Set
\begin{equation}
 B_n(v)=v^{-2q}u^{q-3/4}+v^{-2}\lambda u^{-1/2}
                      +v^{-1}u^{-3/2}.
 \label{pc:eq-B}
\end{equation}

\begin{lemma}
\label{pc:lem-PPP}
Under \eqref{paper:eq-budget} and \eqref{pc:eq-gaps}--\eqref{pc:eq-blockmass}, conditionally on $\F_b$, $N_v$ can be coupled to a Poisson
point process of intensity $\nu_v$ with failure probability $\Pi_v$
satisfying $\E[\1_E\Pi_v]\le CB_n(v)$. The same bound holds after
restriction to any deterministic mark set, including $v<r\le A$.
\end{lemma}
\begin{proof}
Fix a coarse realization. Remove the gap points. Their total failure
probability is bounded by
\(\sum_{l\text{ omitted}}\Prob(U_l>v\mid\F_b)\).
In every retained block erase all points when its count $N_B$ is at
least two. The failure probability is at most
$\E[N_B(N_B-1)\mid\F_b]/2$.
The conditional block processes are independent.

The remaining process in a block is either empty or a single point.
Its intensity $\eta_B$ has mass at most
$\sum_{l\in B}\Prob(U_l>v\mid\F_b)\le m_B/v$.
Couple its Bernoulli count with a Poisson count of the same mean,
using the same normalized mark law when both counts equal one. The
failure probability is at most $C\eta_B((\T\times\mathbb S^1)\times(v,\infty))^2$.
The difference between the original and erased block intensities is
a positive measure of mass at most
$\E[N_B\1_{\{N_B\ge2\}}\mid\F_b]
\le\E[N_B(N_B-1)\mid\F_b]$.
Restore this intensity, and then the omitted gap intensity, by adding
independent Poisson processes. The probability of adding any point is
at most the added mass. This produces exactly intensity $\nu_v$.

For pairs at distance one or two, use \eqref{pc:eq-neighbours}.
For all other within-block pairs, conditional independence bounds their
sum by $m_B^2/(2v^2)$. Thus all Bernoulli and separated-pair costs
are controlled by $Cv^{-2}\sum_Bm_B^2$.
Now average with $\1_E$ and use \eqref{pc:eq-gaps}--\eqref{pc:eq-blockmass}.
Restriction of coupled processes cannot increase their disagreement.
\end{proof}

\subsection{Conditional approximation of Fourier coefficients}
\label{pc:sec-window}
Let $Q_{\beta,w}$ be the probability kernel of the symmetric stable
Poisson integral with intensity $\Lambda_*\otimes r^{-2}dr$ and jump map
\[
 \Phi_w(x,z,r)=(rz e^{\ii jx})_{|j|\le w}.
\]
Compensate radii at most one and take the $L^2$ limit at zero; the
small-radius variance is finite because $\int_0^1r^2r^{-2}dr=1$.
There are finitely many radii above one. Haar phase makes the
compensator of every fixed radial annulus zero. Define
\[
 K_{n,\beta,w}=\Law((a_{n+j}\widehat Z^\beta(n+j))_{|j|\le w}\mid\F_b),
 \qquad \widehat\mu(m)=\int e^{\ii mx}\mu(dx).
\]
The terminal kernel $Q_{\beta,w}$ need not be $\F_b$-measurable:
the distance below is evaluated pointwise, then averaged.

\begin{theorem}
\label{pc:thm-window}
For $w\le n/2$, $0<\delta<1<A$ and $0<h<\delta/4$, let $e_T$ be
\eqref{pc:eq-signed},
$e_C=\E[\1_E(1\wedge\max_{|j|\le w}|D_1(j)|)]$, and
$e_P=\E[\1_E\mathfrak T(\Lambda_n,\Lambda_*)]$.
Then,
\begin{align}
 &\E[\1_E\mathcal W_w(K_{n,\beta,w},Q_{\beta,w})]
 \nonumber\\
 &\quad\le C\bigg\{u^{-7/2}+\frac{w+1}{n\sqrt u}
       +\sqrt{(2w+1)\delta}+A^{-1}+B_n(\delta)
       +\frac{h+(w+1)Ah}{\delta}\bigg\}
 \nonumber\\
 &\qquad+Ce_C+CAh^{-3}\delta^{-1}e_T
       +C[\delta^{-1}+(w+1)\log(A/\delta)]e_P.
 \label{pc:eq-window}
\end{align}
\end{theorem}
\begin{proof}
Write $v_l=(z_le^{\ii jy_l})_{|j|\le w}$. The exact decomposition is
\begin{align*}
 \sum_lv_lU_l={}&D_1+
 \sum_lv_lU_l\1_{\{\delta<U_l\le A\}}
 -\sum_lv_l\E[U_l\1_{\{\delta<U_l\le1\}}\mid\F_b]\\
 &+\sum_lv_l\{U_l\1_{\{U_l\le\delta\}}
       -\E[U_l\1_{\{U_l\le\delta\}}\mid\F_b]\}
 +\sum_lv_lU_l\1_{\{U_l>A\}}.
\end{align*}
The small centered sum costs $C\sqrt{(2w+1)\delta}$ by \eqref{paper:eq-budget}, since
the squared maximum is at most the sum of the squared coordinates.
The target compensated small jumps have the same bound by the Poisson
variance identity and $\E Z^\beta(\T)=2\pi\beta$.
The probability of any source or target jump above $A$ is at most
$C/A$ after averaging. The center costs $e_C$.

Lemma~\ref{pc:lem-PPP} replaces the actual middle process by a Poisson
process of intensity $\nu^{\delta,A}=\nu_\delta|_{r\le A}$ at cost
$CB_n(\delta)$. Its compensator is exactly the deterministic middle
compensator in the display. It is not discarded or replaced by a
center at the shrinking cutoff.

Partition both circles into arcs of length at most $h$ and the radial
interval $[\delta,A]$ into intervals of length at most $h$, with $1$
an endpoint. There are $O(Ah^{-3})$ product cells. On a base cell the
membership weights at $s_l$ are deterministic, so \eqref{pc:eq-signed} applies.
The intensity error on a radial interval $(r_1,r_2]$ is the difference
of two tail errors, and has localized expected absolute value at most
$2e_T/\delta$. The summed grid-intensity error is therefore
$CAh^{-3}\delta^{-1}e_T$.

Move marks to representatives of their product cells, preserving
whether the radius is at most one. The displacement of either the
jump map or its compensated version is at most $C[h+(w+1)Ah]$.
The expected total middle intensity is $O(\delta^{-1})$.
Couple grid counts using the minimum of their two intensities, adding
independent unmatched counts. Their probability cost is at most the
intensity discrepancy; the deterministic compensator costs no more,
since its radii are at most one. This proves the grid and rounding terms.

Finally partially transport $\Lambda_n$ to $\Lambda_*$, sharing each
matched radius. For base distance $a\le1$, the jump displacement is
at most $r(w+1)a$, whose radial integral is
$(w+1)\log(A/\delta)a$. Unmatched mass costs at most $1/\delta$ per
unit mass. The compensated part obeys the same bounds. Finite-grid
approximations give measurable near-optimal couplings, hence the
last term of \eqref{pc:eq-window}.

The cell phase error has expected maximum at most
$L(n+w)d\E Z^\beta(\T)=O(u^{-7/2})$.
Also $\max_{|j|\le w}|a_{n+j}-L|\le C(w+1)/(n\sqrt u)$.
These yield the first two terms by the finite first moment of the
barrier measure. All couplings compare kernels at the same coarse
realization; no independence from the terminal target is asserted.
\end{proof}

\begin{corollary}
\label{pc:cor-window}
For $w\le c_0\lambda^{1/64}$ the left side of \eqref{pc:eq-window}
is at most $C_{\beta,k,c_0}\lambda^{-1/128}$.
For $w_n=\lfloor(\log\log n)^{1/64}\rfloor$, barrier exhaustion gives
unconditional approximation of the auxiliary $D_*$ window by the
stable window with control $D_*/\sqrt\pi$.
\end{corollary}
\begin{proof}
Choose $\delta=\lambda^{-1/32}$, $A=\lambda^{1/32}$ and
$h=\lambda^{-1/8}$. The four potentially dominant powers in
\eqref{pc:eq-window} are respectively
\begin{align*}
 Ah^{-3}\delta^{-1}e_T&=O(\lambda^{-1/16}),&
 (w+1)Ah/\delta&=O(\lambda^{-3/64}),\\
 \sqrt{(2w+1)\delta}&=O(\lambda^{-1/128}),&
 A^{-1}&=\lambda^{-1/32}.
\end{align*}
The terms containing a negative power of $u$ are negligible compared
with every fixed negative power of $\lambda$. Use \eqref{pc:eq-centers} for $e_C$ and \eqref{pc:eq-coarse} for $e_P$.
On $E_\beta$ both the source and target agree with their untruncated
versions. Since the metric is bounded by one, the untruncated error
has limsup at most $C\Prob(E_\beta^c)$; take $\beta\to\infty$ last.
This supplies no global numerical rate after exhaustion.
\end{proof}

\subsection{The canonical stable limit}
\label{canon:subsec-profile}

\begin{proposition}
\label{canon:prop-profile}
For the exact circle field $\phi$, every fixed finite $J\subset\mathbb Z$
satisfies
\[
 (a_{n+j}C_{n+j})_{j\in J}\longrightarrow(\mathcal Z_j)_{j\in J}
 \quad\hbox{stably relative to }\sigma(\phi),\qquad n\to+\infty,
\]
where the joint conditional exponent is
$-\pi^{-1/2}\int_\T|\sum_{j\in J}\overline{\xi_j}e^{\ii jx}|M(dx)$.
\end{proposition}

\begin{proof}
For any fixed finite set of offsets, Corollary~\ref{pc:cor-window}
gives convergence of the conditional kernels on $E_{\beta,b}$.
For a bounded $\F_s$-measurable test variable, take $b\ge s$ and
condition on $\F_b$. The kernel error tends to zero in localized
$L^1$. Replacing $E_{\beta,b}$ by $E_\beta$ costs
$\Prob(E_{\beta,b}\setminus E_\beta)\to0$.
For any bounded $\F_\infty$-measurable test variable $V$, approximate
it in $L^1$ by $\E[V\mid\F_s]$ and then let $s\to\infty$.
Finally exhaust integer barriers, using $Z^\beta=D_*$ on $E_\beta$.
This proves stable convergence relative to the entire auxiliary field.

The conditional exponent of its limiting vector is
\begin{equation}\label{paper:eq-aux-profile}
 -\frac1{\sqrt\pi}\int_\T
       \left|\sum_{j\in J}\overline{\xi_j}e^{\ii jx}\right|D_*(dx).
\end{equation}
Indeed, writing $a=\operatorname{Re}(z\sum_j\overline{\xi_j}e^{\ii jx})$,
\[
 \int_0^\infty(1-\cos(ar))r^{-2}\,dr=\pi|a|/2,
 \qquad \int_{\mathbb S^1}|\operatorname{Re}z|\,m_{\rm H}(dz)=2/\pi.
\]
These identities fix the constant. The Gaussian part and the symmetric
radial drift are zero. The compensated construction yields one common
stable random measure for all finite sets $J$.

Adjoin the independent smooth field $Y$ from
\eqref{crm:eq-same-space-transfer}. Stability extends to
$\F_\infty\vee\sigma(Y)$: first use product tests, then bounded
$L^1$ approximation. Define
\[
 F_{*,k}=\frac1{2\pi}\int_\T F_*(x)e^{\ii kx}\,dx,
 \qquad F_*(x)=\sum_{k\in\mathbb Z}F_{*,k}e^{-\ii kx}.
\]
For fixed $K$, Fourier multiplication by
$F_*^{(K)}=\sum_{|k|\le K}F_{*,k}e^{-\ii kx}$ uses a finite enlarged
offset set.
Its random coefficients and \(q_G^{-1}\) are measurable with respect to
\(\F_\infty\vee\sigma(Y)\).
Joint stable mapping follows by approximating
bounded coefficients by simple random variables and exhausting their
bounds. The normalizing ratios $a_{n+j}/a_{n+j-k}$ tend to one for
each fixed $j,k$.

Here is the uniform estimate that removes $K$. Fix $0<p<1$ and
$s>1/p$, and localize to
\[
 \mathcal L_{\beta,R}=E_\beta\cap
 \{q_G^{-1}\le R,\ B_s(F_*)\le R\},\qquad
 B_s(F_*)=\sup_k(1+|k|)^s|F_{*,k}|.
\]
Since
$a_{n+j}/a_{n+j-k}\le(1+\log(2+|k|)/\log2)^{1/2}$,
$p$-subadditivity and \eqref{paper:eq-all-mode} bound the discarded
terms for every fixed offset by
\begin{equation}\label{paper:eq-multiplier-tail}
 C_{\beta,p}R^{2p}\sum_{|k|>K}(1+|k|)^{-sp}
       \left(1+\frac{\log(2+|k|)}{\log2}\right)^{p/2}.
\end{equation}
This tends to zero uniformly in $n$. Every localization indicator
except $\1_{E_\beta}$ is discarded only after taking absolute
upper bounds; no dependent expectation is factorized.
For the limiting kernels, uniform convergence of $F_*^{(K)}$ and the
finiteness of $D_*$ give the same passage to $F_*$.

The identity \eqref{crm:eq-same-space-transfer} now changes
\eqref{paper:eq-aux-profile} into
\[
 -\frac1{\sqrt\pi}\int_\T
   \left|\sum_{j\in J}\overline{\xi_j}e^{\ii jx}\right|M(dx).
\]
The events $\mathcal L_{\beta,R}$ exhaust probability one, with $K$
removed first for each fixed localization.
Projecting the stable convergence to \(\sigma(\phi_*)\) proves
Proposition~\ref{canon:prop-profile}.
\end{proof}

For the canonical scalar, put $S=M(\T)/\sqrt\pi$. Fourier inversion of
$e^{-S|\xi|}$ yields the density in the introduction at $g=0$; integration in polar coordinates gives its
radial tail. Since $C_{-n}=\overline{C_n}$ and this scalar law is
conjugation invariant, the scalar conclusion holds as $|n|\to\infty$.
Division by the finite positive random mass \(M(\T)\) gives the normalized scalar limit. More generally the normalized vector has control
$M/(\sqrt\pi M(\T))$, so its spatial dependence persists.

The scalar limit also fixes the probability scale: if $d_n/a_n\to0$,
then $d_nC_n\to0$ in probability, by tightness. If $d_n/a_n\to\infty$,
then $|d_nC_n|\to\infty$ in probability, because the limiting Cauchy
law has no atom at zero. None of these distributional statements
substitutes for the almost-sure argument below.

\subsection{The growing canonical window}\label{paper:sec-window-transfer}
\begin{theorem}
\label{canon:thm-window}
For the exact canonical circle measure, let
$w_n=\lfloor(\log\log n)^{1/64}\rfloor$. Then
\begin{equation}\label{canon:eq-window}
 d_{\mathrm{BL},w_n}\left(
 \Law((a_{n+j}C_{n+j})_{|j|\le w_n}),
 \Law((\mathcal Z_j)_{|j|\le w_n})\right)\longrightarrow0.
\end{equation}
The approximation is unconditional. No optimality of the exponent is
asserted. The corresponding fixed-covariance extension is proved in
Theorem~\ref{uni:thm-window}.
\end{theorem}
\begin{proof}
Write $N_w=2w+1$ for the number of coordinates. The metric remains $d_w$ as defined in the introduction.
The same-space comparison \eqref{crm:eq-same-space-transfer} gives
$M=\sqrt{2/\pi}\,q_G^{-1}F_*D_*$, with smooth $F_*$ independent of the
auxiliary field, $q_G=e^{\sqrt2G-\sigma_G^2}$, and
$\operatorname{Var}(G-G_b)=\sigma_G^2e^{-b}$ for
$G_b=(2\pi)^{-1}\int X_b$. The endpoint estimate \eqref{paper:eq-all-mode} holds for every
\(0<p<1\).
Also
$\E[\mathbf1_{E_\beta}D_*(\T)^p]\le(2\pi\beta)^p$.

Write $F_*=\sum F_{*,k}e^{-ikx}$, $B_s(F_*)=\sup_k(1+|k|)^s|F_{*,k}|$, and
$F_*^{(K)}=\sum_{|k|\le K}F_{*,k}e^{-ikx}$. On
$\mathcal L_{\beta,R}=E_\beta\cap\{q_G^{-1}\le R,B_s(F_*)\le R\}$,
subadditivity, \eqref{paper:eq-all-mode},  and
\[
 \frac{a_{n+j}}{a_{n+j-k}}
 \le\left(1+\frac{\log(2+|k|)}{\log2}\right)^{1/2}
\]
bound the expected $p$th power of the prelimit truncation error by
$C_{\beta,R,p,s}N_wK^{1-sp}(\log(2+K))^{p/2}$ when $sp>1$.
For the limit use the same stable random measure for the multipliers
$F_*$ and $F_*^{(K)}$. Each coordinate error is Cauchy of scale at most
$Cq_G^{-1}\|F_*-F_*^{(K)}\|_\infty D_*(\T)$, so the corresponding bound
is $C_{\beta,R,p,s}N_wK^{-p(s-1)}$. This retains the dependence among
offsets. For $w+K\le n/2$ and $s>2$, changing finite-convolution
normalizers costs at most
$C_{\beta,R,p,s}(n\sqrt u)^{-p}$, because
$|a_{n+j}-a_{n+j-k}|\le C|k|/(n\sqrt u)$ and
$\sum_k|k||F_{*,k}|\le C_sR$.

The finite-convolution multiplier must be coarse measurable. Put
$q_b=e^{\sqrt2G_b-\sigma_G^2}$. On $\{q_G^{-1},q_b^{-1}\le R\}$,
$|q_G^{-1}-q_b^{-1}|\le\sqrt2R|G-G_b|$.
For $p<1/2$, Cauchy--Schwarz, Gaussian moments and the mass estimate
at exponent $2p$ show that this replacement costs
$C_{\beta,R,p,s}(u^{p/2}+N_w^{1/2})e^{-bp/2}$ in expected truncated
sup distance, counting source and target together. Moreover,
\[
 \Prob(q_b^{-1}>R)\le\Prob(q_G^{-1}>R/2)+C\exp(-ce^b).
\]
No independence of $G$ from $D_*$ is used.

On $\{q_b^{-1}\le R,B_s(F_*)\le R\}$ the map
\[
 \Psi_{b,K}(x)_j=\sqrt{2/\pi}\,q_b^{-1}
                         \sum_{|k|\le K}F_{*,k}x_{j-k}
\]
is $C_sR^2$-Lipschitz for truncated sup metrics, independently of
$K,w$. Its coefficients are measurable with respect to
$\F_b\vee\sigma(Y)$. Adjoining the independent smooth field leaves
the conditional source law unchanged. Apply
\eqref{pc:eq-window} to the composed test on the enlarged
window $w+K$, retaining the coarse event $E_{\beta,b}$.
Replace $Z^\beta$ by $D_*$ only on $E_\beta$. The bounded-test cost
of each barrier replacement is at most $2\Prob(E_\beta^c)$.
With
\[
 \Delta_{\beta,R}=\Prob(E_\beta^c)+\Prob(q_G^{-1}>R/2)
                         +\Prob(B_s(F_*)>R/2),
\]
we have therefore proved that the canonical BL distance is bounded by
\begin{align*}
 &C\Delta_{\beta,R}+C_R e_{\rm aux}(n,w+K)\\
 &\quad+C_{\beta,R,p,s}\bigl\{N_wK^{1-sp}(\log(2+K))^{p/2}
       +N_wK^{-p(s-1)}+(n\sqrt u)^{-p}\\
 &\hspace{4cm}+(u^{p/2}+N_w^{1/2})e^{-bp/2}+e^{-ce^b}\bigr\},
\end{align*}
where $e_{\rm aux}$ denotes the localized conditional distance in
\eqref{pc:eq-window}, bounded by Corollary~\ref{pc:cor-window}. Choose
$p=1/4$, $s=16$, $K=w=\lfloor\lambda^{1/64}\rfloor$.
The two Fourier-tail terms are
$O(\lambda^{-1/32}(\log\lambda)^{1/8})$ and
$O(\lambda^{-11/256})$; the other transfer errors decrease faster.
The doubled auxiliary window satisfies the same estimate.
Thus the bound is
$C\Delta_{\beta,R}+C_{\beta,R}\lambda^{-1/128}$.
First send $n\to\infty$, then exhaust $R,\beta\to\infty$.
This proves Theorem~\ref{canon:thm-window} and the estimate at each fixed
localization. No numerical rate is asserted after removing the localization.
\end{proof}

\section{The convergent integral class}
\label{paper:sec-upper}
The proof uses one coarse partition throughout each block of logarithmic frequency scales.
For each block, the mass in cells above the cutoff is controlled once, while concentration is applied
simultaneously to the contributions from the remaining cells.

\begin{proposition}[Convergent case for the auxiliary measure]
\label{l4v:cor-upper-integral-class}
For a positive deterministic $C^1$ gauge $H$ satisfying \eqref{intro:eq-gauge},
if $\int^\infty dx/H(x)<\infty$, then almost surely
\begin{equation}\label{l4v:eq-upper-integral-class}
 \lim_{|n|\to\infty}
 \frac{a_n|\widehat D_*(n)|}{H(\vartheta_n)}=0.
\end{equation}
\end{proposition}

\subsection{The predictable positive tail at an arbitrary cutoff}
At time $t$, use the complete three-color circular mesh
\[
 m(t)=3\left\lceil\frac{2\pi e^t}{3}\right\rceil,
 \quad d=2\pi/m(t),\quad y_i=id,\quad J_i=[y_i,y_i+d),
 \quad\theta=e^td\in(1/2,1].
\]
Cells in the same color are separated by $2d>e^{-t}$, so their terminal
restrictions are conditionally independent given $\mathcal F_t$.
Set
\[
 H_i(v)=X_t(y_i+e^{-t}v)-X_t(y_i),\quad
 a_i=e^{-t}w_t(y_i),\quad
 \mathcal A_t=\sum_i a_i e^{2\sqrt2\|H_i\|_\infty}.
\]
Here and below \(\|H_i\|_\infty\) denotes the supremum norm on \([0,1]\).

\begin{lemma}
\label{l4v:lem-positive-tail-budget}
For $0<\sigma<1$ put $h=\beta+\log(1/\sigma)/\sqrt2$ and
\[
 T_{t,\sigma}=\sum_i Z^\beta(J_i)
                           \mathbf1_{\{Z^\beta(J_i)>\sigma\}},
 \qquad B_{t,\sigma}=\mathbb E[T_{t,\sigma}\mid\mathcal F_t].
\]
On $E_{\beta,t}$,
\begin{align}
 \mathbb P(Z^\beta(J_i)>x\mid\mathcal F_t)
 &\le C_k a_i e^{2\sqrt2\|H_i\|_\infty}/x,
       \qquad x>0,\label{l4v:eq-cell-weakL1}\\
 B_{t,\sigma}&\le hW_t(\T)+C_k\mathcal A_t,
 \label{l4v:eq-predictable-positive-tail}\\
 \mathbb E[\mathbf1_{E_{\beta,t}}\mathcal A_t]
 &\le C_{\beta,k}t^{-1/2}.
 \label{l4v:eq-shape-budget}
\end{align}
Consequently
\begin{equation}\label{l4v:eq-discarded-deviation}
 \mathbb P\left(E_{\beta,t}\cap
       \{T_{t,\sigma}>B_{t,\sigma}+x\}\right)
 \le \frac{C_{\beta,k}}{x\sqrt t},\qquad x>0.
\end{equation}
The right-hand side of
\eqref{l4v:eq-discarded-deviation} is independent of \(h\).
\end{lemma}
\begin{proof}
The coarse/fine factorization on $E_{\beta,t}$ is
\[
 Z^\beta(J_i)=a_i Z^{[q_i]}(g_i\mathbf1_{[0,\theta]}),
 \quad g_i=e^{\sqrt2H_i},\quad
 q_i(v)=R_t^\beta(y_i)-H_i(v)>0.
\]
The untruncated fine mass dominates the positive-barrier mass, and
Lemma~\ref{crm:lem-finetail} yields \eqref{l4v:eq-cell-weakL1}.

Apply Lemma~\ref{crm:lem-cell-shift} of Appendix~\ref{crm:sec-endpoint-closure} with
$L$ there replaced by $1/\sigma$. Its identity uses only
$c=a_i/\sigma=e^{\sqrt2(h-R_t^\beta(y_i))}$, so it holds at
the present arbitrary time and cutoff. Multiplication by $\sigma$
gives
\[
 \E[Z^\beta(J_i)\1_{\{Z^\beta(J_i)>\sigma\}}\mid\F_t]
 \le \int_{J_i}w_t(x)\min(R_t^\beta(x),h)\,dx
       +C_ka_i e^{2\sqrt2\|H_i\|_\infty}.
\]
Summation proves \eqref{l4v:eq-predictable-positive-tail}.

Since $E_{\beta,t}$ implies survival at every anchor, rooting at $y_i$
and applying Cauchy--Schwarz gives
\[
 \mathbb E[\mathbf1_{E_{\beta,t}}a_i e^{2\sqrt2\|H_i\|_\infty}]
 \le\beta e^{-t}\mathbb E_{\mathbb Q_{t,y_i}^\beta}
       [R_t^{-1}e^{2\sqrt2\|H_i\|_\infty}]
 \le C_{\beta,k}e^{-t}t^{-1/2}.
\]
There are $O(e^t)$ cells, proving \eqref{l4v:eq-shape-budget}.

To prove \eqref{l4v:eq-discarded-deviation}, set
\[
 Y_i=
 Z^\beta(J_i)\1_{\{Z^\beta(J_i)>\sigma\}}.
\]
On \(E_{\beta,t}\), \eqref{l4v:eq-cell-weakL1} gives
\[
 \Prob(Y_i>x\mid\mathcal F_t)
 \le
 \frac{C_k a_i e^{2\sqrt2\|H_i\|_\infty}}{x},
 \qquad x>0.
\]
Hence the conditional probability that some \(Y_i>x\) is at most
\(C_k\mathcal A_t/x\). On the complementary event, truncate each
\(Y_i\) at \(x\). Its conditional variance satisfies
\[
 \operatorname{Var}
 \bigl(Y_i\1_{\{Y_i\le x\}}\mid\mathcal F_t\bigr)
 \le
 C_k x a_i e^{2\sqrt2\|H_i\|_\infty}.
\]
Within each of the three color classes the variables are conditionally
independent given \(\mathcal F_t\). Chebyshev's inequality for the three
centered color sums, followed by a union bound, therefore gives
\[
 \Prob\left(
 T_{t,\sigma}>B_{t,\sigma}+x
 \,\middle|\,\mathcal F_t
 \right)
 \le \frac{C_k\mathcal A_t}{x}
 \qquad\text{on }E_{\beta,t}.
\]
Multiplying by \(\1_{E_{\beta,t}}\), taking expectations, and using
\eqref{l4v:eq-shape-budget} proves
\eqref{l4v:eq-discarded-deviation}.
\end{proof}

\subsection{Proof of Proposition~\ref{l4v:cor-upper-integral-class}}
The derivative condition implies that $H$ is eventually nondecreasing
and $H(x)\le C_Hx^{A_H}$. The integral and series tests are equivalent.
Moreover
\[
 \frac{j}{2H(j)}
 \le
 \sum_{k=\lceil j/2\rceil}^j\frac1{H(k)}
 \to0,
\]
so $H(j)/j\to\infty$.
Fix a positive rational $\delta$.
Let $j$ range over sufficiently large integers. Put
\[
 N_j=2^j,\quad t_j=N_j\log2-6\log N_j,\quad
 g_j=\delta H(j),\quad
 \eta_j=g_j/\sqrt{t_j},
\]
\[
 \mathcal L_j=\{n\in\mathbb Z:2^{N_j}\le|n|<2^{2N_j}\}.
\]
These blocks are disjoint and cover all sufficiently large frequencies.
Use only the time $t_j$ and its three-color mesh throughout
$\mathcal L_j$, even at its largest frequencies. Fix $K\ge1$ and choose
\[
 \sigma_j=\frac{\eta_j^2}
                   {576(K+1)(\log2)N_j},\qquad
 h_j=\beta+\frac{\log(1/\sigma_j)}{\sqrt2}.
\]
For large $j$, $\eta_j,\sigma_j<1$ and $h_j\le C_{\beta,K,H,\delta}j$.

For each cell in this single partition define
\[
 V_i(n)=\int_{J_i}e^{inx}Z^\beta(dx),\quad
 Y_i(n)=V_i(n)\mathbf1_{\{Z^\beta(J_i)\le\sigma_j\}}
 -\mathbb E[V_i(n)\mathbf1_{\{Z^\beta(J_i)\le\sigma_j\}}
                                  \mid\mathcal F_{t_j}].
\]
Each scalar component has conditional variance at most
$\sigma_jZ_{t_j}^\beta(J_i)$, and $|Y_i|\le2\sigma_j$.
With three colors and real/imaginary components, Bernstein at threshold
$\eta_j/6$ yields
\begin{align}
 &\mathbb P\left(
   \max_{n\in\mathcal L_j}\left|\sum_iY_i(n)\right|>\eta_j,
                  Z_{t_j}^\beta(\T)\le K\right)\notag\\
 &\quad\le24\exp\left\{2N_j\log2
           -\frac{\eta_j^2}{72\sigma_j(K+\eta_j/9)}\right\}
 \le24\,2^{-6N_j}.
 \label{l4v:eq-block-concentration}
\end{align}
No cell phase is frozen: the Fourier integrals $V_i(n)$ are exact. Thus
using this coarse partition at the top of the frequency block creates
no approximation error.

At these geometric scale times, \eqref{l3:eq-SH-BL-rate} is summable:
\[
 \sum_j\mathbb E\left|\sqrt{t_j}W_{t_j}^\beta(\T)
                         -\sqrt{2/\pi}Z^\beta(\T)\right|<\infty.
\]
Consequently $\sup_j\sqrt{t_j}W_{t_j}^\beta(\T)<\infty$ almost surely.
The total derivative martingale also converges and is bounded along
this sequence. Hence the events
\[
 \Omega_{\beta,K,Q}=
 \left\{\sup_j Z_{t_j}^\beta(\T)\le K,\quad
        \sup_j\sqrt{t_j}W_{t_j}^\beta(\T)\le Q\right\},
 \qquad K,Q\in\mathbb N,
\]
exhaust the probability space. On $E_\beta$ we have
$W_{t_j}=W_{t_j}^\beta$, and on $E_\beta\cap\Omega_{\beta,K,Q}$,
\[
 \frac{h_jW_{t_j}(\T)}{\eta_j}
 \le C_{\beta,K,H,\delta}\frac{Qj}{g_j}\longrightarrow0.
\]
The event is used only after the unconditional bounds; all conditional
estimates use its corresponding time-$t_j$ measurable conditions.

Write $T_j=T_{t_j,\sigma_j}$ and $B_j=B_{t_j,\sigma_j}$. Fix any constant $c>0$. By
\eqref{l4v:eq-shape-budget} and \eqref{l4v:eq-discarded-deviation},
\begin{align}
 \mathbb P(E_{\beta,t_j}\cap\{\mathcal A_{t_j}>c\eta_j\})
 &\le C_{\beta,k,c}/g_j,\label{l4v:eq-shape-summable}\\
 \mathbb P(E_{\beta,t_j}\cap\{T_j>B_j+\eta_j\})
 &\le C_{\beta,k}/g_j.\label{l4v:eq-tail-summable}
\end{align}
Both sequences are summable because
\(\sum_j1/H(j)<\infty\).
On $E_\beta\cap\Omega_{\beta,K,Q}$, these estimates and
\eqref{l4v:eq-predictable-positive-tail} give
$T_j+B_j=O(\eta_j)$ almost surely. Crucially, $B_j$ is evaluated at one
time only, so no maximal inequality for predictable tails is needed.

The exact capped decomposition, valid at every frequency, is
\[
 \left|\widehat Z^\beta(n)-\widehat Z_{t_j}^\beta(n)
                         -\sum_iY_i(n)\right|\le T_j+B_j.
\]
On $E_\beta$, the smooth coarse density is $R_{t_j}^\beta w_{t_j}$.
Its expected total variation is at most
$C_{\beta,k}e^{t_j}(1+\sqrt{t_j})$, by Gaussian exponential tilting
and independence of the field from its spatial derivative at a point.
Integration by parts, using $|n|\ge2^{N_j}$ throughout the block, gives
\[
 \mathbb P\left(E_\beta\cap
  \left\{\max_{n\in\mathcal L_j}|\widehat Z_{t_j}^\beta(n)|>\eta_j
  \right\}\right)
 \le C_{\beta,k}\frac{N_j^{-5}}{g_j}.
\]
This is summable as well. Thus, for fixed $\beta,K,Q$, sufficiently
large $j$ depending on these parameters, and a fixed sufficiently large
constant $C_0=C_0(k)$, the entire argument supplies
\begin{align}
 &\mathbb P\left(E_\beta\cap\Omega_{\beta,K,Q}\cap
   \left\{\max_{n\in\mathcal L_j}|\widehat Z^\beta(n)|
                                       >C_0\eta_j\right\}\right)
 \notag\\
 &\qquad\le C_{\beta,k,K,Q,H,\delta}
          \left\{g_j^{-1}+2^{-6N_j}+N_j^{-5}/g_j\right\}.
 \label{l4v:eq-final-summable-block}
\end{align}

Borel--Cantelli and countable exhaustion in $K,Q$ and the integer
barriers give
\[
 \max_{n\in\mathcal L_j}|\widehat D_*(n)|
 \le C_k\delta H(j)/\sqrt{t_j}\quad\hbox{eventually almost surely}.
\]
Throughout $\mathcal L_j$, $\log|n|\asymp t_j\asymp2^j$ and
$\vartheta_n\asymp j$. Integrating the logarithmic derivative bound
therefore gives $H(j)/H(\vartheta_n)\le C_{A_H}$, with a constant
independent of $\delta$. Thus the limsup in
\eqref{l4v:eq-upper-integral-class} is at most $C_{k,A_H}\delta$.
Intersect the probability-one events for positive rational $\delta$
and then let $\delta\downarrow0$.
This proves Proposition~\ref{l4v:cor-upper-integral-class}.
Its canonical transfer is included in the proof of Proposition~\ref{canon:prop-integral}.

\section{The divergent integral class}
\label{paper:sec-lower}
\subsection{Relative excursion probabilities}
\label{pc:sec-excursion}

\begin{theorem}[Relative excursion estimate]
\label{pc:thm-excursion}
Fix $a>0$. Use $S=Z^\beta(\T)/\sqrt\pi$ and
$E=E_{\beta,b}$. Let $r\ge1$, $0<\epsilon\le1$, and
$v=\epsilon r$, with
\(
 r,v\in[\lambda^{-a},\lambda^a].
\)
There is a nonnegative remainder $\mathcal R$ such that, pointwise
on $E$,
\begin{equation}
 \Prob(a_n|\widehat Z^\beta(n)|>r\mid\F_b)
 \ge \frac S{2r}e^{-S/v}-\mathcal R,
 \label{pc:eq-excursion}
\end{equation}
and
\begin{equation}
 \E[\1_E\mathcal R]\le C_{\beta,k,q,a}\left\{
 \frac{\lambda^{-1/2}+u^{-1/5}}v+u^{-1/2}
 +\frac\epsilon r+\frac{\lambda^{-1/2}}r
 +\frac{u^{-7/2}}r+B_n(v)\right\}.
 \label{pc:eq-relative}
\end{equation}
The comparison variable $S$ and the remainder need not be
$\F_b$-measurable.
\end{theorem}

\begin{proof}
The Poisson process in Lemma~\ref{pc:lem-PPP} has exactly one radius
above $2r$ and no other radius above $v$ with probability
$p_{2r}e^{-p_v}$. Its large-jump sum then has modulus above $2r$,
regardless of the phase. On $0\le x\le y$, the function $xe^{-y}$
has both partial derivatives bounded by one. Thus this probability
is bounded below by
\[
 \frac S{2r}e^{-S/v}
       -|p_{2r}-S/(2r)|-|p_v-S/v|.
\]
Coupling back to the actual point process costs $\Pi_v$.
The exact scalar decomposition is
\[
 \sum_lz_lU_l=\sum_lz_lU_l\1_{\{U_l>v\}}+V_v+D_v(0),
 \quad
 V_v=\sum_lz_l\{U_l\1_{\{U_l\le v\}}
                  -\E[U_l\1_{\{U_l\le v\}}\mid\F_b]\}.
\]
Subtract the conditional probability of $|V_v|>r/2$, the coarse
indicator of $|D_v(0)|>r/4$, and the conditional probability that
the Fourier-to-cell error exceeds $r/4$. Off these events the
one-jump event implies $a_n|\widehat Z^\beta(n)|>r$.
These subtractions and the two tail errors define $\mathcal R$.

By \eqref{paper:eq-budget}, conditional Chebyshev costs $Cv/r^2=C\epsilon/r$ after
averaging. Corollary~\ref{l3c:cor-scalar-growing-center} and Markov cost $C\lambda^{-1/2}/r$.
The phase and normalization errors have expected size $O(u^{-7/2})$
as in Theorem~\ref{pc:thm-window}. Finally apply
\eqref{pc:eq-scalar-tail} at $v,2r$ and Lemma~\ref{pc:lem-PPP}.
Enlarging the fixed exponent range to include $2r$ is harmless.
No second moment of $S$ or independence of the large and small sums
is used.
\end{proof}

In the proof of Theorem~\ref{p4:thm-lower-assembly},
\eqref{pc:eq-relative} will be applied on a geometric sequence of frequencies, for which the remainder is summable.
Since the terminal Fourier coefficients are not adapted to the corresponding coarse
filtrations, we first approximate them by finite-time coefficients.

\subsection{Approximation by finite-time Fourier coefficients}
\label{l4w:sec-weak-bridge}

The next estimate compares the terminal Fourier coefficient with the coefficient at time \(T=K\log n\).
Its leading error has order
\(1/(\sqrt K\,r)\), so its coefficient can be made arbitrarily small relative to the \(1/r\) excursion probability by increasing the fixed constant \(K\).

\begin{theorem}[Weak-$L^1$ finite-time approximation]
\label{l4w:thm-weak-bridge}
Fix $\beta>0$.  Put $u=\log n$, $b=u+4\log u$, and let
$T=Ku$, where $K\ge2$ is deterministic and $u$ is sufficiently large
that $T\ge b$.  Then, uniformly over
$1\le r\le u^{1/4}$,
\begin{equation}\label{l4w:eq-weak-bridge}
 \Prob\left(E_{\beta,b}\cap
 \left\{\sqrt u\,
    |\widehat Z^\beta(n)-\widehat Z_T^\beta(n)|>r\right\}\right)
 \le \frac{C_{\beta,k}}{\sqrt K\,r}
           +C_{\beta,k}(1+\log u)^3u^{-1/5}.
\end{equation}
The constants are independent of $K,n,r$.  In particular, on a
geometric sequence $u_j=A^j$, with any fixed $A>1$, the second error
in \eqref{l4w:eq-weak-bridge} is summable in $j$.
The same assertion holds with $a_n=\sqrt{\log(2+n)}$ in place of
$\sqrt u$, after changing the constants by an absolute factor.
\end{theorem}

\begin{proof}
Use the time-$T$ mesh, set $L_T=\sqrt T$, and write
\[
 U_\ell=L_TZ^\beta(J_\ell),\quad
 z_\ell=e^{iny_\ell},\quad
 S=\sum_\ell z_\ell U_\ell,\quad
 P=L_T\sum_\ell z_\ell Z_T^\beta(J_\ell).
\]
Define
\[
 H_\ell(s)=X_T(y_\ell+e^{-T}s)-X_T(y_\ell),
 \qquad 0\le s\le1,
\]
and set
\[
 \mathcal A_T^{\mathrm{br}}
 =\sqrt T\sum_\ell e^{-T}w_T(y_\ell)
             e^{2\sqrt2\|H_\ell\|_\infty}
 =\sqrt T\,\mathcal A_T.
\]
By \eqref{l4v:eq-shape-budget},
\[
\E[\1_{E_{\beta,T}}\mathcal A_T^{\mathrm{br}}]
 \le C_{\beta,k}.
\]
Set $q=r\sqrt{T/u}=r\sqrt K$ and choose the mass cutoff $v=q/8$.
All cap and tail identities below are valid for every $v>0$.
Let
\[
 B_v=\sum_\ell z_\ell
       \E[U_\ell\1_{\{U_\ell\le v\}}\mid\F_T].
\]
The conditional variance of the centered sum of marks at most $v$
is at most $Cv\mathcal A_T^{\mathrm{br}}$, and
$\sum_\ell\Prob(U_\ell>v\mid\F_T)
 \le C\mathcal A_T^{\mathrm{br}}/v$.
On the event that no mark exceeds $v$, $S-B_v$ is exactly that
centered small-mark sum.  Conditional Chebyshev and a union bound give
\[
 \Prob(|S-B_v|>q/4\mid\F_T)
 \le C\mathcal A_T^{\mathrm{br}}/q.
\]
Since $E_{\beta,T}\in\F_T$ and
$\E[\1_{E_{\beta,T}}\mathcal A_T^{\mathrm{br}}]\le C_{\beta,k}$,
this contributes $C_{\beta,k}/q$ to the localized probability.

The exact cutoff-cap identity gives, on $E_{\beta,T}$,
\begin{equation}\label{l4w:eq-cutoff-bridge-center}
 \begin{split}
 B_v-P&=-h_v L_T\sum_\ell z_\ell W_T(J_\ell)+\mathcal R_v,\\
 h_v&=\beta+\frac1{\sqrt2}\log(L_T/v)
     =\beta+\frac1{\sqrt2}\log(8\sqrt u/r).
 \end{split}
\end{equation}
The remainder here includes the fine-field correction term, the log-positive cap correction, the fine-barrier correction, and
$-v\sum_\ell z_\ell\Prob(U_\ell>v\mid\F_T)$.
Taking the absolute value before summing the cell envelopes gives,
for any fixed $1<\eta<3/2$,
\begin{equation}\label{l4w:eq-center-remainder-L1}
 \E[\1_{E_{\beta,T}}|\mathcal R_v|]
 \le C_{\beta,k,\eta}
       \{1+v/\sqrt T+T^{1/2-\eta}h_v^{2\eta}\}.
\end{equation}
The first term uses the uniform capped finite-part bound, independent
of $v$; the last term is the small-radius H\"older estimate. The
fine-barrier coupling costs $v/\sqrt T$ because the cap has range
$[0,v]$. Finally, the weak-$L^1$ tail bound gives
\[
v\sum_\ell\Prob(U_\ell>v\mid\F_T)
 \le C\mathcal A_T^{\mathrm{br}}.
\]
No oscillatory estimate for this last sum is required.
For the stated range of $r$, $h_v>0$ and
$h_v\le C_\beta(1+\log u)$, uniformly in $K$.
Markov's inequality applied to \eqref{l4w:eq-center-remainder-L1}
therefore contributes at most
\[
 \frac{C_{\beta,k}}q+C_{\beta,k}T^{-1/2}
       +\frac{C_{\beta,k,\eta}}q T^{1/2-\eta}
                       (1+\log u)^{2\eta}.
\]

It remains to control the Fourier-specific additive drift in
\eqref{l4w:eq-cutoff-bridge-center}.  This is where a crude total-mass
bound would introduce an unnecessary factor $\log T$ into the leading
term.  Let $W_T^\beta(dx)=I_T^\beta(x)w_T(x)\,dx$ and
$c_{\rm SH}=\sqrt{2/\pi}$.  Applying \eqref{l3:eq-SH-test-rate} with
\(b=T\) and \(f(x)=e^{inx}\) gives
\begin{equation}\label{l4w:eq-SH-Fourier-uniform}
 \E\left|\sqrt T\,\widehat W_T^\beta(n)
                       -c_{\rm SH}\widehat Z^\beta(n)\right|
 \le C_{\beta,k}T^{-1/5},\qquad n\in\mathbb Z.
\end{equation}
The constant is independent of the Fourier frequency: the two test
norms of $e^{inx}$ are fixed.  On $E_{\beta,T}$, $W_T=W_T^\beta$.
Thus the probability that the $h_v$-multiple of the error in
\eqref{l4w:eq-SH-Fourier-uniform} exceeds a fixed multiple of $q$ is
at most $C_{\beta,k}h_vT^{-1/5}/q$.

By \eqref{crm:eq-endpoint-pmoment},
\[
 \E[\1_{E_\beta}|\sqrt u\,\widehat D_*(n)|^{1/2}]
 \le C_{\beta,k},\qquad n\ge2.
\]
Since $Z^\beta=D_*$ on $E_\beta$, Markov gives
\begin{equation}\label{l4w:eq-drift-small}
 \Prob\left(E_\beta\cap
       \{h_v|\widehat Z^\beta(n)|>c q\}\right)
 \le C_{\beta,k,c}\left(\frac{h_v}{q\sqrt u}\right)^{1/2}.
\end{equation}
Lemma~\ref{l4:lem-late-global-crossing} gives
\[
 \Prob(E_{\beta,T}\setminus E_\beta)
 \le C_{\beta,k}T^{-1/2}.
\]
This use of terminal localization is outside all conditional
expectations and therefore creates no independence assertion.

Freezing the Fourier phase in $S$ and $P$ costs at most
$C\sqrt T\,n e^{-T}(Z^\beta(\T)+Z_T^\beta(\T))$.
The analogous drift mesh error is at most
$Ch_v\sqrt T\,n e^{-T}W_T(\T)$.
The two derivative masses have bounded first moments, and
$\E W_T(\T)=2\pi$.  After division by $q$, their total probability
cost is at most
\[
 C_{\beta,k}(1+h_v)\frac{\sqrt u}{r}e^{-(K-1)u},
\]
which is smaller than the displayed remainder in
\eqref{l4w:eq-weak-bridge}, uniformly in $K\ge2$.

Choose $\eta=5/4$ and combine the bounds.  Since $q\ge\sqrt2$,
$T\ge u$, and $h_v=O_\beta(1+\log u)$, every error other than
$C_{\beta,k}/q$ is bounded by
$C_{\beta,k}(1+\log u)^3u^{-1/5}$.
The event in \eqref{l4w:eq-weak-bridge} is therefore controlled by the corresponding time-\(T\) cell difference at threshold \(q\), up to the mesh errors above.
Finally move the localization from $E_{\beta,T}$ to
$E_{\beta,b}$, at cost $C_{\beta,k}b^{-1/2}$, again absorbed by that
remainder.  This proves \eqref{l4w:eq-weak-bridge}.
For the alternate normalizer take $q=r\sqrt T/a_n$ throughout.
Since $a_n/\sqrt u=1+O((nu)^{-1})$, this changes the displayed bounds
by fixed factors only and keeps all constants independent of $K$.
\end{proof}

\begin{remark}
Let $K\ge2$ and $A>2K$, and set
\[
 n_j=\lceil e^{A^j}\rceil,\qquad
 u_j=\log n_j,\qquad
 b_j=u_j+4\log u_j,\qquad T_j=Ku_j.
\]
Then $u_j=A^j+O(e^{-A^j})$ and, for all sufficiently large $j$,
$T_j<b_{j+1}$. The Fourier coefficient at time $T_j$ is therefore
measurable at the next coarse filtration.  If $r_j\ge1$ grows at most
polynomially in $j$, Theorem~\ref{l4w:thm-weak-bridge} gives
\[
 \sum_{j\le N}\Prob\left(E_{\beta,b_j}\cap
 \{\sqrt{u_j}|\widehat Z^\beta(n_j)
                    -\widehat Z_{T_j}^\beta(n_j)|>r_j\}\right)
 \le\frac{C_{\beta,k}}{\sqrt K}\sum_{j\le N}\frac1{r_j}+O_{\beta,k}(1).
\]
Thus the approximation error can be made an arbitrarily small proportion of
\(\sum_{j\le N}r_j^{-1}\), but need not be summable.
Consequently, ordinary conditional Borel--Cantelli does not suffice to pass from the finite-time coefficients to the terminal coefficients.
Lemma~\ref{p4:lem-count} is designed to use precisely this cumulative control.
\end{remark}

\subsection{An adapted counting lemma}

\begin{lemma}
\label{p4:lem-count}
Let $(\mathcal G_j)$ be increasing, let $X_j$ be terminal complex
random variables, let $V_j$ be $\mathcal G_{j+1}$-measurable, and let
$J_j\in\mathcal G_j$. Let deterministic $r_j\ge1$ satisfy
$Q_N=\sum_{j\le N}r_j^{-1}\to\infty$. Suppose $F$ is an event of
probability $p_0>0$, and $c>0$ satisfies
\begin{equation}\label{p4:eq-cumulative-lower}
 \liminf_N\frac1{Q_N}\sum_{j\le N}\mathbf1_{J_j}
     \mathbb P(|X_j|>4r_j\mid\mathcal G_j)\ge c
 \quad\text{almost surely on }F.
\end{equation}
If
\begin{equation}\label{p4:eq-cumulative-error}
 \limsup_N\frac1{Q_N}\sum_{j\le N}
       \mathbb P(J_j,|X_j-V_j|>r_j)\le\eta,
\end{equation}
then
\[
 \mathbb P(F\cap\{|X_j|>r_j\text{ infinitely often}\})
 \ge p_0-16\eta/c.
\]
Neither $F$ nor $X_j$ is required to be adapted.
\end{lemma}

\begin{proof}
Set $B_j=J_j\cap\{|V_j|>2r_j\}$,
$C_j=J_j\cap\{|X_j-V_j|>r_j\}$, and
\[
 p_j=\mathbb P(B_j\mid\mathcal G_j),\quad
 d_j=\mathbb P(C_j\mid\mathcal G_j),\quad
 a_j=\mathbf1_{J_j}\mathbb P(|X_j|>4r_j\mid\mathcal G_j).
\]
Then $p_j\ge a_j-d_j$. Write $P_N=\sum_{j\le N}p_j$,
$D_N=\sum_{j\le N}d_j$, and $N_C(N)=\sum_{j\le N}\mathbf1_{C_j}$.
The tower property gives $\mathbb E D_N=\mathbb E N_C(N)$.
Fatou's lemma and \eqref{p4:eq-cumulative-error} imply
\[
 \mathbb E\liminf_N\frac{D_N+N_C(N)}{Q_N}\le2\eta.
\]
Consequently
\[
 F'=F\cap\left\{\liminf_N\frac{D_N+N_C(N)}{Q_N}<c/8\right\}
 \quad\text{satisfies}\quad
 \mathbb P(F')\ge p_0-16\eta/c.
\]
On $F'$, there are arbitrarily large $N$ for which
$D_N+N_C(N)<cQ_N/8$. Along these indices,
\[
 P_N\ge\sum_{j\le N}a_j-D_N
       \ge(7c/8-o(1))Q_N\longrightarrow\infty.
\]
Since $P_N$ is nondecreasing, it tends to infinity on $F'$.

The variable $\mathbf1_{B_j}-p_j$ is a martingale difference
from $\mathcal G_j$ to $\mathcal G_{j+1}$, and
\[
 \sum_j\frac{\mathbb E[(\mathbf1_{B_j}-p_j)^2
                              \mid\mathcal G_j]}{(1+P_j)^2}
 \le\sum_j\frac{p_j}{(1+P_j)^2}\le1.
\]
Thus the normalized martingale series converges almost surely.
Kronecker's lemma gives
$\sum_{j\le N}\mathbf1_{B_j}/P_N\to1$ on $\{P_N\to\infty\}$.
If $B_j$ occurs and $C_j$ fails, then $|X_j|>r_j$. Along the
selected indices and for large $N$,
\begin{align*}
 \sum_{j\le N}\mathbf1_{\{|X_j|>r_j\}}
 &\ge \tfrac12 P_N-N_C(N)\\
 &\ge\tfrac12\sum_{j\le N}a_j
        -\tfrac12D_N-N_C(N)
 \ge(3c/8-o(1))Q_N\longrightarrow\infty.
\end{align*}
The event $F$ remains in the conclusion because the entire argument
has been made on $F'\subset F$.
\end{proof}

\subsection{Infinitely many auxiliary excursions}

\begin{theorem}[Positive-probability infinite auxiliary limsup]
\label{p4:thm-lower-assembly}
Let $H$ be a fixed deterministic positive $C^1$ function such that,
eventually,
\[
 H(x)\ge x,\qquad 0\le xH'(x)/H(x)\le A_H<\infty,
 \qquad \int^\infty\frac{dx}{H(x)}=\infty.
\]
Then
\[
 \mathbb P\left(
 \limsup_{n\to\infty}\frac{a_n|\widehat D_*(n)|}{H(\vartheta_n)}
       =\infty\right)>0.
\]
\end{theorem}

\begin{proof}
Choose $\beta>0$ and $s\in(0,1)$ once and for all so that
\[
 S_\beta=Z^\beta(\T)/\sqrt\pi,\qquad F=E_\beta\cap\{s\le S_\beta\le s^{-1}\},\qquad
 p_0=\mathbb P(F)>0.
\]
This is possible by positivity and finiteness of $D_*(\mathbb T)$ and
exhaustion by $E_\beta$. Put $c=s/16$, and choose $K\ge2$ so large
that $\eta=C_\beta/\sqrt K\le cp_0/32$. Set $d=4K$ and
\[
 n_j=\lceil e^{d^j}\rceil,\quad u_j=\log n_j,\quad
 b_j=u_j+4\log u_j,\quad T_j=Ku_j.
\]
Eventually $T_j<b_{j+1}$, and $\log u_j=j\log d+o(1)$.
All of $F,p_0,c,K,d$ are chosen before the excursion amplitude.

Fix an arbitrary integer $m\ge1$, and let
\[
 r_j=mH(\vartheta_{n_j}),\quad
 X_j=a_{n_j}\widehat Z^\beta(n_j),\quad
 V_j=a_{n_j}\widehat Z_{T_j}^\beta(n_j),\quad
 \mathcal G_j=\mathcal F_{b_j},\quad J_j=E_{\beta,b_j}.
\]
Only finitely many initial indices need to be deleted. Integration of
the derivative bound gives
$1\le H(y)/H(x)\le(y/x)^{A_H}$ for sufficiently large $y\ge x$.
Consequently $r_j\ge c_mj$, $r_j\le C_mj^{A_H}$,
$r_j\le u_j^{1/4}$ eventually, and
$Q_N=\sum_{j\le N}r_j^{-1}\to\infty$. The last assertion is the
integral test on the fixed-spacing mesh $j\log d$; the negligible
change to $\vartheta_{n_j}$ is harmless by the same ratio estimate.
The variables $V_j$ are $\mathcal G_{j+1}$-measurable.

Apply Theorem~\ref{pc:thm-excursion} at $4r_j$, with $\epsilon_j=j^{-1/4}$. The thresholds
belong to one fixed logarithmic-power range for each fixed $m$.
The first and third terms in \eqref{pc:eq-relative} are at most
constants times $j^{-1/4}/r_j=O(j^{-5/4})$; the center term is
$O(j^{-3/2})$, and every term containing a negative power of $u_j$
decays exponentially up to a polynomial factor. Thus the resulting
remainders $R_j$ satisfy
\[
 \sum_j\mathbb E[\mathbf1_{J_j}R_j]<\infty,
 \qquad \sum_j\mathbf1_{J_j}R_j<\infty\quad\text{almost surely}.
\]
On $F$, all $J_j$ occur, $s\le S_\beta\le s^{-1}$, and
$\epsilon_jr_j\to\infty$. Hence
\[
 \mathbb P(|X_j|>4r_j\mid\mathcal G_j)
 \ge \frac{S_\beta}{8r_j}
         e^{-S_\beta/(4\epsilon_jr_j)}-R_j
 \ge \frac c{r_j}-R_j
\]
eventually on $F$. Summing proves
\eqref{p4:eq-cumulative-lower}. It does not assert that
$R_j=o(1/r_j)$ pointwise.

Theorem~\ref{l4w:thm-weak-bridge} gives
\[
 \sum_{j\le N}\mathbb P(J_j,|X_j-V_j|>r_j)
 \le\eta Q_N+C_\beta\sum_j(1+\log u_j)^3u_j^{-1/5}.
\]
The second sum is finite. Dividing by $Q_N\to\infty$ proves
\eqref{p4:eq-cumulative-error}. Lemma~\ref{p4:lem-count} therefore
gives
\[
 \mathbb P\bigl(F\cap
  \{a_{n_j}|\widehat Z^\beta(n_j)|>
              mH(\vartheta_{n_j})\text{ infinitely often}\}\bigr)
 \ge p_0/2.
\]
The lower probability is the same for every integer $m$. Since
$Z^\beta=D_*$ on $F$, the event on the left is contained in
$F\cap\{\limsup_n a_n|\widehat D_*(n)|/H(\vartheta_n)\ge m\}$.
These latter events decrease in $m$. Continuity from above proves
that their intersection, the infinite-limsup event intersected with
$F$, has probability at least $p_0/2>0$.
\end{proof}

\subsection{Smooth multiplication and the canonical integral classes}
\label{canon:subsec-integral}
\begin{definition}[Moderate asymptotically shift-invariant weight]
A positive sequence $(h_n)_{n\in\mathbb Z}$ is admissible here if
\[
 \frac{h_n}{h_{n-k}}\longrightarrow1\quad(|n|\to\infty)
 \quad\text{for every fixed integer }k,
\]
and for some finite $C,A$,
\[
 \frac{h_n}{h_{n-k}}\le C(1+|k|)^A
 \qquad(n,k\in\mathbb Z).
\]
For a finite measure $\mu$ put
$\mathcal L_h(\mu)=\limsup_{|n|\to\infty}h_n|\widehat\mu(n)|$,
allowing values in $[0,\infty]$.
\end{definition}

\begin{lemma}
\label{l4t:lem-pathwise-multiplier}
Let $h$ be admissible, $\mu$ a finite complex measure, and
$b\in C^\infty(\mathbb T;\mathbb C)$. Write
$b(x)=\sum_{k\in\mathbb Z}b_ke^{-ikx}$. Then
\begin{equation}\label{l4t:eq-envelope-multiplier}
 \mathcal L_h(b\mu)\le
       \left(\sum_k|b_k|\right)\mathcal L_h(\mu).
\end{equation}
If $b$ is nowhere zero, its reciprocal is smooth and also
\begin{equation}\label{l4t:eq-envelope-inverse}
 \mathcal L_h(\mu)\le
       \left(\sum_k|(b^{-1})_k|\right)\mathcal L_h(b\mu).
\end{equation}
In particular the three alternatives
$\mathcal L_h=0$, $0<\mathcal L_h<\infty$, and
$\mathcal L_h=\infty$ are preserved by such multiplication.
\end{lemma}
\begin{proof}
If $\mathcal L_h(\mu)=\infty$, the first inequality is automatic.
Otherwise $B=\sup_n h_n|\widehat\mu(n)|<\infty$, since each of
the finitely many remaining coefficients is finite. The exact Fourier
convolution gives
\[
 h_n\widehat{b\mu}(n)
   =\sum_k b_k\frac{h_n}{h_{n-k}}
                      h_{n-k}\widehat\mu(n-k).
\]
For $|k|\le K$, take the limsup using the fixed-shift ratios. The
remaining absolute tail is bounded, uniformly in $n$, by
$CB\sum_{|k|>K}|b_k|(1+|k|)^A$, which tends to zero by smoothness.
Letting $K\to\infty$ proves
\eqref{l4t:eq-envelope-multiplier}. Apply it to $b\mu$ and $b^{-1}$
to prove \eqref{l4t:eq-envelope-inverse}. The three alternatives
follow, including the infinite case by contradiction if necessary.
\end{proof}

For $h_n=a_n/H(\vartheta_n)$, with any positive extension over finitely
many initial indices, these hypotheses hold. To verify this, set
$\theta(t)=\log\log(e^e+t)$ and
$q(t)=\sqrt{\log(2+t)}/H(\theta(t))$. For large $t$,
\[
 \frac{d}{dt}\log q(t)
 =\frac1{2(2+t)\log(2+t)}
 -\frac{H'(\theta(t))}{H(\theta(t))}
                 \frac1{(e^e+t)\log(e^e+t)}.
\]
Since $0\le H'(\theta)/H(\theta)\le A_H/\theta\to0$, this derivative
is eventually nonnegative and bounded above by
$1/[2(2+t)\log(2+t)]$. Thus $q$ is eventually increasing and fixed shifts
have ratio one. If $|n|\ge|n-k|$ are both large, monotonicity of $H$ bounds
$h_n/h_{n-k}$ by $a_n/a_{n-k}\le C(1+|k|)^{1/2}$; in the reverse
ordering use monotonicity of $q$. Positivity of the finitely many initial
values changes only the constant. Hence the weights are admissible.

On the coupling \eqref{crm:eq-same-space-transfer}, the multiplier
$F_*$ is smooth and strictly positive, and $q_G$ is a finite positive
scalar. 
By \eqref{crm:eq-same-space-transfer} and
Lemma~\ref{l4t:lem-pathwise-multiplier}, the conclusion
\(\mathcal L_h=0\) from
Proposition~\ref{l4v:cor-upper-integral-class} and the
positive-probability divergence conclusion of
Theorem~\ref{p4:thm-lower-assembly} transfer from \(D_*\) to \(M\).
The following zero--one argument upgrades the latter conclusion to an almost-sure statement.
\begin{proposition}
\label{l4t:prop-envelope-zero-one}
For canonical zero-constant-mode circle chaos and a fixed admissible
weight $h$, each of the events
\[
 \{\mathcal L_h(M)=0\},\quad
 \{0<\mathcal L_h(M)<\infty\},\quad
 \{\mathcal L_h(M)=\infty\}
\]
has probability zero or one. In particular, if
\(\{\mathcal L_h(M)=\infty\}\)
has positive probability, then it holds almost surely.
\end{proposition}
\begin{proof}
Realize the exact circle field through its independent sine and cosine
Gaussian coordinates. A deterministic real trigonometric polynomial
$v$ of mean zero is a Cameron--Martin shift. The canonical mollifier
definition gives, on the same probability space,
\[
 M_{\phi+v}=e^{\sqrt2v}M_\phi.
\]
The multiplier is smooth and strictly positive. Therefore each displayed
event is invariant, modulo null sets, under every finite Gaussian
coordinate translation. It suffices to use rational coordinate shifts,
a countable family, so the exceptional sets can be intersected.

Condition on all but the first $m$ Gaussian
coordinates. Their density is strictly positive with respect to
Lebesgue measure. By Fubini and translation equivalence, the conditional
event indicator is almost everywhere invariant under rational
translations of these $m$ coordinates. A bounded measurable function
with this invariance is constant almost everywhere: convolve it with
a smooth approximate identity, use continuity and density of rational
translations, and then remove the convolution. The event is therefore
independent of the first $m$ coordinates for every $m$. Kolmogorov's
zero--one law for the independent Gaussian coordinate sequence completes
the proof. This proves class membership, not constancy of a finite
positive numerical limsup.
\end{proof}
\begin{proposition}[Canonical integral classes]\label{canon:prop-integral}
For the exact circle measure $M$, and any positive deterministic $C^1$
gauge satisfying \eqref{intro:eq-gauge}, almost surely
\[
 \mathcal L_h(M)=
 \begin{cases}
 0,&\displaystyle\int_{x_0}^\infty\frac{dx}{H(x)}<\infty,\\[3pt]
 \infty,&\displaystyle\int_{x_0}^\infty\frac{dx}{H(x)}=\infty,
 \end{cases}
 \qquad h_n=\frac{a_n}{H(\vartheta_n)}.
\]
The same alternatives hold after division by $M(\T)$; one event works
for each prescribed countable family of such gauges.
\end{proposition}
\begin{proof}
For a divergent gauge,
Theorem~\ref{p4:thm-lower-assembly},
\eqref{crm:eq-same-space-transfer}, and
Lemma~\ref{l4t:lem-pathwise-multiplier}
show that
\(\mathbb P(\mathcal L_h(M)=\infty)>0\).
Proposition~\ref{l4t:prop-envelope-zero-one} then gives
\(\mathcal L_h(M)=\infty\) almost surely.
For a convergent gauge,
Proposition~\ref{l4v:cor-upper-integral-class} and the same smooth
transfer give $\mathcal L_h(M)=0$ almost surely.
This proves the canonical alternatives. Multiplication by the positive
finite scalar $1/M(\T)$ preserves both alternatives. Finally a
countable intersection yields one event for any prescribed countable
family of gauges; no assertion for a sample-dependent gauge is used.
\end{proof}

\section{Smooth covariance universality}
\label{uni:sec-universality}
The preceding sections establish the stable Fourier limits and integral classes for the canonical field.
We now extend these conclusions to every covariance in \eqref{intro:eq-covariance}.
The stable-limit argument also uses the independent endpoint bound, whereas the integral-class argument does not use the stable-limit theorem.
Smooth covariance decompositions provide a broader setting for such comparisons \cite{JSW}; here we use an explicit two-sided comparison at the joint field--measure level.

\subsection{A two-sided smooth covariance comparison and critical chaos}
\label{uni:subsec-construction}
\begin{lemma}
\label{uni:lem-sandwich}
For every real symmetric $g\in C^\infty(\T^2)$ there are real stationary
smooth covariances $C$ and $S_0$ such that $S_0$ has strictly positive,
rapidly decreasing Fourier coefficients and
$C+g\ge S_0$, $C-g\ge S_0$ as quadratic forms. In particular
$C$ and $C\pm g$ are covariances. Every real smooth covariance on
$\T^2$ has a Gaussian realization with almost surely smooth paths.
\end{lemma}
\begin{proof}
Use normalized Haar measure only for this operator calculation. In the
orthonormal basis $e_n(x)=e^{\ii nx}$ the operator with kernel $g$ has matrix
\[
 G_{mn}=\frac1{(2\pi)^2}\int_{\T^2}g(x,y)e^{\ii ny-\ii mx}\,dx\,dy.
\]
For every integer $N\ge0$, integration by parts gives
$|G_{mn}|\le C_N(1+m^2)^{-N}(1+n^2)^{-N}$. Set
$r_m=\sum_n|G_{mn}|$ and $d_m=r_m+r_{-m}$. These sequences decrease
rapidly, and $d_m$ is even and nonnegative. Since $G$ is Hermitian, for
finitely supported $z$,
\[
 |\langle Gz,z\rangle|
 \le\tfrac12\sum_{m,n}|G_{mn}|(|z_m|^2+|z_n|^2)
 =\sum_m r_m|z_m|^2\le\sum_m d_m|z_m|^2.
\]
Thus $-D\le G\le D$, where $De_m=d_me_m$, and density extends these
inequalities to $L^2$. Choose any even sequence $s_m>0$ decreasing
exponentially in $|m|$, let
$S_0(x,y)=\sum_ms_me^{\ii m(x-y)}$, and let $S$ be its diagonal
operator. Replace $D$ by $D+S$ and let $C$ be the corresponding stationary
kernel. Then $C$ is real and smooth and satisfies $C\pm g\ge S_0$.
Positivity on finite sets of points follows by smooth
approximation of point masses. This proves the first assertion, including
the zero mode.

For any real smooth positive covariance $A$, its positive operator is
trace class. Its real $L^2$ Gaussian realization satisfies
\[
 \E\|Y_A\|_{H^q}^2=\sum_m(1+m^2)^q A_{mm}<\infty
 \qquad(q\in\mathbb N).
\]
A countable intersection and Sobolev embedding give one smooth version.
Its Fourier truncations converge in every smooth seminorm, so it is a
measurable $C^\infty$-valued random element with covariance $A$.
\end{proof}

Fix admissible $g$. On one space take $Y_+$ independent of $\phi$, with
covariance $C+g$; on another, extend the given target space by $Y_-$
independent of $X^g$, with covariance $C$. Then
\begin{equation}\label{uni:eq-field-bridge}
 H_+=\phi+Y_+\ \stackrel{\Law}=\ H_-=X^g+Y_-,\qquad
 \operatorname{Cov}(H_\pm)=K_g+C=:K_\sharp.
\end{equation}
The equality holds in a fixed separable distribution space, for example
$H^{-1}(\T;\mathbb R)$. The Fourier variances of $K_0$ and the rapid
decay of the smooth remainder ensure realizations there. No inverse
covariance or nondegeneracy assumption is used.

For a centered field $X$ write $M_{X,\varepsilon}$ for the approximation
in \eqref{intro:eq-chaos}. If $Y$ is independent of $X$ and smooth, define
\[
 F_{Y,\varepsilon}=e^{\sqrt2Y_\varepsilon-\E Y_\varepsilon^2},\qquad
 F_Y=e^{\sqrt2Y-\E Y^2}.
\]
Additivity of the variances gives the exact finite-scale identity
\begin{equation}\label{uni:eq-finite-factorization}
 M_{X+Y,\varepsilon}=F_{Y,\varepsilon}M_{X,\varepsilon}.
\end{equation}
Both $F_{Y,\varepsilon}$ and its inverse converge uniformly almost surely
to $F_Y$ and its inverse. The variance correction in these formulas is a
function of $x$, not necessarily a constant. Continuity of multiplication
in the weak topology of finite measures shows that either approximation
family converges in probability if and only if the other does; their limits
then satisfy $M_{X+Y}=F_YM_X$ on this coupling.

\begin{proposition}
\label{uni:prop-construction}
For every admissible $g$, \eqref{intro:eq-chaos} has a full-family weak
probability limit $M_g$, measurable with respect to $X^g$, with
$0<M_g(\T)<\infty$ almost surely.
It is independent of the auxiliary construction above and inherits the standard mollifier-independence of canonical circle chaos.
There is one Borel map $\mathfrak M_\sharp$ on the common augmented
field space such that, with $F_\pm=e^{\sqrt2Y_\pm-\E Y_\pm^2}$,
\begin{equation}\label{uni:eq-measure-bridge}
\begin{gathered}
 N_+=\mathfrak M_\sharp(H_+)=F_+M,\qquad
 N_-=\mathfrak M_\sharp(H_-)=F_-M_g,\\
 (H_+,N_+)\stackrel{\Law}=(H_-,N_-).
\end{gathered}
\end{equation}
\end{proposition}
\begin{proof}
We first record the measurability point. If Borel maps
$\mathcal A_\varepsilon:S\to T$, with $S$ standard Borel and $T$ Polish,
satisfy $\mathcal A_\varepsilon(U)\to V$ in probability, a deterministic
almost surely convergent subsequence defines a Borel limit map
$\mathfrak M$ on its convergence set; give it a fixed value elsewhere.
Then $V=\mathfrak M(U)$ almost surely. For $U'\stackrel{\Law}=U$,
\[
 \Prob\big(d(\mathcal A_\varepsilon(U'),\mathfrak M(U'))>\eta\big)
 =\Prob\big(d(\mathcal A_\varepsilon(U),\mathfrak M(U))>\eta\big)
 \longrightarrow0.
\]
Thus the same full-family convergence holds for \(U'\).
The subsequence selects a version; it does not replace full-family convergence.

For fixed covariance $Q$, convolution and the deterministic variance
function make $M_{X,\varepsilon}$ a Borel map
$\mathcal A_{Q,\varepsilon}(X)$ into the Polish space of finite positive
measures. Canonical construction and \eqref{uni:eq-finite-factorization}
give $M_{H_+,\varepsilon}\to F_+M$. The observation above supplies
$\mathfrak M_\sharp(H_+)=F_+M$ and gives the same convergence for $H_-$,
using the same maps $\mathcal A_{K_\sharp,\varepsilon}$ on both sides.
It also gives the joint field--measure law in \eqref{uni:eq-measure-bridge}.
The inverse identity
\[
 M_{X^g,\varepsilon}=F_{-,\varepsilon}^{-1}M_{H_-,\varepsilon}
\]
therefore converges weakly in probability to $M_g:=F_-^{-1}N_-$.
Positivity, finiteness, and positive total mass follow from those of $N_-$
and the strictly positive continuous multiplier on the compact circle.
Every target prelimit is a function of $X^g$ alone. A second use of the
Borel-limit observation makes $M_g$ a function of $X^g$, and proves the
convergence on the original target space as well.

Uniqueness of the probability limit shows that \(M_g\) does not depend on the auxiliary construction. For two standard mollifiers whose canonical limits agree, the two forward limits
are both $F_+M$. Their Borel limit maps agree under the common augmented
law, hence also on $H_-$. Multiplication by $F_-^{-1}$ proves agreement
of the target limits. This is pairwise mollifier-independence, not an
uncountable intersection of exceptional events.
\end{proof}

\subsection{Random smooth multiplication and localized endpoint bounds}
\label{uni:subsec-multiplier}
Throughout this subsection,
\[
 f_k=\frac1{2\pi}\int_\T f(x)e^{\ii kx}\,dx,\quad
 f(x)=\sum_k f_ke^{-\ii kx},\quad
 B_s(f)=\sup_k(1+|k|)^s|f_k|.
\]
In particular $\widehat{f\mu}(m)=\sum_kf_k\widehat\mu(m-k)$.
We shall use the all-integer bound
\begin{equation}\label{uni:eq-ratio}
 \frac{a_m}{a_{m-k}}\le\ell(k):=
 \left(1+\frac{\log(1+|k|)}{\log2}\right)^{1/2},
 \qquad \frac{a_m}{a_{m-k}}\longrightarrow1\quad(|m|\to\infty)
\end{equation}
for each fixed $k$. It follows from
$2+|m|\le(2+|m-k|)(1+|k|)$.

\begin{proposition}
\label{uni:prop-multiplier}
Let \(\mu\) be an almost surely finite positive random measure that is \(\mathcal E\)-measurable, where \(\mathcal E\) is a sigma-field, and fix \(c>0\).
Suppose that,
for every finite $J\subset\mathbb Z$, the vectors
$(a_{n+j}\widehat\mu(n+j))_{j\in J}$ converge stably relative to
$\mathcal E$, as $n\to+\infty$, with conditional exponent
\begin{equation}\label{uni:eq-input-profile}
 -c\int_\T\left|\sum_{j\in J}\overline{\xi_j}e^{\ii jx}\right|\mu(dx)
 \qquad(J\subset\mathbb Z\text{ finite}).
\end{equation}
Suppose also that, for one $p\in(0,1)$, increasing $A_r\in\mathcal E$
exhaust probability one and
\begin{equation}\label{uni:eq-endpoint-input}
 \sup_{m\in\mathbb Z}\E[\1_{A_r}|a_m\widehat\mu(m)|^p]\le C_r<\infty.
\end{equation}
Then the following assertions hold.
\begin{enumerate}
\item[(i)] \textup{Weighted stable limits.} For every fixed finite family of \(\mathcal E\)-measurable random \(C^\infty\) functions \(f_1,\ldots,f_d\), the vector
$(a_n\widehat{f_\ell\mu}(n))_{\ell=1}^d$ converges stably relative to
$\mathcal E$ to a stable vector with conditional exponent
\begin{equation}\label{uni:eq-weighted-kernel}
 -c\int_\T\left|\sum_{\ell=1}^d\overline{\xi_\ell}f_\ell(x)\right|\mu(dx).
\end{equation}

\item[(ii)] \textup{Localized endpoint bounds.}
For every $\mathcal E$-measurable random smooth $b$, the measure $b\mu$ inherits a
localized all-mode bound of the form \eqref{uni:eq-endpoint-input}, with
the same $p$. This inheritance requires only the endpoint hypothesis,
not \eqref{uni:eq-input-profile}.
\end{enumerate}

In particular, positive smooth multiplication, and its inverse when nowhere zero, preserve both hypotheses with the same $c$ and $p$. No independence or moments of the smooth norms
are assumed. The functions are measurable as $C^\infty$-valued random elements, and all families and functions are fixed as $n$ varies.
\end{proposition}
\begin{proof}
Fix $s>1/p$. For $f_\ell^{(K)}=\sum_{|k|\le K}(f_\ell)_ke^{-\ii kx}$,
\[
 a_n\widehat{f_\ell^{(K)}\mu}(n)
 =\sum_{|k|\le K}(f_\ell)_k\frac{a_n}{a_{n-k}}
                         a_{n-k}\widehat\mu(n-k).
\]
For fixed \(K\), this is a finite linear map whose coefficients are \(\mathcal E\)-measurable. Such maps preserve stable convergence: for a finite-valued matrix, partition according to its
\(\mathcal E\)-measurable values in the stable tests; approximate a general matrix $B$ in probability by finite-valued $B^{(q)}$ and use tightness of the source vectors $V_n$, since
\[
 \Prob(\|(B-B^{(q)})V_n\|>\eta)
 \le\Prob(\|B-B^{(q)}\|>\eta/T)+\Prob(\|V_n\|>T).
\]
The same bound treats the limiting vector and the convergent normalizer
ratios in \eqref{uni:eq-ratio}. This proves the truncated assertion.

Set $L_{r,R}=A_r\cap\{\max_\ell B_s(f_\ell)\le R\}$ and
\[
 \tau(K)=\sum_{|k|>K}(1+|k|)^{-sp}\ell(k)^p,\qquad
 \eta_s(K)=\sum_{|k|>K}(1+|k|)^{-s}.
\]
Both tend to zero. Exact Fourier convolution, \eqref{uni:eq-ratio}, and
$p$-subadditivity give
\begin{equation}\label{uni:eq-prelimit-tail}
 \sup_n\E\!\left[\1_{L_{r,R}}
  \max_{\ell\le d}|a_n\widehat{(f_\ell-f_\ell^{(K)})\mu}(n)|^p\right]
 \le dC_rR^p\tau(K).
\end{equation}
First bound the random coefficients by $R(1+|k|)^{-s}$ on the localization;
only then discard the extra smooth-norm indicator in nonnegative expectations.

On one extension, take a conditional Poisson random measure on
$\T\times\mathbb S^1\times(0,\infty)$, with intensity
\[
 c\mu(dx)m_{\rm H}(dz)r^{-2}dr.
\]
Compensate $r\le1$ and integrate
$f(x)rz$. After localization of $\mu(\T)$ and $\|f\|_\infty$, the small
jumps converge in $L^2$ and there are finitely many large jumps. Angular
symmetry gives zero drift. The identities
\[
 \int_0^\infty(1-\cos(ar))r^{-2}dr=\tfrac\pi2|a|,\qquad
 \int_{\mathbb S^1}|\operatorname{Re}(zH)|\,m_{\rm H}(dz)
 =\tfrac2\pi|H|
\]
give \eqref{uni:eq-weighted-kernel}.
This constructs a common noise for all weights and truncations.
A coordinate error is isotropic Cauchy of scale at most
$cR\eta_s(K)\mu(\T)$. Its conditional $p$th moment is this scale to the
power $p$ times a finite constant depending only on $p$. The zero mode of
\eqref{uni:eq-endpoint-input} gives
$\E[\1_{A_r}\mu(\T)^p]\le(\log2)^{-p/2}C_r$. Consequently, for the full and truncated
limiting vectors $Z$ and $Z^{(K)}$,
\begin{equation}\label{uni:eq-limit-tail}
 \widetilde\E[\1_{L_{r,R}}\|Z-Z^{(K)}\|_\infty^p]
 \le C_{p,c}\,dC_rR^p\eta_s(K)^p.
\end{equation}
The factor $(\log2)^{-p/2}$ has been absorbed into the constant.

For any bounded \(\mathcal E\)-measurable variable and bounded Lipschitz test, insert the two truncated vectors and use \eqref{uni:eq-prelimit-tail} and
\eqref{uni:eq-limit-tail} with Markov's inequality. First let $n\to\infty$
at fixed $K$, then $K\to\infty$ at fixed $r,R$, then exhaust $R,r$.
The localization complements cost only their probabilities times the test
bound. This proves stable convergence for bounded Lipschitz tests. Taking
the environmental variable equal to one gives tightness; bounded weighted
laws are then tight by domination, which extends the conclusion to all
bounded continuous tests. Fixed offsets in the weighted statement follow
by replacing $f_\ell$ by $e^{\ii j_\ell x}f_\ell$ and using
$a_{n+j_\ell}/a_n\to1$.

For endpoint inheritance alone, on
$A_r\cap\{B_s(b)\le R\}$ the same convolution calculation gives
\begin{equation}\label{uni:eq-endpoint-inheritance}
 \sup_m\E[\1_{A_r\cap\{B_s(b)\le R\}}
                 |a_m\widehat{b\mu}(m)|^p]
 \le C_rR^p\sum_k(1+|k|)^{-sp}\ell(k)^p<\infty.
\end{equation}
Taking $R=r$ gives increasing exhausting events. This proof does not use the stable convergence assumption. For positive $F$, apply the weighted result to
$Fe^{\ii jx}$: the exponent is \eqref{uni:eq-input-profile} with control
$F\mu$. If $F$ is nowhere zero, $F^{-1}$ is a measurable smooth random
function and the same argument applies again.
\end{proof}

\begin{lemma}
\label{uni:lem-environment}
The following assertions hold.
\begin{enumerate}
\item[(i)] \textup{Projection.}
Stable convergence relative to $\mathcal E$ projects to a smaller
sigma-field whenever its limit kernel is measurable there.

\item[(ii)] \textup{Enlargement.}
If the prelimit is $\mathcal E$-measurable, stable convergence
relative to $\mathcal E$ also holds relative to any larger base-space
sigma-field, with the same kernel.

\item[(iii)] \textup{Localization.}
Fix $p>0$. If each $V_m$ is measurable with respect to a standard Borel random
element $U$ and
\[
 L_r\uparrow\Omega\text{ almost surely},\qquad
 \sup_m\E[\1_{L_r}|V_m|^p]\le K_r,
\]
there are increasing Borel conditions $D_r$ on $U$, exhausting its law,
such that
\[
 \sup_m\E[\1_{\{U\in D_r\}}|V_m|^p]\le2K_r.
\]
\end{enumerate}

Thus, for a measure measurable with respect to the field,
the stable Fourier convergence and the localized endpoint bound
are preserved under equality in law of the underlying fields.
\end{lemma}


\begin{proof}
Projection follows by restricting the class of measurable test variables in the definition of stable convergence.
For enlargement, if $A$ is
a bounded variable on the larger base space and $h$ is bounded continuous,
\[
 \E[Ah(V_n)]=\E[\E(A\mid\mathcal E)h(V_n)]
 \longrightarrow\E[\E(A\mid\mathcal E)Qh]=\E[AQh].
\]
The added limiting noise is conditionally independent of further base
information given the kernel environment. For localization, take monotone
versions $q_r(U)=\E[\1_{L_r}\mid U]$ and let $D_r=\{q_r\ge1/2\}$.
Then $q_r\uparrow1$ and
\[
 \E[\1_{\{U\in D_r\}}|V_m|^p]
 \le2\E[q_r(U)|V_m|^p]=2\E[\1_{L_r}|V_m|^p].
\]
Truncation justifies this identity without global integrability. Borel
versions of the countably many conditions can be chosen consistently.
The concluding transport statement is equality of expectations of Borel
functions of the field, using the same Borel measure map on both spaces.
\end{proof}

\begin{lemma}
\label{uni:lem-canonical-endpoint}
For every fixed $p\in(0,1)$ there are increasing
$A_r\in\sigma(\phi)$, exhausting probability one, such that
\begin{equation}\label{uni:eq-canonical-endpoint}
 \sup_{m\in\mathbb Z}\E[\1_{A_r}|a_m C_m|^p]<\infty.
\end{equation}
\end{lemma}
\begin{proof}
The all-mode auxiliary bound \eqref{paper:eq-all-mode} was derived
directly from Theorem~\ref{crm:thm-endpoint-moment} in
Section~\ref{paper:sec-common}. On the coupling
\eqref{crm:eq-same-space-transfer},
$M=\sqrt{2/\pi}\,q_G^{-1}F_*D_*$. Endpoint inheritance
\eqref{uni:eq-endpoint-inheritance}, or its direct convolution proof,
gives canonical bounds on
$E_r\cap\{q_G^{-1}\le r,B_s(F_*)\le r\}$, for $s>1/p$.
These increase to full probability. Lemma~\ref{uni:lem-environment}
makes their replacements measurable with respect to $\phi$ alone.
\end{proof}

\subsection{Proof of the general stable Fourier results}
\label{uni:subsec-profile-proof}
\begin{proof}[Proof of Theorem~\ref{contract:A}]
Use the coupling \eqref{uni:eq-field-bridge} together with the measures
in \eqref{uni:eq-measure-bridge}. Fix any $p\in(0,1)$.
Proposition~\ref{canon:prop-profile} and
Lemma~\ref{uni:lem-canonical-endpoint} give the canonical stable Fourier
limit and the localized endpoint moment bound.
Since its prelimit is $\phi$-measurable, the canonical stable convergence
enlarges to $\sigma(\phi,Y_+)$. Proposition~\ref{uni:prop-multiplier}
applied to $F_+$ gives the stable Fourier limit and endpoint bound for $N_+=F_+M$,
with unchanged stable constant $1/\sqrt\pi$.

Now $N_+=\mathfrak M_\sharp(H_+)$. Its limiting kernel is a function
of $H_+$, so Lemma~\ref{uni:lem-environment} projects the stable convergence and
the endpoint localizations to $\sigma(H_+)$. Both transfer by the joint
law \eqref{uni:eq-measure-bridge} to $H_-$. The prelimit coefficients
of $N_-$ are $H_-$-measurable, so the stable convergence may therefore be viewed relative to
\(\sigma(H_-,Y_-)=\sigma(X^g,Y_-)\).

Apply Proposition~\ref{uni:prop-multiplier} to the dependent multiplier
$b=F_-^{-1}$. For $P(x)=\sum_{j\in J}\overline{\xi_j}e^{\ii jx}$,
its exponent is exactly
\[
 -\frac1{\sqrt\pi}\int_\T|bP|\,dN_-
 =-\frac1{\sqrt\pi}\int_\T|P|\,dM_g.
\]
The positivity of $b$ and stable index one give this identity with no
additional diagonal correction. Both the target measure and this kernel
are $X^g$-measurable; projection to $\sigma(X^g)$ proves the theorem.
The common-noise construction in Proposition~\ref{uni:prop-multiplier}
retains the joint law of every finite offset set.

Taking $J=\{0\}$ gives the stated planar Cauchy law. Multiplication by
the finite
\(\sigma(X^g)\)-measurable scalar $1/M_g(\T)$ gives its deterministic
normalized scale and the normalized vector control
$M_g/(\sqrt\pi M_g(\T))$.
\end{proof}

\begin{corollary}
\label{uni:cor-target-endpoint}
For each fixed admissible $g$ and $p\in(0,1)$, there are increasing events
$E_r^g\in\sigma(X^g)$ exhausting probability one such that, for every $r$,
\[
 \sup_{m\in\mathbb Z}\E[\1_{E_r^g}|a_m C_{g,m}|^p]<\infty.
\]
For any fixed finite family of $X^g$-measurable random smooth functions,
the weighted stable limit with conditional exponent
\eqref{uni:eq-weighted-kernel} holds with
$\mu=M_g$ and $c=1/\sqrt\pi$.
\end{corollary}
\begin{proof}
Transport endpoint bounds through $F_+$, the common field law, and
$F_-^{-1}$ using only \eqref{uni:eq-endpoint-inheritance} and
Lemma~\ref{uni:lem-environment}; project once more to $X^g$. Thus the
endpoint assertion does not use Theorem~\ref{contract:A}. The weighted assertion then
follows from Theorem~\ref{contract:A} and
Proposition~\ref{uni:prop-multiplier}.
\end{proof}

\begin{proof}[Proof of Corollary~\ref{contract:functional}]
Fix $s>1$. Choose $p\in(1/s,1)$ and then $r$ with
$1/p<r<s$. Use the equivalent Fourier norm
\[
 \|T\|_{H^{-r}}^2
 =\sum_{j\in\mathbb Z}(1+j^2)^{-r}|T(e^{ij\cdot})|^2.
\]
By Corollary~\ref{uni:cor-target-endpoint}, for increasing
$E_R^g\in\sigma(X^g)$ exhausting probability one,
\[
 K_R:=\sup_{m\in\mathbb Z}
 \E[\1_{E_R^g}|a_mC_{g,m}|^p]<\infty.
\]
Since $p/2<1$, subadditivity and \eqref{uni:eq-ratio}, used with
$m=n$ and $k=-j$, give uniformly in $n$
\begin{align*}
 \E[\1_{E_R^g}\|\mathcal T_n^g\|_{H^{-r}}^p]
 &\le\sum_{j\in\mathbb Z}(1+j^2)^{-rp/2}
       \E[\1_{E_R^g}|a_nC_{g,n+j}|^p]\\
 &\le K_R\sum_{j\in\mathbb Z}(1+j^2)^{-rp/2}\ell(j)^p<\infty.
\end{align*}
The last series converges because $rp>1$. Thus, after choosing $R$ and
then a large $H^{-r}$ ball, Markov's inequality and the compact embedding
$H^{-r}(\T)\hookrightarrow H^{-s}(\T)$ give tightness of
$(\mathcal T_n^g)_n$ in $H^{-s}$.

The conditional stable noise $\mathcal S_{M_g}$ has an $H^{-s}$-valued
realization. Indeed, in the compensated Poisson construction used in
Proposition~\ref{uni:prop-multiplier},
$\sup_x\|\delta_x\|_{H^{-s}}<\infty$ for $s>1/2$. Hence the compensated
small jumps converge conditionally in $L^2(H^{-s})$, while only finitely
many jumps have radius greater than one.

For a fixed finite $J\subset\mathbb Z$,
\[
 \bigl(\mathcal T_n^g(e^{ij\cdot})\bigr)_{j\in J}
 =\left(\frac{a_n}{a_{n+j}}a_{n+j}C_{g,n+j}\right)_{j\in J}.
\]
The ratios tend to one, so Theorem~\ref{contract:A} gives stable
convergence of every finite family of Fourier coordinates to those of
$\mathcal S_{M_g}$.

It remains to pass from coordinates to bounded continuous functions on
$H^{-s}$. Let $V$ be bounded and $\sigma(X^g)$-measurable and let
$F\in C_b(H^{-s};\mathbb R)$. Tightness of the prelimit family and of the
limiting law gives a compact set $K\subset H^{-s}$ carrying arbitrarily
large probability for both. On $K$, bounded continuous cylinder functions
depending on finitely many real and imaginary Fourier coordinates form a
point-separating self-adjoint algebra containing constants. By
Stone--Weierstrass they are uniformly dense in $C(K)$. Choose such a
cylinder function $G$ approximating $F$ on $K$ and clip $G$ to the range
$[-\|F\|_\infty,\|F\|_\infty]$. The finite-coordinate stable convergence
applies to $G$. The two replacement errors are bounded by the uniform
approximation error on $K$ plus a constant times the probability of
$K^c$. Letting first $n\to\infty$, then the approximation error and tail
probability tend to zero, proves the asserted stable convergence in
$H^{-s}$.

Finally, $H^s(\T)$ is separable and embeds continuously into
$C(\T)$ for $s>1$, while the dual pairing
$H^{-s}\times H^s\to\mathbb C$ is continuous. The deterministic weighted
consequence follows by continuous mapping. For a fixed finite family of
$\sigma(X^g)$-measurable $H^s$-valued tests, approximate the tests in probability by finite-range
\(\sigma(X^g)\)-measurable \(H^s\)-valued simple functions and localize
their $H^s$ norms; tightness of $\mathcal T_n^g$ and the same argument for
the limit remove the approximation. This gives the stated random-test
extension without independence assumptions.
\end{proof}

\subsection{Transfer of the explicit growing window}
\label{uni:subsec-window-transfer}
The fixed-dimensional covariance transfer above permits a multiplier that
depends on the augmented target field.  For a growing window we need one
additional quantitative fact: the smooth multiplier constructed below can
be replaced, to arbitrarily high logarithmic accuracy, by a smooth multiplier
measurable at the same coarse scale as the canonical window comparison.

\begin{lemma}
\label{uni:lem-coarse-regression}
Fix an admissible $g$.  The covariance $C$ in
Lemma~\ref{uni:lem-sandwich} may be chosen so that there is a stationary
smooth covariance $S_0$ with strictly positive rapidly decreasing Fourier
coefficients and
\[
 C+g\ge S_0,\qquad C-g\ge S_0
\]
as quadratic forms.  On an extension of the canonical auxiliary coupling
there are positive smooth random functions $B_g$ and $B_{g,b}$ such that
\begin{enumerate}
\item[(i)] $M_g^+:=B_gM$ has the same law as $M_g$;
\item[(ii)] $B_{g,b}$ is measurable with respect to
$\F_b\vee\mathcal G$, where $\mathcal G$ is generated by finitely many
smooth Gaussian fields independent of the auxiliary scale field;
\item[(iii)] for every integer $q\ge0$ and every $A<\infty$,
\begin{equation}\label{uni:eq-regression-certificate}
 \E\|B_g-B_{g,b}\|_{C^q}^2
 \le C_{A,q,g}\{N^{-A}+e^{-b}(1+N)^3\}
\end{equation}
when the regression cutoff is $N\ge1$;
\item[(iv)] after taking $N=\lfloor\lambda\rfloor$, where
$\lambda=\log\log n$, the families $B_{g,b}$ and $B_{g,b}^{-1}$ are tight
in every fixed smooth seminorm.  Moreover, for every fixed $q$ there are increasing events
\(\mathcal G_R^g\uparrow\Omega\)
such that
\begin{equation}\label{uni:eq-regression-bad}
 \Prob\bigl(\mathcal G_R^g,\
       \|B_g-B_{g,b}\|_{C^q}>\lambda^{-1}\bigr)
 \le C_{R,q,g}\lambda^{-2}
\end{equation}
for all sufficiently large $n$.
\end{enumerate}
\end{lemma}
\begin{proof}
Start with the diagonal majorant $D$ in the proof of
Lemma~\ref{uni:lem-sandwich}.  Let $S_0$ be any real stationary smooth
covariance with strictly positive rapidly decreasing Fourier coefficients,
and replace $C$ by $D+S_0$.  Then $C\pm g\ge S_0$.

On the canonical auxiliary coupling, choose a centered smooth Gaussian
field $Y_+$ of covariance $C+g$, independently of the joint auxiliary
construction $(X,Y)$, where $X$ denotes the entire auxiliary scale field
and $Y$ is the smooth field in \eqref{crm:eq-same-space-transfer}.
Write $H_+=\phi_*+Y_+$ and adjoin a jointly Gaussian smooth field
$\widetilde Y_-$ so that
\[
 (H_+,\widetilde Y_-)\stackrel{\Law}=(H_-,Y_-).
\]
Thus
\[
 \operatorname{Cov}(\widetilde Y_-)=C,
 \qquad
 \operatorname{Cov}(\widetilde Y_-,H_+)=C.
\]
Let
\[
 m=\E[\widetilde Y_-\mid H_+].
\]
The residual covariance is deterministic, so we may realize
$\widetilde Y_-=m+Y_0$ with $Y_0$ a smooth Gaussian field independent of
the entire canonical-side construction.  Then
$X^{g,+}:=H_+-\widetilde Y_-$ has covariance $K_g$ and is independent of
$\widetilde Y_-$.  Applying the smooth critical-chaos factorization in
two ways gives
\begin{equation}\label{uni:eq-target-copy}
 M_g^+=B_gM,
 \qquad
 B_g=F_+F_{\widetilde Y_-}^{-1},
 \qquad
 M_g^+\stackrel{\Law}=M_g.
\end{equation}

It remains to coarsen $m$.  Let $P_N$ be Fourier projection to
$|k|\le N$ and put
$m_N=\E[\widetilde Y_-\mid P_NH_+]$.  In the Fourier basis write
$Q=\operatorname{Cov}(H_+)=K_0+C+g$, where $K_0=K_g-g$, and split the
modes into $L=\{|k|\le N\}$ and $U=L^c$.  Since
$Q\ge K_0+S_0=:A$ and $A$ is diagonal, with
$A_{kk}\ge c(1+|k|)^{-1}$, the Schur complement
\[
 S_N=Q_{UU}-Q_{UL}Q_{LL}^{-1}Q_{LU}
\]
satisfies
\begin{equation}\label{uni:eq-schur-regression}
 S_N\ge A_{UU},
 \qquad S_N^{-1}\le A_{UU}^{-1},
 \qquad \|Q_{LL}^{-1}\|\le C(1+N).
\end{equation}
The first inequality follows from the variational characterization of the Schur complement applied to $Q\ge A$.

The inverse inequality is understood in the sense of quadratic forms.
More explicitly, let $U_J=\{k:N<|k|\le J\}$ and
$S_{N,J}=(S_N)_{U_JU_J}$.  For every vector $v$ on $U_J$,
\[
 \langle v,S_{N,J}^{-1}v\rangle
 \le \sum_{k\in U_J}A_{kk}^{-1}|v_k|^2
 \le C\sum_{k\in U_J}(1+|k|)|v_k|^2.
\]
The infinite-mode regression estimate below follows by letting
$J\to\infty$, using $L^2$ convergence of conditional expectations
with respect to the increasing Fourier sigma-fields.

The cross-covariance of $\partial_x^q\widetilde Y_-(x)$ with the Fourier
modes of $H_+$ is a Fourier coefficient of the smooth kernel $C$; the
off-diagonal Fourier matrix of $Q$ is that of the smooth kernel $C+g$.
Hence both decrease faster than every power.  Conditioning first on the
low modes, the residual cross-covariance vector on $U$ is therefore bounded
by
\[
 |b_k|\le C_{q,M}(1+N)(1+|k|)^{-M}
\]
for every $M$.  Equation~\eqref{uni:eq-schur-regression} yields, for every
$A,q<\infty$,
\[
 \sup_x\E|\partial_x^q(m-m_N)(x)|^2
 \le C_{A,q,g}N^{-A}.
\]
Using two additional derivatives and Sobolev embedding gives the same
bound for $\E\|m-m_N\|_{C^q}^2$.

Now use the canonical auxiliary coupling
$\phi_*=X_*+Y-G$ from \eqref{crm:eq-same-space-transfer}, so that
\[
 H_+=X_*+Y-G+Y_+.
\]
Put
\[
 H_{+,b}=X_b+Y-G_b+Y_+.
\]
The future covariance kernel has $L^1$ mass $O(e^{-b})$, hence
\[
 \sup_k\E|\widehat{X_*-X_b}(k)|^2\le Ce^{-b},
 \qquad \E|G-G_b|^2=\sigma_G^2e^{-b}.
\]
Consequently
\[
 \E\|P_N(H_+-H_{+,b})\|_{\ell^2}^2
 \le Ce^{-b}(2N+1).
\]
By the finite-dimensional regression formula, the map
$H\mapsto m_N(H)$ is the low-mode cross-covariance followed by
$Q_{LL}^{-1}P_N$.  Rapid decay of that cross-covariance and
\eqref{uni:eq-schur-regression} show that its $C^q$ operator norm is at
most $C_{q,g}(1+N)$.  Hence
\[
 \E\|m_N(H_+)-m_N(H_{+,b})\|_{C^q}^2
 \le C_{q,g}e^{-b}(1+N)^3.
\]
Set
\[
 \widetilde Y_{-,b,N}=m_N(H_{+,b})+Y_0,
 \qquad
 B_{g,b}=F_+\exp\{-\sqrt2\widetilde Y_{-,b,N}+C(x,x)\}.
\]
The asserted coarse measurability holds with
$\mathcal G=\sigma(Y,Y_+,Y_0)$.
We now justify the second-moment estimate for the exponential multiplier.

Fix $q\ge0$, put $r=q+2$, and write $m_{N,b}=m_N(H_{+,b})$.
The decomposition
\[
 H_+=H_{+,b}+(X_*-X_b)-(G-G_b)
\]
has a future term independent of $H_{+,b}$.
Hence, as covariance forms,
\[
 \operatorname{Cov}(m_{N,b})
 \le \operatorname{Cov}(m_N(H_+))
 \le \operatorname{Cov}(\widetilde Y_-)=C.
\]
Since $C$ is smooth, this gives uniformly bounded covariance traces
in $H^r$ for $m_{N,b}$; the same holds for $m$.

Set
\[
 U=\log B_g,\qquad U_{N,b}=\log B_{g,b}.
\]
These Gaussian fields have the same deterministic mean $-g(x,x)$.
The preceding covariance bounds, together with the fixed smooth Gaussian
fields $Y_+$ and $Y_0$, give uniformly bounded $H^r$ covariance traces
for their centered versions.
The Gaussian exponential-moment formula therefore gives a uniform
square-exponential moment in $H^r$ at a sufficiently small positive parameter.
After absorbing the fixed mean, Sobolev embedding yields, for every
fixed $c>0$,
\[
 \sup_{N\ge1,\,b\ge0}
 \E\exp\!\left\{
 c\bigl(\|U\|_{C^q}+\|U_{N,b}\|_{C^q}\bigr)
 \right\}<\infty.
\]

Moreover, $U-U_{N,b}=-\sqrt2(m-m_{N,b})$ is centered Gaussian.
Applying the preceding derivative estimates through order $r$ and
integrating over $\T$ yields
\[
 \E\|U-U_{N,b}\|_{H^r}^2
 \le C_{A,q,g}\bigl(N^{-A}+e^{-b}(1+N)^3\bigr).
\]
For real $u,v\in C^q(\T)$, the algebra property of $C^q$ gives
\[
 \|e^u-e^v\|_{C^q}
 \le C_q\|u-v\|_{C^q}
 \exp\!\left\{C_q\bigl(\|u\|_{C^q}+\|v\|_{C^q}\bigr)\right\}.
\]
Using Cauchy--Schwarz, the uniform exponential moments, and the
centered Gaussian Hilbert-space bound
$\E\|Z\|^4\le3(\E\|Z\|^2)^2$, we obtain
\[
 \begin{aligned}
 \E\|B_g-B_{g,b}\|_{C^q}^2
 &\le C_{q,g}
       \bigl(\E\|U-U_{N,b}\|_{H^r}^4\bigr)^{1/2}\\
 &\le C_{q,g}\E\|U-U_{N,b}\|_{H^r}^2\\
 &\le C_{A,q,g}\bigl(N^{-A}+e^{-b}(1+N)^3\bigr).
 \end{aligned}
\]
This proves \eqref{uni:eq-regression-certificate}.

Take $N=\lfloor\lambda\rfloor$, $b=\log n+4\log\log n$, and $A=6$.
Markov's inequality gives, for all sufficiently large $n$,
\[
 \Prob\!\left(\|B_g-B_{g,b}\|_{C^q}>\lambda^{-1}\right)
 \le C_{q,g}\lambda^2
       \bigl(\lambda^{-6}+e^{-b}(1+\lambda)^3\bigr)
 \le C_{q,g}\lambda^{-2}.
\]
For example, the increasing events
\[
 \mathcal G_R^g
 =\{\|B_g\|_{C^q}+\|B_g^{-1}\|_{C^q}\le R\}
\]
exhaust probability one and satisfy \eqref{uni:eq-regression-bad}.

Finally, \eqref{uni:eq-regression-certificate} gives
$B_{g,b}\to B_g$ in probability in every fixed $C^q$ norm.
Since $B_g$ is strictly positive on the compact circle, continuity of
inversion on positive $C^q$ functions also gives
$B_{g,b}^{-1}\to B_g^{-1}$ in probability in $C^q$.
Both families are therefore tight in every fixed smooth seminorm.
\end{proof}

\begin{theorem}[Smooth-covariance growing-window refinement]
\label{uni:thm-window}
Fix an admissible $g$ and let
\[
 w_n=\lfloor(\log\log n)^{1/64}\rfloor.
\]
Then
\begin{equation}\label{uni:eq-window}
 d_{\mathrm{BL},w_n}\left(
 \Law((a_{n+j}C_{g,n+j})_{|j|\le w_n}),
 \Law((\mathcal Z_{g,j})_{|j|\le w_n})\right)\longrightarrow0.
\end{equation}
The approximation is unconditional and pointwise in the fixed covariance
remainder $g$; no uniformity in $g$ or optimality of the exponent $1/64$
is asserted.
\end{theorem}
\begin{proof}
Use the representation \eqref{uni:eq-target-copy} on the canonical-side coupling, together with
\[
 M=\sqrt{2/\pi}\,q_G^{-1}F_*D_*,
 \qquad
 M_g^+=\sqrt{2/\pi}\,B_gq_G^{-1}F_*D_*.
\]
Put $u=\log n$, $b=u+4\log u$, and
$\lambda=\log u$.  Choose
\[
 p=\frac14,
 \qquad \sigma=16,
 \qquad K=w=\lfloor\lambda^{1/64}\rfloor,
 \qquad \varepsilon_n=\lambda^{-1},
\]
and take the regression cutoff $N=\lfloor\lambda\rfloor$ in
Lemma~\ref{uni:lem-coarse-regression}.  Let $N_w=2w+1$.

We first replace $B_g$ by $B_{g,b}$.  Choose $q>\sigma+2$ and put
\[
 \mathcal R_R^g=\mathcal G_R^g\cap
 \{q_G^{-1}\le R,\ \|F_*\|_{C^q}\le R\},
 \qquad
 \mathcal A_n^g=\{\|B_g-B_{g,b}\|_{C^q}\le\varepsilon_n\}.
\]
On $\mathcal R_R^g\cap\mathcal A_n^g$, Fourier convolution,
\eqref{uni:eq-ratio}, the localized all-mode endpoint bound and
$p$-subadditivity give
\begin{equation}\label{uni:eq-window-regression-pre}
 \E\left[\1_{E_\beta\cap\mathcal R_R^g\cap\mathcal A_n^g}
  \max_{|j|\le w}
  \left|a_{n+j}\widehat{(B_g-B_{g,b})M}(n+j)\right|^p
  \right]
 \le C_{\beta,R,g}N_w\varepsilon_n^p.
\end{equation}
Here the factor $q_G^{-1}F_*$ from
$M=\sqrt{2/\pi}\,q_G^{-1}F_*D_*$ is included in the localized smooth
multiplier before applying the endpoint estimate.  The same common stable
random measure used in the canonical window proof gives
\begin{equation}\label{uni:eq-window-regression-limit}
 \widetilde\E\left[\1_{E_\beta\cap\mathcal R_R^g\cap\mathcal A_n^g}
  \max_{|j|\le w}|Z_{g,j}^{+}-Z_{g,b,j}^{+}|^p\right]
 \le C_{\beta,R,g}N_w\varepsilon_n^p.
\end{equation}
Here \(Z_{g,j}^+\) has control measure \(B_gM/\sqrt\pi\), while
\(Z_{g,b,j}^+\) has control measure \(B_{g,b}M/\sqrt\pi\).
By Lemma~\ref{uni:lem-coarse-regression},
$\Prob(\mathcal G_R^g\cap(\mathcal A_n^g)^c)
\le C_{R,g}\lambda^{-2}$.
Since $1\wedge x\le x^p$, these estimates apply directly to the truncated
sup metric.

For the actual conditional window comparison use only coarse events.  Put
\[
 \mathcal C_{\beta,R,b}^g
 =E_{\beta,b}\cap\{q_b^{-1}\le R,
                   B_\sigma(B_{g,b}F_*)\le R\}.
\]
This event is measurable with respect to
$\F_b\vee\mathcal G$, with $\mathcal G$ independent of the auxiliary
future.  The families $q_b^{-1}$ and $B_\sigma(B_{g,b}F_*)$ are tight;
indeed $q_b\to q_G$ and $B_{g,b}\to B_g$ in probability in every fixed smooth seminorm.  
For $R\ge1$, define
\[
 \Delta^g_{\beta,R}
 :=R^{-1}+\limsup_{n\to\infty}
 \Prob\bigl((\mathcal C_{\beta,R,b}^g)^c
       \cup(\mathcal R_R^g)^c\bigr).
\]
The preceding tightness and localization bounds give
\[
 \lim_{R\to\infty}\lim_{\beta\to\infty}
 \Delta^g_{\beta,R}=0.
\]
The positive slack $R^{-1}$ ensures that, for each fixed
$\beta,R,g$ and all sufficiently large $n$,
\[
 \Prob\bigl((\mathcal C_{\beta,R,b}^g)^c
       \cup(\mathcal R_R^g)^c\bigr)
 \le \Delta^g_{\beta,R}.
\]
Here $E_\beta\subset E_{\beta,b}$ and the existing estimate for
$q_b-q_G$ controls the scalar part.  The non-coarse exact variables are
used only in the replacement bounds, never as conditioning events.

On $\mathcal C_{\beta,R,b}^g$, truncate the Fourier series of
$B_{g,b}F_*$ at $|k|\le K$.  The resulting convolution map, including
the scalar $q_b^{-1}$, is $C_R$-Lipschitz for the truncated sup metrics,
uniformly in $K,w$, and is measurable with respect to the enlarged coarse
sigma-field.  Conditioning also on the independent fields in $\mathcal G$
does not change the auxiliary future law.  Therefore the conditional
window estimate \eqref{pc:eq-window} applies to the composed test on the
enlarged source window $w+K$.  Exactly as in the proof of
Theorem~\ref{canon:thm-window}, the Fourier tails and normalizer replacement
contribute
\begin{align*}
 C_{\beta,R,g}\bigl\{
 &N_wK^{1-\sigma p}(\log(2+K))^{p/2}
 +N_wK^{-p(\sigma-1)}+(n\sqrt u)^{-p}\\
 &+(u^{p/2}+N_w^{1/2})e^{-bp/2}+e^{-ce^b}\bigr\},
\end{align*}
while the conditional source-window term is
$C_{R,g}e_{\rm aux}(n,w+K)$.  Hence
\begin{align}\label{uni:eq-window-master}
 d_{\mathrm{BL},w} \le{}&
 C\Delta^g_{\beta,R}+C_{R,g}\lambda^{-2}
 +C_{\beta,R,g}N_w\lambda^{-p}
 +C_{R,g}e_{\rm aux}(n,w+K)\notag\\
 &+C_{\beta,R,g}\bigl\{
 N_wK^{1-\sigma p}(\log(2+K))^{p/2}
 +N_wK^{-p(\sigma-1)}+(n\sqrt u)^{-p}\notag\\
 &\hspace{34mm}+(u^{p/2}+N_w^{1/2})e^{-bp/2}+e^{-ce^b}\bigr\}.
\end{align}
Corollary~\ref{pc:cor-window}, with the harmless constant enlargement
$w+K\le2\lambda^{1/64}$, gives
$e_{\rm aux}(n,w+K)=O_{\beta}(\lambda^{-1/128})$.
The terms introduced by the present multiplier approximation satisfy
\[
 N_w\lambda^{-p}=O(\lambda^{-15/64}),
 \qquad
 N_wK^{1-\sigma p}(\log K)^{p/2}
 =O(\lambda^{-1/32}(\log\lambda)^{1/8}),
\]
\[
 N_wK^{-p(\sigma-1)}=O(\lambda^{-11/256}).
\]
For each fixed $\beta,R,g$, all explicit $n$-dependent error terms in \eqref{uni:eq-window-master}
are $O_{\beta,R,g}(\lambda^{-1/128})$.
Consequently, for all sufficiently large $n$,
\[
 d_{\mathrm{BL},w_n}
 \le C\Delta^g_{\beta,R}
      +C_{\beta,R,g}\lambda^{-1/128}.
\]
First let $n\to\infty$, then $\beta\to\infty$, and finally $R\to\infty$.
The displayed limit for $\Delta^g_{\beta,R}$ gives convergence to zero.
No quantitative rate is asserted after removing the localization.

The positive multiplier $B_g$ sends the common index-one stable noise with
control $M/\sqrt\pi$ to that with control $B_gM/\sqrt\pi=M_g^+/\sqrt\pi$.
By \eqref{uni:eq-target-copy}, the prelimit window of $M_g^+$ has the same
law as that of $M_g$, and the corresponding conditional stable window has
the same unconditional law as $(\mathcal Z_{g,j})$.  This proves
\eqref{uni:eq-window}.
\end{proof}

\subsection{Transfer of the integral classes}
\label{uni:subsec-integral-proof}
\begin{proof}[Proof of Theorem~\ref{contract:B}]
Fix $g$ and $H$, and put $h_n=a_n/H(\vartheta_n)$. The weights are
admissible for Lemma~\ref{l4t:lem-pathwise-multiplier}, as verified in
Section~\ref{canon:subsec-integral}. Proposition~\ref{canon:prop-integral}
gives $\mathcal L_h(M)=v_H$ almost surely, where $v_H=0$ or $\infty$
according as $\int^\infty dx/H(x)$ converges or diverges.

By \eqref{uni:eq-measure-bridge} and the deterministic multiplier lemma,
$\mathcal L_h(N_+)=v_H$ almost surely. Every Fourier coefficient is a
continuous function of a finite measure in the weak topology, and
$\mathcal L_h$ is a countable limsup of Borel functions. Equality in law
therefore transfers this event to $N_-$. Multiplication by the positive
smooth inverse $F_-^{-1}$ then gives $\mathcal L_h(M_g)=v_H$.
This is a statement about $X^g$ alone and hence holds on its original
space. No new zero--one argument is needed, since the canonical event
already has probability one. The covariance-transfer step above uses neither
Theorem~\ref{contract:A} nor any additional endpoint moment estimate.

Division by $M_g(\T)$ preserves both alternatives. Intersecting the
probability-one events for a prescribed countable family of gauges gives
the final assertion. There is no sample-dependent gauge claim.
\end{proof}

\begin{remark}
All covariance statements are for each fixed admissible $g$, with no
uniform rate over a class of remainders. The approximation in Theorem~\ref{uni:thm-window} is likewise pointwise in $g$ and makes no claim of optimality for the displayed exponent. The choice $H(x)=x^2$
in Theorem~\ref{contract:B} also gives
$\widehat M_g(n)\to0$ almost surely, since $\vartheta_n^2/a_n\to0$.
\end{remark}

\appendix

\section{Fine-field estimates and bounds under the rooted law}
\label{crm:sec-fine}

Let $Y_s(v)$ be the real-line fine field with scale covariance kernel
$k(e^s(v-v'))$, initially zero. The fine-measure estimates below are on $I=[0,1]$;
the rooted estimates for local field increments may use any fixed bounded rescaled window. Write $Z^a$ for its terminal positive derivative measure with constant barrier $a>0$, and $D$ for its untruncated derivative measure.
For a bounded nonnegative function $q(v)$ the notation $Z^{[q]}(g)$ means
\[
 \lim_{s\to\infty}\int_I g(v)(q(v)+\sqrt2s-Y_s(v))
 \1_{\{\inf_{r\le s}(q(v)+\sqrt2r-Y_r(v))>0\}}
 e^{\sqrt2Y_s(v)-s}\dd v.
\]
Zero barrier contributes zero. For bounded $q,g$, domination by a
constant-barrier measure supplies convergence and uniform integrability;
thus
\begin{equation}\label{crm:eq-profilemean}
 \E Z^{[q]}(g)=\int_I q(v)g(v)\dd v.
\end{equation}
Use versions monotone in the barrier and common to all comparisons.
They follow first at finite time, then by terminal limits and a countable
constant-barrier approximation. Deterministic piecewise continuous barrier functions \(q\), including the positive parts below, suffice.

\begin{lemma}
\label{crm:lem-crossing}
For $S_I=\sup_{s\ge0,v\in I}(Y_s(v)-\sqrt2s)$,
\[
 \Prob(S_I\ge a)\le C_k e^{-\sqrt2 a}\qquad(a\ge0).
\]
The same estimate holds on the auxiliary circle, with a different constant.
\end{lemma}
\begin{proof}
Choose $0<\eta<1/4$ with $k(v)\ge1/2$ on $|v|\le\eta$.
For integer $r\ge1$, deterministic anchor $x$, $0\le v\le1$ and
$|w|\le\eta$, Gaussian regression on the entire root path through $r$ gives
\[
 Y_{r+v}(x+e^{-r}w)-Y_r(x)
 =\int_0^r[k(e^{s-r}w)-1]\dd Y_s(x)+G(v,w),
\]
where $G$ is independent of that path. Suppose the root has stayed below
$a+\sqrt2s$ through $r$, and let $R_s=a+\sqrt2s-Y_s(x)$.
With $A_s=1-k(e^{s-r}w)$, monotonicity gives $A_s\ge0$, $\dd A_s\ge0$,
$A_r\le1/2$. Integration by parts bounds the regression mean by
\[
 A_rR_r-A_0a-\int_0^rR_s\dd A_s-\sqrt2\int_0^rA_s\dd s
 \le R_r/2.
\]
Variance contraction under orthogonal projection and smoothness of $k$ give
\[
 \E|G(v,w)-G(v',w')|^2\le C_k(|v-v'|+|w-w'|^2),\qquad G(0,0)=0.
\]
The covering number at canonical distance $\epsilon$ is at most
$C\epsilon^{-3}$. The Gaussian entropy bound gives a uniform expected
supremum, and the Gaussian supremum concentration bound gives
\[
 \Prob(\sup G\ge z/2)\le g(z):=\min(1,C e^{-c z^2}),\quad z\ge0.
\]
This is an estimate for the independent residual, not for a field
conditioned on a global survival event.

Let $\tau_a$ be the first global crossing time. Use a deterministic
$\eta e^{-r}$ mesh of $I$, containing its endpoints, with at most $C_Ie^r$
points. If $\tau_a\in[r,r+1]$, a neighboring deterministic mesh root has
survived through $r$, and its residual supremum is at least $R_r/2$.
Consequently
\[
 \Prob(\tau_a\in[r,r+1])\le C_Ie^r
 \E[\1_{\{\inf_{s\le r}R_s>0\}}g(R_r)].
\]
The root under the original probability has drift $+\sqrt2$. With
\[
 p_s^0(a,z)=(2\pi s)^{-1/2}
 \big(e^{-(a-z)^2/(2s)}-e^{-(a+z)^2/(2s)}\big),
\]
the last expectation is
$e^{-r-\sqrt2a}\int_0^\infty e^{\sqrt2z}p_r^0(a,z)g(z)\dd z$.
For $s\in[r,r+1]$, direct comparison in the reflection formula gives
$p_r^0(a,z)\le2\sqrt2p_s^0(a,z)$. Hence the killed Brownian Green kernel yields
\[
 \sum_{r\ge1}p_r^0(a,z)\le2\sqrt2\int_0^\infty p_s^0(a,z)\dd s
 =4\sqrt2\min(a,z)\le4\sqrt2z.
\]
The mesh factor cancels $e^{-r}$, and summation gives
$C_Ie^{-\sqrt2a}\int_0^\infty z e^{\sqrt2z}g(z)\dd z$.
The initial time interval $[0,1]$ has a Gaussian supremum tail, which is
bounded by $Ce^{-\sqrt2a}$ for $a\ge1$. Extend to $0\le a<1$ by enlarging
$C$. Finally use levels increasing to $a$ to pass from finite-time
crossings to $\{S_I\ge a\}$; no attainment at infinite time is assumed.
The periodic proof is identical with a circular mesh.
\end{proof}

\begin{lemma}
\label{crm:lem-finel2}
For real bounded $g$ supported in $I$ and $a>0$,
\[
 \E|Z^a(g)|^2\le a^2\left|\int_Ig\right|^2
       +C_k e^{\sqrt2a}\|g\|_2^2.
\]
In particular $\E Z^a(I)^2\le C_k'e^{\sqrt2a}$, uniformly in $a>0$.
\end{lemma}
\begin{proof}
In this proof, write
\[
 R_s(x)=a+\sqrt2s-Y_s(x),\qquad
 I_s(x)=\mathbf1_{\{\inf_{0\le r\le s}R_r(x)>0\}},
\]
and let
\[
 Z_s^a(g)=\int_I g(x)R_s(x)I_s(x)
                  e^{\sqrt2Y_s(x)-s}\dd x
\]
be the finite-time approximation of $Z^a(g)$.
At finite time the stochastic coefficient of the positive derivative
density is
$V_s(x)=I_s(x)(\sqrt2R_s(x)-1)e^{\sqrt2Y_s(x)-s}$.
The covariance kernel has $L^1$ norm $e^{-s}\|k\|_1$. Cauchy--Schwarz
and Young's inequality bound the expected bracket density of $Z_s^a(g)$ by
$e^{-s}\|k\|_1\|g\|_2^2\E V_s(0)^2$.
A $2\sqrt2$ exponential tilt and the killed drifted Brownian transition give
\begin{equation}
\label{crm:eq-fine-double-tilt}
 \E V_s(0)^2
 =e^{s+\sqrt2a}
 \int_0^\infty
 (\sqrt2z-1)^2e^{-\sqrt2z}p_s^0(a,z)\dd z.
\end{equation}
Integrate in $s$ and use
$\int_0^\infty p_s^0(a,z)\dd s=2\min(a,z)\le2z$.
The remaining integral is finite and independent of $a$.
Stopped stochastic integrals first justify the calculation; the finite
bound removes stopping and supplies terminal $L^2$ convergence.
Finally $a^2e^{-\sqrt2a}$ is bounded.
\end{proof}

\begin{lemma}
\label{crm:lem-finetail}
There is $C_D<\infty$ such that
$\Prob(D(I)>v)\le C_D/v$ for every $v>0$.
Moreover $0\le Z^{[q]}(g)\le D(g)$ for nonnegative bounded $g,q$.
On $\{S_I<\inf_I q\}$ these terminal quantities are equal.
\end{lemma}
\begin{proof}
For the tail, use Wong's fixed-weight theorem \cite{Wong}, Theorem 1.2,
with $g=1$ on the interval. The fine covariance for $0<|z|\le1$ is
\[
 \int_1^\infty k(vz)\frac{\dd v}{v}
 =\log\frac1{|z|}+c_k-\int_0^{|z|}\frac{k(v)-1}{v}\dd v.
\]
Here
\[
 c_k=\int_0^1\frac{k(v)-1}{v}\dd v.
\]
The regular part is smooth on a neighborhood of $I\times I$; Wong's
Lemma 2.3 verifies the difference-of-continuous-covariances hypothesis.
The interval boundary has zero Lebesgue measure. After the deterministic change from Wong's Seneta--Heyde normalization to the derivative normalization used here, his theorem gives the required \(v^{-1}\) tail bound for all sufficiently large \(v\). Increasing \(C_D\) covers the remaining \(v>0\).
For comparisons, first compare with a constant barrier exceeding
$\|q\|_\infty$, then let that barrier increase. On a no-crossing event,
the only difference in derivative factors is a bounded function times
the additive critical approximation, whose terminal limit is zero.
This proves the equalities for bounded spatial barrier functions \(q\) as well.
\end{proof}

\subsection{Rooted increment estimates and cell estimates}
\label{crm:sec-rooted}
\label{crm:sec-upper}

\begin{lemma}
\label{crm:lem-shape}
For every $A,\lambda<\infty$, with
$H_{b,y}(v)=X_b(y+e^{-b}v)-X_b(y)$ on $[-A,A]$, one has
\[
 \sup_{\beta>0,\,b\ge b_A,\,y\in\T}
 \E_{\mathbb Q_{b,y}^\beta}e^{\lambda\|H_{b,y}\|_\infty}<\infty.
\]
Here $b_A$ is any sufficiently large fixed time that excludes period
images on this window. For $A=1$ one may use $b_A=1$.
\end{lemma}
\begin{proof}
Let $a_s(v)=k(e^{s-b}v)-1$. Orthogonal regression on the entire root path gives
$H(v)=m(v)+G(v)$, where $G$ is independent of that path and
$m(v)=\int_0^b a_s(v)\dd X_s(y)$. Rooting changes only the root path.
Under rooting it has $R_s=|\beta e_1+B_s^{(3)}|$, and integration by parts gives
\[
 m(v)=\sqrt2\int_0^ba_s(v)\dd s-a_0(v)(R_b-\beta)
       +\int_0^b\partial_sa_s(v)(R_s-R_b)\dd s.
\]
Smoothness and evenness give
$|a_s(v)|+|\partial_sa_s(v)|\le C_{k,A}e^{-2(b-s)}$.
Since $|R_s-R_b|\le|B_s-B_b|$,
\[
 \|m\|_\infty\le C+C e^{-2b}|B_b|
       +C\int_0^b e^{-2(b-s)}|B_b-B_s|\dd s.
\]
Time reversal and H\"older reduce the last term to
$I=\int_0^\infty e^{-2v}|B_v|\dd v$. Jensen with probability measure
$e^{-v}\dd v$ gives
$\E e^{\lambda I}\le\int_0^\infty e^{-v}
\E e^{\lambda e^{-v}|B_v|}\dd v<\infty$, since $ve^{-2v}$ is bounded.
The endpoint term has a uniformly bounded Gaussian scale.

For the Gaussian residual, $G(0)=0$ and variance contraction under projection gives
$\sup_v\E|G'(v)|^2\le-k''(0)/2$.
The fundamental theorem of calculus bounds $\|G\|_\infty$ by a constant
times $\|G'\|_{L^2[-A,A]}$.
Jensen on that finite interval bounds the exponential of a sufficiently
small multiple of $\|G'\|_2^2$ by Gaussian one-dimensional exponential
squares. The inequality $\lambda r\le cr^2+\lambda^2/(4c)$ supplies
all fixed linear exponential moments. Combine the two parts.
\end{proof}

For cell $J_j$ write
\[
 H_j(v)=X_b(y_j+\rho v)-X_b(y_j),\quad
 a_j=e^{\sqrt2X_b(y_j)-2b},\quad
 M_j=e^{\sqrt2\sup_{[0,1]}H_j},\quad c_j=La_j.
\]
The supremum includes $v=0$, so $M_j\ge1$. Let
$\mathcal A_n=\sum_jc_jM_j$.

\subsection*{Proof of Proposition~\ref{crm:prop-budget}}

\begin{proof}
Conditionally on $\F_b$, the rescaled future is the fixed fine field. For each $j$, let $D_j$ be the untruncated derivative measure
of the rescaled future field
\[
 \bigl(X_{b+s}(y_j+\rho v)-X_b(y_j+\rho v)\bigr)
       _{s\ge0,\,v\in[0,1]},
\]
constructed from this same future field.
Its exact exponential and spatial Jacobian give the prefactor $a_j$.
The starting spatial barrier is $r_j-H_j(v)$, where $r_j=R_b(y_j)$;
coarsely killed points contribute zero. Positive-barrier comparison gives
\[
 \mu(J_j)\le a_jD_j(e^{\sqrt2H_j}\1_{[0,d/\rho]})
            \le a_jM_jD_j([0,1]).
\]
Each $D_j$ has the fixed fine law and is independent of $\F_b$; distinct
$D_j$ are not claimed independent. Lemma~\ref{crm:lem-finetail} gives
the one-cell weak-$L^1$ bound, and summing proves the second estimate in
\eqref{paper:eq-budget}.

The root identity and survival at the deterministic anchor imply
\[
 \E[\1_{E_{\beta,b}}a_j e^{C\|H_j\|_\infty}]
 \le\beta e^{-b}\E_{\mathbb Q_{b,y_j}^\beta}
           [R_b^{-1}e^{C\|H_j\|_\infty}].
\]
The BES density is $(r/\beta)p_b^0(\beta,r)$, and
$1-e^{-2\beta r/b}\le2\beta r/b$ gives
$\E_{\mathbb Q}R_b^{-2}\le2/b$.
Cauchy--Schwarz, Lemma~\ref{crm:lem-shape}, $m\le Ce^b$, and
$L/\sqrt b=O(1)$ prove \eqref{crm:eq-general-envelope}.
Finally use
$\E[U^2\1_{U\le\delta}\mid\F_b]
\le2\int_0^\delta v\Prob(U>v\mid\F_b)\dd v$.
\end{proof}

\begin{lemma}
\label{crm:lem-bracket}
For deterministic $f\in L^2(\T;\mathbb C)$ and $s\ge1$,
\[
 \E\frac{\dd}{\dd s}\langle Z^\beta(f),\overline{Z^\beta(f)}\rangle_s
 \le C_{\beta,k}\|f\|_2^2s^{-3/2}.
\]
Hence for $1\le t<b$,
$\E|Z_b^\beta(f)-Z_t^\beta(f)|^2
 \le C_{\beta,k}\|f\|_2^2(b-t)t^{-3/2}$.
\end{lemma}
\begin{proof}
The stochastic coefficient is the same $V_s$ as in
Lemma~\ref{crm:lem-finel2}, now periodic. The kernel $q_{e^s}$ has $L^1$
norm $e^{-s}\|k\|_1$. Young and Cauchy--Schwarz give the bracket bound
$e^{-s}\|k\|_1\|f\|_2^2\E V_s(0)^2$.
Use \eqref{crm:eq-fine-double-tilt}, together with
$p_s^0(\beta,r)\le2\beta r/(\sqrt{2\pi}s^{3/2})$.
The remaining exponentially weighted polynomial integral is finite.
The fine second-moment argument also proves global square integrability
and validates the stochastic isometry. No sign of $V_s(x)V_s(y)$ is assumed.
\end{proof}

\begin{lemma}
\label{crm:lem-linear}
For fixed $f\in C^1(\T)$ and the scales \eqref{crm:eq-scales},
\[
 \E[\1_{E_\beta}|L\widehat{fZ_t^\beta}(n)|]
 \le C_{\beta,k}\|f\|_{C^1}u^{-5},
\]
\[
 \E|L(\widehat{fZ_b^\beta}(n)-\widehat{fZ_t^\beta}(n))|^2
 \le C_{\beta,k}\|f\|_2^2\frac{\log u}{\sqrt u}.
\]
Both conclusions remain negligible after multiplication by any fixed
power of $\log u$. The same holds for each fixed nonzero integer multiple
of $n$, and for a fixed integer frequency offset.
\end{lemma}
\begin{proof}
On $E_\beta$ the early density is $R_tw_t$, smooth and periodic, with derivative
$(\sqrt2R_t-1)w_t\partial_xX_t$. Evenness gives independence of
$X_t(x)$ and $\partial_xX_t(x)$; the latter has standard deviation at most
$Ce^t$. The exponential tilt by $w_t$ makes $R_t$ have law
$\beta-N(0,t)$. Therefore the expected variation of
$fR_tw_t$ is at most
$C\|f\|_{C^1}(1+\beta+\sqrt t)(1+e^t)$, without needing a survival
indicator on this upper bound. Integration by parts gives the first estimate,
since $e^t/n=u^{-6}$. Lemma~\ref{crm:lem-bracket} gives the second,
as $L^2=u$ and $b-t=10\log u$. Fixed multiples and offsets only change
constants in integration by parts; their $L^2$ norms remain bounded.
\end{proof}

\section{An independent endpoint moment bound}
\label{crm:sec-endpoint-closure}

We prove the fractional-moment estimate without using a limiting
Fourier law. Use the scales in \eqref{crm:eq-scales}, and set
$h=\beta+\log L/\sqrt2$. The lowered fine barrier retains the
conditional first moment responsible for the truncated center.
Write $\ell=d$ for the cell length and $\mu=Z^\beta$, and recall that
$U_j=L\mu(J_j)$.
\begin{lemma}
\label{crm:lem-cell-shift}
For a cell anchored at $y$, put $\alpha=e^b\ell\le1$ and
\[
 H(v)=X_b(y+e^{-b}v)-X_b(y),\qquad r=R_b(y),
\]
\[
 g(v)=e^{\sqrt2H(v)},\qquad c=L e^{\sqrt2\beta-\sqrt2r}=e^{\sqrt2(h-r)}.
\]
On the coarse event $E_{\beta,b}$, let $Z^{[q]}$ be the positive
terminal fine chaos with spatial starting barrier $q(v)=r-H(v)>0$,
and set
\[
 Z=Z^{[q]}(g\1_{[0,\alpha]}),\quad
 V=Z^{[(q-h)_+]}(g\1_{[0,\alpha]}),\quad \Delta=Z-V.
\]
A zero starting barrier denotes the identically zero killed process.
The fine field is independent of $\F_b$, and $U=cZ$. With fine
expectation denoted temporarily by $\E_{\mathrm f}$,
\begin{align}
 \E_{\mathrm f}\Delta
 &=\int_0^\alpha g(v)\min(q(v),h)\dd v,
 \label{crm:eq-shift-mean}\\
 \left|\E[U\1_{\{U>1\}}\mid\F_b]
       -c\E_{\mathrm f}\Delta\right|
 &\le C_k c\exp\{2\sqrt2\|H\|_\infty\}.
 \label{crm:eq-cell-shift-error}
\end{align}
The norm is on $[0,1]$. The estimate is uniform in \(r\) and \(H\) subject to \(q(v)=r-H(v)>0\); no compact restriction on the endpoint
\(r\) or on the function \(H\) is imposed.
\end{lemma}

\begin{proof}
The coarse/fine change of variables gives the factor $c$ exactly.
Positive spatial-barrier martingales are monotone in their starting
barrier, have terminal expectation equal to that barrier, and are
bounded by a constant-barrier martingale on this fixed interval.
These facts also apply to a bounded nonnegative barrier that can vanish:
use the zero killed process at its zeros, or approximate from above
and use the common constant-barrier $L^2$ bound. Thus
$0\le V\le Z$ and \eqref{crm:eq-shift-mean} follows by subtraction.

We use the following fine-field estimates. Put
$M=e^{\sqrt2 H_{\max}}$, where $H_{\max}=\max H$ and
$H_{\min}=\min H$.
By Lemma~\ref{crm:lem-finetail}, the fixed untruncated unit-cell chaos $D$ satisfies
$\Prob(D([0,1])>x)\le C_D/x$. Monotonicity gives
\begin{equation}\label{crm:eq-shift-original-tail}
 \Prob_{\mathrm f}(cZ>1)\le C_DcM.
\end{equation}
The sharp all-time barrier estimate of Lemma~\ref{crm:lem-crossing} is
\[
 \Prob_{\mathrm f}(S\ge a)\le C_k e^{-\sqrt2a},\qquad
 S=\sup_{s\ge0,\,v\in[0,1]}(Y_s(v)-\sqrt2s).
\]
Its established statement for $a\ge1$ extends to all real $a$ by
increasing $C_k$ and using the bound one for probabilities.
If $r-h-H_{\max}>0$ and $S<r-h-H_{\max}$, neither barrier is ever
hit. On that event the difference of the finite-time densities is
$h g(v)e^{\sqrt2Y_s(v)-s}$. Its total mass tends to zero, since
critical additive chaos tends to zero and $g$ is a fixed bounded
function after conditioning. Hence $\Delta=0$ on this event.
If $r-h-H_{\max}\le0$, the trivial probability bound suffices.
In both cases,
\begin{equation}\label{crm:eq-shift-support}
 \Prob_{\mathrm f}(\Delta>0)
 \le C_k e^{-\sqrt2(r-h-H_{\max})}=C_kcM.
\end{equation}
This support statement concerns terminal measures; it does not assert
that their finite-time difference vanishes before the crossing.

Let $p_0=(r-h-H_{\min})_+$. If $p_0=0$, then $V=0$.
Otherwise spatial monotonicity and the explicit constant-barrier
second moment give
\begin{equation}\label{crm:eq-shift-second}
 \E_{\mathrm f}V^2
 \le M^2\{p_0^2+6e^{\sqrt2p_0}\}
 \le C M^2e^{\sqrt2p_0}.
\end{equation}
The harmless numerical constant can be enlarged for the chosen kernel
normalization. There is no polynomial barrier loss in the exponential
term. Put $A=\{cZ>1\}$. The exact identity
\[
 \E_{\mathrm f}[cZ\1_A]-c\E_{\mathrm f}\Delta
 =c\E_{\mathrm f}[V\1_A]
   -c\E_{\mathrm f}[\Delta\1_{A^c}]
\]
retains both truncation corrections. By \eqref{crm:eq-shift-support},
\[
 c\E_{\mathrm f}[\Delta\1_{A^c}]
 \le\Prob_{\mathrm f}(\Delta>0)\le C_kcM.
\]
If $p_0>0$, Cauchy--Schwarz, \eqref{crm:eq-shift-original-tail}, and
\eqref{crm:eq-shift-second} imply
\[
 c\E_{\mathrm f}[V\1_A]
 \le C_k c\{c M^3e^{\sqrt2p_0}\}^{1/2}
 =C_kcM^{3/2}e^{-H_{\min}/\sqrt2}.
\]
Here $ce^{\sqrt2p_0}=e^{-\sqrt2H_{\min}}$ follows from
$c=e^{\sqrt2(h-r)}$. Since $H(0)=0$, we have
$H_{\min}\le0\le H_{\max}$, and the two displayed errors are bounded
by the right side of \eqref{crm:eq-cell-shift-error}.
\end{proof}

\begin{lemma}
\label{crm:lem-additive-negligible}
For every deterministic $f\in C^1(\T)$,
\begin{equation}\label{crm:eq-additive-negligible}
 \1_{E_\beta}hL\int_\T f(x)e^{\ii nx}w_b(x)\dd x
 \longrightarrow0
 \quad\hbox{in probability}.
\end{equation}
For each $0<p<1$, its localized $p$th absolute moment tends to zero.
\end{lemma}

\begin{proof}
On $E_\beta$ the two fixed barriers $\beta$ and $\beta+1$ survive
through time $b$, so, as finite-time measures,
\[
 Z_b^{\beta+1}(\dd x)-Z_b^\beta(\dd x)=w_b(x)\dd x.
\]
For either fixed barrier $\gamma\in\{\beta,\beta+1\}$, apply
Lemma~\ref{crm:lem-bracket} to the deterministic function
$f(x)e^{\ii nx}$. It gives
\[
 \E\left|hL\left(Z_b^\gamma(fe^{\ii n\cdot})
             -Z_t^\gamma(fe^{\ii n\cdot})\right)\right|^2
 \le C_{\beta,k}h^2L^2\|f\|_2^2(b-t)t^{-3/2}
 \le C_{\beta,k}\|f\|_2^2
       \frac{(\log u)^3}{\sqrt u}\longrightarrow0.
\]
Integration by parts as in Lemma~\ref{crm:lem-linear}, applied on
$E_\gamma$ to the smooth density $R_t^\gamma w_t$, uses the bound
\[
 \E\|(fR_t^\gamma w_t)'\|_1
 \le C_k\bigl(\|f'\|_\infty(\gamma+\sqrt t)
       +\|f\|_\infty e^t(1+\gamma+\sqrt t)\bigr).
\]
Consequently
\[
 \E[\1_{E_\gamma}L|Z_t^\gamma(fe^{\ii n\cdot})|]
 \le C_{\beta,k}\|f\|_{C^1}u^{-5}
\]
for large $n$. Multiplication by $h=O_\beta(\log u)$ still tends
to zero. Since $E_\beta\subset E_{\beta+1}$, subtraction proves
\eqref{crm:eq-additive-negligible}. The $p$th moment assertion follows
from Jensen applied to the early $L^1$ and band $L^2$ bounds.
\end{proof}

\begin{proposition}
\label{crm:prop-compensator-expansion}
For deterministic $f\in C^1(\T)$ define
\[
 K_n(f)=\sum_j f(y_j)e^{\ii ny_j}
      \E[U_j\1_{\{U_j>1\}}\mid\F_b].
\]
There are $\F_b$-measurable errors $\mathcal E_n(f)$ and
$\mathcal B_n(f)$ such that on $E_{\beta,b}$,
\begin{equation}\label{crm:eq-compensator-expansion}
 K_n(f)=hL\sum_j f(y_j)e^{\ii ny_j}\int_{J_j}w_b(x)\dd x
          +\mathcal E_n(f)-\mathcal B_n(f),
\end{equation}
where
\begin{align}
 \sup_n\E[\1_{E_{\beta,b}}|\mathcal E_n(f)|]
 &\le C_{\beta,k}\|f\|_\infty,
 \label{crm:eq-compensator-error}\\
 \E[\1_{E_{\beta,b}}|\mathcal B_n(f)|]
 &\le C_{\beta,k}\|f\|_\infty\frac{h^3}{u}
 \longrightarrow0.
 \label{crm:eq-low-radial-error}
\end{align}
In particular $K_n(f)$ is tight on $E_\beta$, and its localized
$p$th moments are uniformly bounded for every $0<p<1$.
\end{proposition}

\begin{proof}
Sum Lemma~\ref{crm:lem-cell-shift} with coefficients
$f(y_j)e^{\ii ny_j}$. The exact mean in
\eqref{crm:eq-shift-mean}, and the coarse/fine Jacobian, give
\[
 c_j\E_{\mathrm f}\Delta_j
 =L\int_{J_j}w_b(x)\min(R_b(x),h)\dd x.
\]
Thus \eqref{crm:eq-compensator-expansion} holds with
\[
 \mathcal B_n(f)=L\sum_j f(y_j)e^{\ii ny_j}
                   \int_{J_j}w_b(x)(h-R_b(x))_+\dd x.
\]
For the sum of one-cell errors, Lemma~\ref{crm:lem-shape}, together with the rooting estimate used in
Proposition~\ref{crm:prop-budget}, gives
\[
 \E\left[\1_{E_{\beta,b}}
     L\sum_j e^{\sqrt2X_b(y_j)-2b}
                 e^{2\sqrt2\|H_j\|_\infty}\right]
 \le C_{\beta,k}L b^{-1/2}\le C_{\beta,k}.
\]
For precision, finite rooting bounds each summand's expectation by
\[\beta e^{-b}\E_{\mathbb Q_{b,y_j}^\beta}
[e^{2\sqrt2\|H_j\|_\infty}/R_b(y_j)],\]
Cauchy--Schwarz, Lemma~\ref{crm:lem-shape}, and
\(\E_{\mathbb Q}R_b^{-2}\le2/b\)
give the preceding bound.

Let $p_b^0(\beta,r)$ be the killed Brownian transition density.
At any deterministic point the exponential tilt gives
\[
 \E[I_b(x)w_b(x)\phi(R_b(x))]
   =\int_0^\infty\phi(r)p_b^0(\beta,r)\dd r,
 \qquad
 p_b^0(\beta,r)\le\frac{2\beta r}{\sqrt{2\pi}b^{3/2}}.
\]
Since $E_{\beta,b}$ implies survival of this one root,
\[
 \E[\1_{E_{\beta,b}}|\mathcal B_n(f)|]
 \le 2\pi L\|f\|_\infty
             \int_0^h(h-r)p_b^0(\beta,r)\dd r
 \le C\beta\|f\|_\infty Lh^3b^{-3/2}.
\]
This is \eqref{crm:eq-low-radial-error} because $L=\sqrt u$ and
$b\asymp u$.

The mesh-to-integral error for the first term in
\eqref{crm:eq-compensator-expansion} is at most
\[
 hL\ell(n\|f\|_\infty+\|f'\|_\infty)W_b(\T),\qquad
 W_b(\T)=\int_\T w_b(x)\dd x.
\]
It tends to zero in $L^1$ because $\E W_b(\T)=2\pi$ and
$Lne^{-b}=u^{-7/2}$. Lemma~\ref{crm:lem-additive-negligible}
settles the resulting Fourier integral.
Together with
\eqref{crm:eq-compensator-error},
\eqref{crm:eq-low-radial-error}, and
Lemma~\ref{crm:lem-additive-negligible},
these estimates prove the asserted tightness and localized \(p\)th moment bounds.
All conditional fine statements were made
given $\F_b$; $E_{\beta,b}$ was used only as a coarse indicator.
\end{proof}

\subsection*{Proof of Theorem~\ref{crm:thm-endpoint-moment}}

\begin{proof}
Put $a_j=f(y_j)e^{\ii ny_j}$, so that
$T_{n,f}=\sum_j a_jU_j$.
The phase and amplitude replacement obeys
\begin{equation}
\label{crm:eq-endpoint-mesh-error}
 \left|
 L\int_\T f(x)e^{\ii nx}\mu(\dd x)-T_{n,f}
 \right|
 \le
 L\ell
 \bigl(n\|f\|_\infty+\|f'\|_\infty\bigr)\mu(\T).
\end{equation}
Its $L^1$ norm is bounded by $C_\beta\|f\|_{C^1}u^{-7/2}$.
The linear conditional mean
$P_n(f)=\E[T_{n,f}\mid\F_b]$ differs from
$LZ_b^\beta(fe^{\ii n\cdot})$ by the same deterministic coefficient
times $Z_b^\beta(\T)$. Its localized $p$th moment is bounded, and
indeed tends to zero, by the early/band estimates used above.
The conditional mean of the marks below the unit cutoff is $D_n(f)=P_n(f)-K_n(f)$.
Proposition~\ref{crm:prop-compensator-expansion} bounds its localized
$p$th moment by $C_{\beta,k,p}\|f\|_{C^1}^p$.

It remains to bound the stochastic remainder. For
\[
 S_n(f)=\sum_j a_j\left(U_j\1_{\{U_j\le1\}}
       -\E[U_j\1_{\{U_j\le1\}}\mid\F_b]\right),
\]
conditional independence within each of the three colors and the
conditional second-moment estimate in
Proposition~\ref{crm:prop-budget} give
\[
 \E[|S_n(f)|^2\mid\F_b]
 \le C_k\|f\|_\infty^2\mathcal A_n,
 \qquad
 \sup_n\E[\1_{E_{\beta,b}}\mathcal A_n]<\infty.
\]
Using $E_\beta\subset E_{\beta,b}$ before taking expectations, then
Jensen, yields the required localized $p$th moment bound for $S_n(f)$.
For the uncentered large jumps, subadditivity for $0<p<1$ gives
\[
 \left|\sum_j a_jU_j\1_{\{U_j>1\}}\right|^p
 \le\|f\|_\infty^p\sum_jU_j^p\1_{\{U_j>1\}}.
\]
The conditional tail bound in \eqref{paper:eq-budget} implies,
by the scalar layer-cake identity,
\begin{align*}
 \sum_j\E[U_j^p\1_{\{U_j>1\}}\mid\F_b]
 &=\sum_j\Prob(U_j>1\mid\F_b)
   +p\int_1^\infty v^{p-1}
                  \sum_j\Prob(U_j>v\mid\F_b)\dd v\\
 &\le \frac{C_D}{1-p}\mathcal A_n.
\end{align*}
This gives the desired localized moment bound after conditioning on
$\F_b$ with the coarse indicator $\1_{E_{\beta,b}}$.
Finally $T_{n,f}=D_n(f)+S_n(f)+\sum_j a_jU_j\1_{\{U_j>1\}}$.
Subadditivity, \eqref{crm:eq-endpoint-mesh-error}, and the preceding estimates prove
\eqref{crm:eq-endpoint-pmoment} for large $n$. The finitely many
remaining integers are absorbed using
$\E\mu(\T)=2\pi\beta$ and Jensen. Homogeneity in $f$ gives the stated
norm dependence.
\end{proof}

\section{Barrier comparison and removal of the early-scale increment}
\label{crm:sec-response-toolkit}

We first compare the actual finite-barrier cell mass with the corresponding untruncated fine-field mass. We then show that the contribution of scales up to time \(t\) can be removed from the local spatial increment with a summable \(L^1\) error.

At the late time \(b\) from \eqref{crm:eq-scales}, use deterministic cells
$J_j=[y_j,y_j+d)$, where $0<d\le e^{-b}$, and put
\[
 \theta=e^bd,\quad
 r_j=\beta+\sqrt2b-X_b(y_j),\quad
 H_j(u)=X_b(y_j+e^{-b}u)-X_b(y_j),\quad 0\le u\le1,
\]
\[
 g_j(u)=e^{\sqrt2H_j(u)}\mathbf1_{[0,\theta]}(u),\quad
 M_j=e^{\sqrt2\sup_{[0,1]}H_j},
\]
\[
 a_j=e^{\sqrt2X_b(y_j)-2b}=e^{\sqrt2\beta-\sqrt2r_j},\quad
 c_j=L\,a_j.
\]
The symbol $\mathbb E_D$ below means expectation in an independent fine
field, with the displayed coarse parameters held fixed.

\subsection*{Proof of Proposition~\ref{crm:prop-cap-barrier-removal}}

\begin{proof}
The exact coarse/fine scaling identity
\eqref{paper:eq-cell-scaling} gives, on \(E_{\beta,b}\),
\[
 \mu(J_j)=a_j Z^{[r_j-H_j]}(g_j)
\]
as a conditional fine-field identity. The spatial starting barrier is
positive on the entire unit window, since $E_{\beta,b}$ imposes coarse
survival at every spatial point. Lemma~\ref{crm:lem-finetail} gives
\(Z^{[r_j-H_j]}(g_j)\le D(g_j)\), and the two terminal masses agree on
\(\{S_I<r_j-\sup_{[0,1]}H_j\}\). Since the cap takes values in $[0,1]$, the difference between the two capped expectations is at most the crossing probability. Lemma~\ref{crm:lem-crossing} therefore gives
\[
 \mathbb P_D(S_I\ge r_j-\sup H_j)
 \le C_k e^{-\sqrt2(r_j-\sup H_j)}
 =C_k e^{-\sqrt2\beta}a_jM_j.
\]
This proves the one-cell estimate. Summation gives
\eqref{crm:eq-summed-cap-barrier-removal}; \eqref{paper:eq-budget}
bounds the expectation of
$\mathbf1_{E_{\beta,b}}\mathcal A_n$ uniformly in $n$.
\end{proof}

For completeness, the exponential-moment bound lemma provides exactly the
integrability needed for random weights in
Lemma~\ref{crm:lem-capped-first-moment}. For every fixed $\lambda\ge0$,
\begin{equation}
\label{crm:eq-general-profile-budget}
 \sup_n\mathbb E\!\left[
 \mathbf1_{E_{\beta,b}}L
 \sum_j a_j e^{\lambda\|H_j\|_\infty}\right]<\infty.
\end{equation}
Indeed, the finite derivative-root identity bounds each expected term by
\[
 e^{-b}\beta\,
 \mathbb E_{\mathbb Q_{b,y_j}^\beta}
 \frac{e^{\lambda\|H_j\|_\infty}}{R_{b}(y_j)}.
\]
By Cauchy--Schwarz, $\mathbb E_{\mathbb Q}R_b^{-2}\le2/b$,
and Lemma~\ref{crm:lem-shape}, this is bounded by
$C_{\beta,k,\lambda}e^{-b}/\sqrt b$.
There are $O(e^{b})$ cells and $b\asymp u$. Since
$\int g_j^2\le M_j^2$ and $(\int g_j)^2\le M_j^2$, the remainder in \eqref{crm:eq-capped-remainder} is summable over the cells with a uniformly bounded localized \(L^1\) norm.
The estimate \eqref{crm:eq-general-profile-budget} supplies the integrability required for the random-weight error terms following \eqref{l3c:eq-cap-identity}.

\subsection{Removal of the early-scale spatial increment}
\label{l1:sec-old-profile}

We use $\kappa=1/\sqrt2$ and the all-integer scales
$u=\log n$, $L=\sqrt u$, $b=u+4\log u$, $t=u-6\log u$.
Fix $\beta>0$ and put $\gamma=\beta+1$.
The localization in this section is the single-anchor indicator
$I_b^\gamma(y)$; it equals one on the original global event $E_\beta$.
Unlike multiplication by $\mathbf1_{E_\beta}$, it preserves spatial
locality of the random functional.

For a bounded nonnegative fine weight $g$ on $[0,1]$, define
\[
 B_g^+(T)=\mathbb E\min\{D(g),T\}
            -\kappa\left(\int g\right)\log_+T,\qquad T>0.
\]
For a continuous function $H$ on $[0,1]$ and $0<\theta\le1$, put
\[
 g_{H,\theta}(v)=e^{\sqrt2H(v)}\mathbf1_{[0,\theta]}(v),\qquad
 \Psi_\theta(H,T)=
 \int_0^\theta H(v)e^{\sqrt2H(v)}\,dv+B_{g_{H,\theta}}^+(T).
\]
Lemma~\ref{crm:lem-capped-first-moment}, together with
\(\left|\int_0^\theta He^{\sqrt2H}\right|
 \le\|H\|_\infty e^{\sqrt2\|H\|_\infty}\), 
gives
\begin{equation}\label{l1:eq-Psi-envelope}
 |\Psi_\theta(H,T)|\le C_k e^{C_k\|H\|_\infty},
 \qquad T>0,\quad0<\theta\le1.
\end{equation}
All constants below are independent of $T$ and $\theta$.

\begin{lemma}
\label{l1:lem-finitepart-continuity}
If $g,g'$ have the same support, $g\ge0$, and
$|g'-g|\le\delta g$ with $0\le\delta\le1/2$, then
\begin{equation}\label{l1:eq-relative-weight}
 |B_{g'}^+(T)-B_g^+(T)|
 \le C_k\delta(1+\|g\|_\infty^2),\qquad T>0.
\end{equation}
Consequently, for arbitrary continuous $H,H'$ on $[0,1]$,
\begin{equation}\label{l1:eq-Psi-profile-continuity}
 |\Psi_\theta(H,T)-\Psi_\theta(H',T)|
 \le C_k\min\{\|H-H'\|_\infty,1\}
       e^{C_k(\|H\|_\infty+\|H'\|_\infty)}.
\end{equation}
\end{lemma}

\begin{proof}
For positive $g$, this is the weight-concavity argument in
Lemma~\ref{l1:lem-response-lipschitz}: write
$g'=g+\delta v$ with $|v|\le g$, use the bounded concave function
$s\mapsto B_{g+sv}^+(T)$ on $[-1,1]$, and compare its secant slopes.
The same argument allows vanishing $g$, because $v=0$ wherever $g=0$.
This proves \eqref{l1:eq-relative-weight}. Apply it to the exponential
weights when $\|H-H'\|_\infty$ is small, and to the individual
envelopes otherwise. The mean-value bound for $xe^{\sqrt2x}$ gives
\eqref{l1:eq-Psi-profile-continuity}.
\end{proof}

\begin{lemma}
\label{l1:lem-scaled-old-profile}
Let $1\le t\le b$, $\rho=e^{-b}$, and $\sigma=e^{t-b}\le1$.
At a deterministic anchor $y$, define
\[
 H^{\rm old}(v)=X_t(y+\rho v)-X_t(y),\qquad
 H(v)=X_b(y+\rho v)-X_b(y),\qquad
 H^{\rm new}=H-H^{\rm old}.
\]
Let $\widetilde{\mathbb P}_{b,y}$ be the Gaussian exponential tilt
with density $w_b(y)=e^{\sqrt2X_b(y)-b}$ relative to $\mathbb P$.
For every fixed $\lambda<\infty$,
\begin{equation}\label{l1:eq-old-profile-exponential}
 \sup_{b\ge t\ge1,\,y}
 \widetilde{\mathbb E}_{b,y}
 \exp\{\lambda\sigma^{-1}\|H^{\rm old}\|_\infty\}<\infty.
\end{equation}
The full increment \(H\) and the new increment \(H^{\rm new}\) have uniformly bounded exponential moments
of their sup norms under this same tilt.
\end{lemma}

\begin{proof}
The tilt changes Gaussian means and preserves covariance. The mean of \(H^{\rm old}\) is exactly
\[
 m_{\rm old}(v)=\sqrt2\int_0^t[k(e^{s-b}v)-1]\,ds,
\]
so evenness and smoothness give
$\|m_{\rm old}\|_\infty\le C_k\sigma^2$.
For the centered part \(G\) of \(H^{\rm old}\), $G(0)=0$ and
\[
 \sup_{v\in[0,1]}\mathbb E|G'(v)|^2
 \le-k''(0)e^{-2b}\int_0^te^{2s}\,ds\le C_k\sigma^2.
\]
The fundamental theorem of calculus gives
$\|G\|_\infty\le\|G'\|_{L^2[0,1]}$.
Jensen bounds the exponential of a small constant times
$\sigma^{-2}\|G'\|_2^2$ by the corresponding one-dimensional Gaussian
exponential squares. A linear exponential is bounded by a Gaussian
square exponential times a deterministic constant. This proves
\eqref{l1:eq-old-profile-exponential}.
Applying the same calculation to the scale intervals \([0,b]\) and \([t,b]\) gives uniformly bounded exponential moments of \(\|H\|_\infty\) and \(\|H^{\rm new}\|_\infty\), respectively.
All these calculations use the same one-point tilt; no independence
between these increments is asserted or needed.
\end{proof}

\subsection*{Proof of Lemma~\ref{l1:lem-profile-removal}}

\begin{proof}
Write $\Delta_j=\Psi_\theta(H_j,T_j)
                     -\Psi_\theta(H_j^{\rm new},T_j)$.
Its threshold $T_j$ is random, but
\eqref{l1:eq-Psi-profile-continuity} is uniform in that threshold.
By Cauchy--Schwarz, \eqref{l1:eq-Psi-profile-continuity} and
Lemma~\ref{l1:lem-scaled-old-profile} give
\[
 \widetilde{\mathbb E}_{b,y_j}|\Delta_j|
 \le C_k e^{t-b}.
\]

Drop the survival indicator only in the nonnegative absolute error.
Since $c_j=Le^{-b}w_b(y_j)$, the one-point exponential tilt gives
\[
 \mathbb E[I_b^\gamma(y_j)c_j|\Delta_j|]
 \le Le^{-b}\widetilde{\mathbb E}_{b,y_j}|\Delta_j|
 \le C_kLe^{-b}e^{t-b}.
\]
Sum over $O(e^b)$ anchors and use $|\zeta_j|\le A$.
Conditional expectation is an $L^1$ contraction, proving the final
assertion. On $E_\beta$ every anchor indicator is one, so this locally
measurable replacement agrees there with the original family of cell terms.
\end{proof}

\section{Fluctuation estimates for local functionals}
\label{lev1:sec-late-fluctuation}

Theorem~\ref{lev1:thm-late-decoration} applies to any local functional $A_j=A_j(c_j,H_j)$ satisfying the exponential bound \eqref{lev1:eq-decoration-envelope}.
Its proof uses conditional
covariance and one-point rooting.
In particular, it applies to the local functional $J_H+B_g(1/c)$.

\begin{lemma}
\label{lev1:lem-fixed-window-shape}
For every fixed $A,\lambda>0$ there is $C_{k,A,\lambda}<\infty$ such that
\[
 \sup_{\beta>0,\,s\ge1,\,x\in\T}
 \E_{\mathbb Q_{s,x}^{\beta}}
 \exp\!\left\{\lambda\sup_{|v|\le A}
 |X_s(x+e^{-s}v)-X_s(x)|\right\}\le C_{k,A,\lambda},
\]
provided the local window is interpreted in the circle and $s$ is
sufficiently large in terms of $A$; bounded remaining times can be
absorbed into the constant. For $1<q<3$,
\begin{equation}\label{lev1:eq-root-weighted-shape}
 \E_{\mathbb Q_{s,x}^{\beta}}
 \left[\frac{e^{-\sqrt2R_s(x)}}{R_s(x)}
          e^{\lambda\|H_{s,x}\|_{[-A,A]}}\right]
 \le C_{k,A,\lambda,q}s^{-3/(2q)}.
\end{equation}
The constants in these assertions are independent of $\beta>0$.
\end{lemma}

\begin{proof}
The first assertion is Lemma~\ref{crm:lem-shape}, which already
allows any fixed rescaled window $[-A,A]$.

The BES$(3)$ transition density obeys
\[
 p_s^{(3)}(\beta,r)=\frac r\beta p_s^0(\beta,r)
 \le\frac{2r^2}{\sqrt{2\pi}s^{3/2}},\qquad r>0.
\]
Consequently, for $1<q<3$,
\[
 \E_{\mathbb Q}\frac{e^{-q\sqrt2R_s}}{R_s^q}
 \le Cs^{-3/2}\int_0^\infty r^{2-q}e^{-q\sqrt2r}\dd r
 \le C_qs^{-3/2}.
\]
H\"older's inequality with conjugate exponents $q$ and $q/(q-1)$, together with the exponential-moment estimate in Lemma~\ref{crm:lem-shape}, proves \eqref{lev1:eq-root-weighted-shape}.
\end{proof}

\subsection*{Proof of Theorem~\ref{lev1:thm-late-decoration}}

\begin{proof}
Anchor rooting and Cauchy--Schwarz first give
\[
 \sum_j\E|Y_j|
 \le C_{k,\gamma}C_0\beta Lb^{-1/2}.
\]
Indeed $I_b^\beta(y_j)w_b(y_j)\dd\Prob
=\beta R_b(y_j)^{-1}\dd\mathbb Q_{b,y_j}^\beta$,
$\E_{\mathbb Q}R_b^{-2}\le2/b$, and there are $O(e^b)$ cells.
Thus the conditional mean is well-defined on the whole probability
space.

\emph{The two-point estimate.}
Take two anchors $x,y$ at circular distance $d\le3e^{-t}$.
For $d>0$ put $s=\min(b,\log(1/d))$; on the diagonal put $s=b$.
Then $s\ge t-\log3$. Choose a local lift so that both anchors and
both $\rho$-windows lie in a fixed interval of radius at most
$4e^{-s}$ about $x$. We claim
\begin{equation}\label{lev1:eq-pair-weight}
 \E\left[I_b^\beta(x)I_b^\beta(y)w_b(x)w_b(y)
 e^{\gamma(\|H_{b,x}\|_\infty+\|H_{b,y}\|_\infty)}\right]
 \le C_{k,\gamma,\eta}\beta e^{\sqrt2\beta}
                 e^ss^{-\eta}.
\end{equation}

If $s<b$, then $d=e^{-s}$. Since $k$ is supported in $[-1,1]$,
the root increments at $x$ and $y$ after $s$ have zero covariance.
The product of their future exponential martingales therefore has
conditional expectation one. Integrating against this product is an
ordinary Gaussian exponential tilt of the future field.
Under this tilt, the contributions of scales in $[s,b]$ to
$H_{b,x}$ and $H_{b,y}$ remain Gaussian, with uniformly bounded covariance and uniformly bounded mean. The contribution to the mean coming from the tilt at the same anchor is bounded by
\[
 C\int_s^b e^{-2(b-v)}\dd v.
\]
For the contribution coming from the other anchor, use
\[
 \int_s^b\!
 \left[k(e^v(x+\rho z-y))-k(e^v(x-y))\right]\dd v,
 \qquad 0\le z\le1.
\]
If $d\ge2\rho$, its integrand vanishes for $v>s+\log2$, and the
remaining interval has bounded length. If $d<2\rho$, then
$b-s<\log2$. Thus the cross-root contribution is uniformly bounded, including at the support boundary. By symmetry, the corresponding contribution to the mean of $H_{b,y}$ satisfies the same bound. Standard Gaussian supremum exponential moments
on these fixed smooth windows are uniform after the tilt.
The $\F_s$-measurable contributions to both local increments are differences of $H_{s,x}$ evaluated at arguments in $[-4,4]$.
Their sup norms are therefore bounded by $2\|H_{s,x}\|_{[-4,4]}$.

Drop the two survival indicators to the one indicator $I_s^\beta(x)$.
After conditioning on $\F_s$ and integrating the tilted future,
the left side of \eqref{lev1:eq-pair-weight} is at most
\[
 C_{k,\gamma}\E\left[I_s^\beta(x)w_s(x)w_s(y)
                     e^{C_{k,\gamma}\|H_{s,x}\|_{[-4,4]}}\right].
\]
When $s=b$, the same inequality holds directly without a future tilt.
Since $|y-x|\le e^{-s}$,
\[
 w_s(y)=e^{s+\sqrt2\beta-\sqrt2R_s(x)}
       e^{\sqrt2H_{s,x}(e^s(y-x))}.
\]
Root only at $x$. The last expectation is bounded by
\[
 C_{k,\gamma}\beta e^{s+\sqrt2\beta}
 \E_{\mathbb Q_{s,x}^\beta}
 \left[\frac{e^{-\sqrt2R_s(x)}}{R_s(x)}
                 e^{C'_{k,\gamma}\|H_{s,x}\|_{[-4,4]}}\right].
\]
Apply Lemma~\ref{lev1:lem-fixed-window-shape} with
$q=3/(2\eta)\in(1,3/2)$. This proves
\eqref{lev1:eq-pair-weight} and, by
\eqref{lev1:eq-decoration-envelope},
\begin{equation}\label{lev1:eq-Y-pair}
 \E|Y_jY_k|
 \le C_{k,\gamma,\eta}C_0^2\beta e^{\sqrt2\beta}
             L^2e^{-2b}e^ss^{-\eta}.
\end{equation}
No sign assumption on the local functionals \(A_j\) or on the corresponding covariance is used.

\emph{The product of conditional means.}
Represent two expectations conditional on $\F_t$ using two independent
copies of the future after $t$, sharing exactly the same coarse field.
For $d\le3e^{-t}$, drop the copied survival indicators to
$I_t^\beta(x)$. Each copied future root exponential tilt has mean one. The self- and
cross-root estimates used to prove \eqref{lev1:eq-pair-weight}, together with Lemma~\ref{lev1:lem-fixed-window-shape}, give uniformly
bounded exponential moments for the future contributions to the two local increments. Their $\F_t$-measurable contributions are differences
of $H_{t,x}$ evaluated in the fixed window $[-4,4]$. Rooting at $x$ at time $t$ gives
\begin{equation}\label{lev1:eq-conditional-copy}
 \E\left[\E(|Y_j|\mid\F_t)\E(|Y_k|\mid\F_t)\right]
 \le C_{k,\gamma,\eta}C_0^2\beta e^{\sqrt2\beta}
             L^2e^{-2b}e^tt^{-\eta}.
\end{equation}
There is no cross-copy future covariance. This argument includes
$j=k$ and does not require the copied future paths to survive.

\emph{Covariance support and the discrete sum.}
Each $Y_j$ depends on the future field only on its anchor path and
its $\rho$-window. Conditional on $\F_t$, disjoint such future
restrictions are independent whenever the anchor separation exceeds
$e^{-t}+2\rho\le3e^{-t}$. Thus the conditional covariance vanishes
outside this deterministic neighborhood.
For the remaining pairs,
\begin{align*}
 \E\left|\E[Y_j\overline{Y_k}\mid\F_t]
       -\E[Y_j\mid\F_t]\E[\overline{Y_k}\mid\F_t]\right|
 &\le\E|Y_jY_k|\\
 &\quad+\E\left[\E(|Y_j|\mid\F_t)
                         \E(|Y_k|\mid\F_t)\right].
\end{align*}
Equations \eqref{lev1:eq-Y-pair} and
\eqref{lev1:eq-conditional-copy} bound the two terms separately.

There are $O(e^b)$ anchors and spacing comparable to $e^{-b}$.
Pairs with $d\le\rho$ contribute at most a constant times
$L^2b^{-\eta}$. For $\rho<d\le3e^{-t}$, the regular-grid sum
is bounded, by grouping distances into mesh intervals or geometric
shells, by
\[
 C L^2\int_\rho^{3e^{-t}}
           \frac{\dd d}{d(\log(1/d))^{\eta}}
 \le C L^2\int_{t-\log3}^{b}s^{-\eta}\dd s
 \le C L^2(b-t+1)(t-\log3)^{-\eta}.
\]
The number of near pairs is
$O(e^{2b}e^{-t}+e^b)$, so summing
\eqref{lev1:eq-conditional-copy} contributes at most $CL^2t^{-\eta}$.
Multiplying by $A_0^2$ and restoring the common factor
$C_{k,\gamma,\eta}C_0^2\beta e^{\sqrt2\beta}$ proves
\eqref{lev1:eq-late-variance}.
The argument also proves the needed finite second moments, justifying
the conditional variance expansion for the finite row.
At the endpoint scales, the bound is
$O_{\beta,k,\gamma,\eta}(u^{1-\eta}\log u)$, which tends to zero.
\end{proof}

\section{Proofs of the signed tail estimates}
\label{l3:sec-radial-smoothing-audit}

We use the notation $b_H,q_H,Q_H$ from
Section~\ref{l3c:sec-uniform-center}. The Gaussian smoothing estimate
is given in Lemma~\ref{l3:lem-moving-shape-smoothing}. Here we combine it with localization and the local
fluctuation estimate.

\subsection*{Proof of Theorem~\ref{l3:thm-signed-tail-audit}}

\begin{proof}
By \eqref{l3:eq-smoothing-inputs}, the local functional
$A_j=Q_{H_j}(\log(v/c_j))$ satisfies
\eqref{lev1:eq-decoration-envelope}. Therefore
Theorem~\ref{lev1:thm-late-decoration}, followed by
Cauchy--Schwarz, gives
\[
 \mathbb E|T_v-\mathbb E(T_v\mid\mathcal F_t)|
 \le C_{\gamma,k,\eta}
       L(\tau+1)^{1/2}(t-\log3)^{-\eta/2}.
\]
We bound the conditional means without removing the early-scale spatial increment.
At an anchor write $r=R_t^\gamma(y)$, $c_b=L\rho w_b(y)$, and
$H_{\rm old}(v)=X_t(y+\rho v)-X_t(y)$. Tilt the future field by its
exponential martingale at $y$. Its root increment becomes
$\sqrt2s+B_s$, where $B$ is standard Brownian motion. The future
survival event is $\{\max_{s\le\tau}B_s<r\}$ and the terminal
threshold is exactly
\[
 \log(v/c_b)=\sqrt2(r-h_v)-\sqrt2B_\tau,
 \qquad h_v=\gamma+\kappa\log(L/v).
\]
Under this tilt, the full local increment admits the Gaussian regression
\[
 H=H_{\rm old}+m_\tau+a_\tau B_\tau+G_\tau,
\]
where $G_\tau$ is independent of $B_\tau$ and
\[
 A_\tau(v)=\int_0^\tau[k(e^{s-\tau}v)-1]ds,
 \qquad m_\tau=\sqrt2 A_\tau,\qquad a_\tau=A_\tau/\tau.
\]
Evenness and smoothness of $k$ at the origin give
\[
 \|m_\tau\|_\infty\le C_k,
 \qquad
 \|a_\tau\|_\infty\le C_k/\tau.
\]
Moreover, $G_\tau$ is independent of $B_\tau$, and the derivative-variance estimate and Gaussian supremum argument used in the proof of Lemma~\ref{crm:lem-shape} give uniformly bounded exponential moments of $\|G_\tau\|_\infty$. Hence
Lemma~\ref{l3:lem-moving-shape-smoothing} applies conditionally with
$H_0=H_{\rm old}$.

For $r\ge R$, removal of future killing has error at most
\[
 C e^{C\|H_{\rm old}\|}
    \mathbb P(\max_{s\le\tau}B_s\ge r)^{1/2}
 \le C e^{C\|H_{\rm old}\|}u^{-4},
\]
after increasing $K$. This uses Cauchy--Schwarz and the uniform exponential moments of the local increment under the tilted law; no independence between the
local increment and the Brownian maximum is used.
Similarly $Q_H(s)=q_H(s)-\kappa G_H\mathbf1_{\{s\le0\}}$ differs
from $q_H(s)$ in expectation by at most the same bound, because
$s\le0$ entails $B_\tau\ge r-h_v\ge K\sqrt{\tau\log u}$.
The remaining conditional mean is bounded by
$Ce^{C\|H_{\rm old}\|}\tau^{-1/2}$.

For $r<R$ use only the envelope. To sum this sector, root at the anchor
at time $t$:
\[
 \mathbb E[I_t^\gamma w_t e^{C\|H_{\rm old}\|}
                     \mathbf1_{\{R_t<R\}}]
 =\gamma\mathbb E_{\mathbb Q_{t,y}^\gamma}
       [R_t^{-1}e^{C\|H_{\rm old}\|}\mathbf1_{\{R_t<R\}}]
 \le C\gamma R^{1/2}t^{-3/4}.
\]
Indeed the rooted early-scale increment \(H_{\rm old}\), on the smaller window $\rho\le e^{-t}$,
has uniform exponential moments, and the BES$(3)$ density bound gives
$\mathbb E_{\mathbb Q}[R_t^{-2}\mathbf1_{\{R_t<R\}}]
\le CRt^{-3/2}$. Without the indicator the inverse-root bound gives
$\mathbb E[I_t^\gamma w_t e^{C\|H_{\rm old}\|}]
\le C\gamma t^{-1/2}$.
The conditional root tilt has prefactor $L\rho I_t^\gamma w_t$;
there are $O(\rho^{-1})$ anchors. Summation yields the first three
terms in \eqref{l3:eq-signed-tail-rate}. All estimates are uniform in
$v\in[\delta,A]$ and in the deterministic bounded weights.
\end{proof}

\subsection*{Proof of Corollary~\ref{l3:cor-common-early-tail-audit}}

\begin{proof}
The conditional-mean estimate in the preceding proof bounds each anchor
in absolute value and uses only $h_v\le h_\delta$.
It therefore holds for every $v\ge\delta$ and yields the corollary
without the late fluctuation term.
\end{proof}

\subsection*{Proof of Corollary~\ref{l3:cor-actual-tail-audit}}

\begin{proof}
On $E_{\beta,b}$ all the local indicators $I_b^\gamma(y_j)$ equal
one. The full fine variable is $c_jD_{H_j}$. It agrees with $U_j$
when the fine field does not cross the spatial barrier
$R_b^\beta(y_j)-H_j(v)$. Lemma~\ref{crm:lem-crossing} bounds the corresponding conditional crossing probability by
\[
 C\exp\{-\sqrt2R_b^\beta(y_j)+\sqrt2\sup H_j\}
 =\frac{C e^{-\sqrt2\beta}}L
                   c_j e^{\sqrt2\sup H_j}.
\]
This remains a valid upper bound if it exceeds one. By \eqref{paper:eq-budget}, the localized expectation of
$\sum_j c_j e^{\sqrt2\sup H_j}$ is uniformly bounded.
Multiplication by $v\le A$ gives the error $C_\beta A/L$.
The remaining difference is $T_v$, bounded by Theorem~\ref{l3:thm-signed-tail-audit}. No global
localization event was inserted inside a conditional expectation.
\end{proof}

\section{Adjacent cells and separated blocks}
\label{crm:sec-anticlustering}
\begin{lemma}
\label{crm:lem-adjacent-mixed-moment}
Let $D$ be the fixed real-line fine derivative chaos. For every $0<h<1$,
\[
 \mathbb E\bigl[D([-1,0])^hD([0,1])^h\bigr]<\infty.
\]
The same conclusion holds for any two fixed bounded intervals separated
by a point, including intervals with a positive gap.
\end{lemma}

\begin{proof}
Wong's Lemma~3.2 in \cite{Wong} applies to sets separated by a hyperplane
and gives the exponent range
$0\le h<1/2+1/(2\sqrt d)$. For $d=1$ this is exactly $0\le h<1$.
Use the fixed weight $1$ and intervals on a bounded open neighborhood of
$[-2,2]$. The fine covariance has the form
$-\log|x-y|+f(x,y)$ with smooth regular part on this neighborhood,
so the covariance decomposition hypothesis is the one already checked
in the proof of Lemma~\ref{crm:lem-finetail}. Wong uses the Seneta--Heyde normalization;
the deterministic conversion to the present derivative normalization
preserves finiteness. Deterministic endpoints have zero critical mass,
so the shared endpoint can be assigned to either interval.
\end{proof}

Applying Lemma~\ref{crm:lem-shape} on the fixed rescaled window
$[-A,A]$, for every fixed $A<\infty$ and $\lambda>0$,
\begin{equation}
\label{crm:eq-fixed-window-shape}
 \sup_{b\ge b_A,y,\beta>0}
 \mathbb E_{\mathbb Q_{b,y}^{\beta}}
 \exp\!\left\{\lambda\sup_{|v|\le A}
 |X_b(y+e^{-b}v)-X_b(y)|\right\}<\infty,
\end{equation}
provided $b$ is large enough that the window has a single local lift;
the finitely many remaining times are absorbed into the constant.

\begin{lemma}
\label{crm:lem-higher-power-budget}
Fix $1<q<3/2$, a fixed window size $A$, and deterministic anchors
$y_j$ with at most $C e^{b}$ points. Put
\[
 c_j=L\,e^{\sqrt2X_{b}(y_j)-2b},\qquad
 M_j^{(A)}=
 \exp\!\left\{\sqrt2\sup_{|v|\le A}
 [X_{b}(y_j+e^{-b}v)-X_{b}(y_j)]\right\}.
\]
Then
\begin{equation}
\label{crm:eq-higher-power-budget}
 \mathbb E\!\left[\mathbf1_{E_{\beta,b}}
 \sum_j(c_jM_j^{(A)})^q\right]
 \le C_{\beta,k,A,q}u^{q/2}b^{-3/4}
 \longrightarrow0.
\end{equation}
\end{lemma}

\begin{proof}
Write $r=R_{b}(y_j)$ and
$a_j=e^{\sqrt2\beta-\sqrt2r}=e^{-b}w_{b}(y_j)$.
The point-root identity bounds the expected contribution of one anchor
by
\[
 u^{q/2}e^{-b}\beta e^{\sqrt2\beta(q-1)}
 \mathbb E_{\mathbb Q_{b,y_j}^{\beta}}
 \left[\frac{e^{-\sqrt2(q-1)r}(M_j^{(A)})^q}{r}\right].
\]
Cauchy--Schwarz and \eqref{crm:eq-fixed-window-shape} reduce this to
$C u^{q/2}e^{-b}$ times
\[
 \left(\mathbb E_{\mathbb Q}
 \frac{e^{-2\sqrt2(q-1)R_{b}}}{R_{b}^2}\right)^{1/2}.
\]
The killed Brownian reflection formula implies
\[
 p_b^{(3)}(\beta,r)\le
 \frac{2r^2}{\sqrt{2\pi}\,b^{3/2}},\qquad r>0.
\]
The last display is consequently at most
$C_q b^{-3/4}$. Summing over $O(e^{b})$ anchors proves the estimate.
Its limit is zero because $b\asymp u$ and $q/2<3/4$.
\end{proof}

\subsection*{Proof of Proposition~\ref{crm:prop-local-anticlustering}}

\begin{proof}
Choose $h\in(1/2,3/4)$ and put $q=2h$. For a deterministic anchor
$y$ and a fixed $A\ge1$, write
\[
 c_b(y)=L e^{\sqrt2X_b(y)-2b},\qquad
 M_{b,A}(y)=
 \exp\!\left\{\sqrt2\sup_{|v|\le A}
 \bigl[X_b(y+e^{-b}v)-X_b(y)\bigr]\right\}.
\]

For an adjacent pair, relabel the cells cyclically as
$J_j,J_{j+1}$ and let $y^*=y_{j+1}$ be their shared boundary.
In fine coordinates the cells are $[-\theta,0]$ and
$[0,\theta]$, where $\theta=e^bd\in(1/2,1]$.
Positive terminal comparison and the common supremum of the coarse
increment give
\[
 U_j\le c_b(y^*)M_{b,1}(y^*)D([-\theta,0]),\qquad
 U_{j+1}\le c_b(y^*)M_{b,1}(y^*)D([0,\theta])
\]
in their joint conditional fine-field coupling. The fixed intervals
$[-1,0]$ and $[0,1]$ dominate the two masses. Mixed-moment Markov
inequality and Lemma~\ref{crm:lem-adjacent-mixed-moment} therefore
bound the joint conditional exceedance probability by
\[
 C_h\varepsilon^{-2h}
 \bigl(c_b(y^*)M_{b,1}(y^*)\bigr)^{2h}.
\]

For a pair at index distance two, relabel the cells cyclically as
$J_j,J_{j+2}$ and take the midpoint of the intervening cell,
\[
 y^{**}=y_{j+1}+\frac d2
 \qquad\text{modulo }2\pi.
\]
Their fine intervals are contained respectively in
$[-3/2,-1/4]$ and $[1/4,3/2]$. The same argument, now using the
coarse spatial increment on $[-2,2]$, bounds the corresponding
joint conditional exceedance probability by
\[
 C_h\varepsilon^{-2h}
 \bigl(c_b(y^{**})M_{b,2}(y^{**})\bigr)^{2h}.
\]
The adjacent-pair anchors $y^*$ and the distance-two anchors $y^{**}$
form deterministic families with $O(e^b)$ points. Hence
Lemma~\ref{crm:lem-higher-power-budget}, applied with exponent
$q=2h$ and window sizes $A=1$ and $A=2$, proves
\eqref{crm:eq-anticlustering}. Circular seam pairs admit the same local
lift for all sufficiently large $n$.

By the cell scaling \eqref{paper:eq-cell-scaling} and
Lemma~\ref{crm:lem-finetail}, 
$p_j(\varepsilon)\le C\varepsilon^{-1}c_jM_j$.
Thus the summed $q$th powers have localized expectations tending to zero
by Lemma~\ref{crm:lem-higher-power-budget}. This proves infinitesimality.
Since $0\le p_j\le1$ and $q<2$, the bounded graph degree and
$2p_jp_k\le p_j^2+p_k^2\le p_j^q+p_k^q$ also give the asserted
edge-product limit.
\end{proof}
\subsection{Protected-block geometry and square means}
\label{crm:sec-blocks}
Use the full circular mesh and all-integer scales of
\eqref{crm:eq-scales}. Put
\[
 K=\lceil u^2\rceil,\qquad p=\left\lfloor\frac{m}{K+2}\right\rfloor,
 \qquad
 B_q=\bigcup_{j=q(K+2)}^{q(K+2)+K-1}J_j,\quad 0\le q<p.
\]
After each block discard the following two mesh cells, and discard the
remaining fewer than $K+2$ cells at the end of the mesh. Write $G_n$ for
the union of these discarded cells. The block length is
$\Delta_n=Kd\asymp u^2e^{-b}$, and every circular gap between retained
blocks has length at least $2d>e^{-b}$, including the seam. In particular,
\begin{equation}\label{crm:eq-block-geometry}
 |G_n|\le C(u^{-2}+u^2e^{-b}),\qquad
 \sum_q|B_q|\le2\pi.
\end{equation}
By terminal locality, the restrictions $\mu|_{B_q}$ are mutually
independent given $\F_b$. This conditional independence is asserted under
the original probability measure, without conditioning on $E_\beta$.

Write
\[
 V_q=L\mu(B_q),\qquad
 m_q=\E[V_q\mid\F_b]=LZ_b^\beta(B_q).
\]
The gap error has the unconditional bound
\begin{equation}\label{crm:eq-block-gap}
 \E[L\mu(G_n)]=\beta L|G_n|=O_\beta(u^{-3/2}).
\end{equation}
The exponentially small term in \eqref{crm:eq-block-geometry} is absorbed
in the last bound. Freezing a Fourier phase and a $C^1$ weight at the left
endpoint $x_q$ of each block incurs an error at most
\begin{equation}\label{crm:eq-block-phase}
 L\Delta_n(n\|f\|_\infty+\|f'\|_\infty)\mu(\T),
\end{equation}
which tends to zero in $L^1$, since $Ln\Delta_n=O(u^{-3/2})$.
One may instead retain the individual cell phases inside each block.

\begin{lemma}
\label{crm:lem-block-squaremeans}
For every fixed deterministic $\beta>0$,
\begin{equation}\label{crm:eq-block-squaremeans}
 \E\sum_{q<p}m_q^2
 \le C_{\beta,k}\frac{\log u}{\sqrt u}\longrightarrow0.
\end{equation}
This estimate is unconditional; it does not require a global barrier
localization event.
\end{lemma}
\begin{proof}
The stochastic coefficient of $Z_s^\beta$ is
\[
 V_s(x)=I_s^\beta(x)(\sqrt2R_s(x)-1)w_s(x).
\]
The double-tilt calculation in Lemma~\ref{crm:lem-bracket} gives
\begin{equation}\label{crm:eq-point-bracket-alltime}
 \E V_s(0)^2\le C_{\beta,k}e^s(1+s)^{-3/2},\qquad s\ge0.
\end{equation}
For $s\ge1$ this is the killed-Brownian density estimate already proved;
on $0\le s\le1$ direct finite-time Gaussian moments bound the left side.
Stationarity and Cauchy--Schwarz bound
$|\E[V_s(x)V_s(y)]|$ by $\E V_s(0)^2$.
Since $q_{e^s}\ge0$, $\|q_{e^s}\|_1=e^{-s}\|k\|_1$, and the retained
blocks have total length at most $2\pi$,
\[
 \sum_q\iint_{B_q\times B_q}q_{e^s}(x-y)\dd x\dd y
 \le C_k\min(\Delta_n,e^{-s}).
\]
The bound by $C\Delta_n$ uses $q_{e^s}\le1$ and
$\sum_q|B_q|^2\le2\pi\Delta_n$; the other bound uses its $L^1$ norm.
The geometry is circular, so no boundary term occurs at the seam.
Martingale isometry now gives
\[
 \E\sum_q Z_b^\beta(B_q)^2
 \le C\beta^2\Delta_n+
 C_{\beta,k}\int_0^b
 \min(\Delta_ne^s,1)(1+s)^{-3/2}\dd s.
\]
Put $r=\log(1/\Delta_n)=b-2\log u+O(1)\asymp u$.
On $[0,r]$, split the integral at $r/2$: the first part is exponentially
small, and the second is bounded by
$C(1+r)^{-3/2}\int_{r/2}^re^{s-r}\dd s$.
The part on $[r,b]$ is at most
$(b-r)(1+r)^{-3/2}=O((\log u)u^{-3/2})$.
Multiply by $L^2=u$ to obtain \eqref{crm:eq-block-squaremeans}.
Lemma~\ref{crm:lem-bracket} gives the required square integrability of the fixed-barrier martingales, so the martingale isometry above is
justified.
\end{proof}

\subsection*{Statements and Declarations}

\noindent\textbf{Funding.}
X. F. is partially supported by the National Science and Technology Council,
Taiwan, grant no.~114-2115-M-A49-003-MY3.

\smallskip
\noindent\textbf{Data availability.}
This article has no associated data.

\smallskip
\noindent\textbf{AI Statement.}
The authors used an artificial-intelligence tool for language editing,
\LaTeX\ formatting and limited assistance with local
mathematical reasoning; all mathematical content was written, checked, and approved
by the authors.

\end{document}